\documentclass[10pt]{amsart}
\usepackage{amssymb,amsmath,amsthm}
\usepackage[colorlinks]{hyperref}
\usepackage{esint}
\usepackage[margin=1in]{geometry}
\usepackage[latin1]{inputenc}
\usepackage{amsfonts}
\usepackage{enumitem}
\usepackage{color}
\usepackage{bm}
\usepackage [english]{babel}
\usepackage{graphicx,dsfont}
\newtheorem{theorem}{Theorem}[section]
\newtheorem{definition}{Definition}[section]
\newtheorem{proposition}{Proposition}[section]
\newtheorem{lemma}{Lemma}[section]
\newtheorem{corollary}{Corollary}[section]

\newtheorem{remark}{Remark}[section]
\theoremstyle{plain} 

\let\div\undefined
\DeclareMathOperator{\div}{div}

\catcode`\@=11
\makeatletter

\@addtoreset{equation}{section}

\begin{document}


\title[Irregular Evolution double-phase problem]
{Irregular double-phase evolution problem: existence and global regularity}

\author{Rakesh Arora}
\address{Department of Mathematical Sciences, Indian Institute of Technology (IIT-BHU), Varanasi 221005, India}
\email{rakesh.mat@iitbhu.ac.in}

\author{Sergey Shmarev}
\address{Department of Mathematics, University of Oviedo, c/Federico Garc\'{i}a Lorca 18, 33007, Oviedo, Spain}
\email{shmarev@uniovi.es}
\thanks{The first author acknowledges the support of the ANRF (erstwhile SERB) Research Grant SRG/2023/000308, India, and Seed grant IIT(BHU)/DMS/2023-24/493. The second author was partially by the Research Grant MCI-21-PID2020-116287GB-I00, Spain}


\begin{abstract}
We investigate the homogeneous Dirichlet problem for the irregular double-phase evolution equation,
\[
u_t-\operatorname{div} \left( a(z)|\nabla u|^{p(z)-2} \nabla u + b(z)|\nabla u|^{q(z)-2} \nabla u\right)=f(z),\quad z=(x,t)\in Q_T:=\Omega\times (0,T),
 \]
where $\Omega \subset \mathbb{R}^N$, $N \geq 2$ is a bounded domain, $T>0$, The non-differentiable coefficients $a(z)$, $b(z)$, the free term $f$, and the variable exponents $p$, $q$ are given functions. The coefficients $a$ and $b$ are nonnegative, bounded, satisfy the inequality
\[ a(z)+b(z)\geq \alpha \quad \text{in} \ Q_T, \quad \text{and} \quad |\nabla a|, |\nabla b|, a_t, b_t \in L^d(Q_T)
\]
for some constant $\alpha>0$, and with $d>2$ depending on $\sup p(z)$, $\sup q(z)$, $N$, and the regularity of initial data $u(x,0)$. The free term $f$ and initial data $u(x,0)$  satisfy
\[ f\in L^\sigma(Q_T) \ \text{with} \ \sigma>2 \quad \text{and} \quad |\nabla u(x,0)|\in L^{r}(\Omega) \ \text{with} \ r\geq \max \bigg\{2,\sup_{Q_T}p(z),\sup_{Q_T}q(z)\bigg\}.
\]
The variable exponents $p,q \in C^{0,1}(\overline{Q}_T)$ satisfy the balance condition
\[
\frac{2N}{N+2} < p(z), q(z)< +\infty \ \text{in} \ \overline Q_T \quad \text{and} \quad \max\limits_{\overline Q_T}|p(z)-q(z)|< \dfrac{2}{N+2}.
\]
Under the above assumptions, we establish the existence of a solution, which is obtained as the limit of classical solutions to a family of regularized problems and

\begin{itemize}
    \item preserves initial temporal integrability: $$|\nabla u(\cdot, t)| \in L^r(\Omega) \ \text{for a.e.} \ t \in (0,T),$$
    \item gains global higher integrability: $$|\nabla u|^{\min\{p(z), q(z)\} + s +r} \in L^1(Q_T) \ \text{for any} \ s \in \left(0, \frac{4}{N+2}\right),$$
    \item attains second-order regularity:
    \[
a(z) |\nabla u|^{\frac{p+r-2}{2}}+b(z) |\nabla u|^{\frac{q+r-2}{2}}\in L^2(0,T;W^{1,2}(\Omega)).
\]
\end{itemize}
\end{abstract}

\maketitle

\tableofcontents

\section{Introduction}
In this work, we consider the following irregular double-phase parabolic problem
\begin{equation}
\label{eq:main}
\begin{split}
& u_t-\operatorname{div}\mathcal{F}(z,\nabla u)\nabla u=f(z)\quad \text{in $Q_T=\Omega\times (0,T)$},
\\
& \text{$u=0$ on $\partial\Omega\times (0,T)$},\qquad \text{$u(x,0)=u_0(x)$ in $\Omega$},
\end{split}
\end{equation}
where $\Omega\subset \mathbb{R}^N$, $N\geq 2$, is a bounded domain, $T>0$ and the nonlinear flux has the form
\begin{equation}
\label{eq:flux}
\mathcal{F}(z,\xi)\nabla u=\left(a(z)|\xi|^{p(z)-2} + b(z)|\xi|^{q(z)-2}\right)\nabla u.
\end{equation}
The coefficients $a$, $b$ and the exponents $p$, $q$ are given functions of the variable  $z=(x,t)\in Q_T$. The exponents $p,q: Q_T \mapsto \mathbb{R}$  are assumed to satisfy the conditions
\begin{equation}
\label{assum1}
\begin{split}
&\frac{2N}{N+2} < p^- \leq p(z) \leq p^+,\quad
\frac{2N}{N+2} < q^- \leq q(z) \leq q^+ \quad  \text{in $\overline{Q}_T$}
\end{split}
\end{equation}
with positive constants $p^\pm$, $q^\pm$, and
\begin{equation}
\label{eq:Lip-p-q}
\begin{split}
\text{$p,q \in C^{0,1}(\overline{Q}_T)$ with the Lipschitz constant $L_{p,q}$}.
\end{split}
\end{equation}
We assume that the exponents satisfy the balance condition
\begin{equation}
\label{eq:structure-prelim-2}
\max\limits_{\overline Q_T}|p(z)-q(z)|< \dfrac{2}{N+2}.
\end{equation}
Introduce the functions
\begin{equation}
\label{eq:s}
\underline{s}(z)=\min\{p(z),q(z)\}, \qquad \overline{s}(z)=\max\{p(z),q(z)\},
\end{equation}
and make the agreement to use the upper indexes $^\pm$ for the maximum/minimum of a function in its domain. The free term $f$ and the initial datum $u_0$ satisfy the conditions
\begin{equation}
\label{eq:coeff}
\begin{split}
f\in L^{\sigma}(Q_T) & \quad \text{with} \quad \sigma>2 \quad \text{and}
\\
u_0\in W^{1,r}_0(\Omega) &\quad  \text{with} \quad
\begin{cases}
\max\{p^+,q^+,2\} \leq r < +\infty \quad & \text{if} \ \sigma\geq N+2,
\\
\max\{p^+, q^+, 2\}\leq r \leq \frac{N(\min\{p^-, q^-\}(\sigma-1) -\sigma + 2)}{N+2-\sigma} \quad & \text{if} \ \sigma\in (2,N+2).
\end{cases}
\end{split}
\end{equation}
The modulating coefficients $a,b : Q_T \to \mathbb{R}$ are nonnegative, bounded, and such that
\begin{equation}
\label{eq:mod-coeff}
\begin{split}
& \text{there exists a constant $\alpha>0$ such that $a(z)+b(z)\geq \alpha>0$ in $\overline{Q}_T$},
\\
&
|\nabla a|,|\nabla b|,a_t, b_t\in L^d(Q_T)\quad
\text{with $d>2+\dfrac{N+2}{2}\left(\overline{s}^++r\right)$}.
\end{split}
\end{equation}

Equation \eqref{eq:main} arises within the framework of ``double-phase problems", a class of nonlinear PDEs in which the growth of the energy flux $\mathcal{F}$ depends on two competing spatial-temporal power law behaviors. This heterogeneous structure is modulated by the coefficients $a(z)$ and $b(z)$, where $z = (x,t)\in Q_T$. The term double-phase reflects the phenomenon of the transition between distinct growth regimes: in regions where both $a(z)$ and $b(z)$ are positive, the flux exhibits a composite growth influenced by both exponents $p(z)$ and $q(z)$, whereas in sub regions where either $a(z) = 0$ or $b(z) = 0$, the behavior degenerates to a single-phase expression governed by only one of the exponents. The interplay between the spatial variability of the coefficients and the nonstandard growth conditions gives rise to intricate mathematical challenges related to existence theory, regularity theory, compactness methods, and variational convergence.

In recent years, the double-phase setting in its stationary and evolutionary forms has been the focus of extensive analytical study. We refer the reader to the comprehensive contributions in \cite{Mingione-Radulescu-2021, Arora-Blanco-Winkert-2025, Ar-Sh-RACSAM-2023, Chl-Gw-Gw-Wr-2021} for an in-depth overview of the current theory and recent developments concerning both stationary and evolution double-phase problems.


In this work, we investigate two analytical questions concerning the irregular double-phase evolution problem \eqref{eq:main}, which features nonsmooth modulating coefficients and a free term.
\begin{itemize}
    \item \textit{Can a weak solution to the evolution double-phase problem \eqref{eq:main} be obtained when the modulating coefficients $a(z)$, $b(z)$, and the free term $f(z)$ lack smoothness?}
    \item \textit{How does the regularity (or lack thereof) of the initial condition $u_0(x)$, modulating coefficients $a(z)$, $b(z)$ and the free term $f(z)$ influence the integrability and second-order regularity properties of the constructed solution?}
\end{itemize}

The first question is central to developing a robust existence theory for double-phase problem with variable growth. The key idea is to consider a sequence of regularized non-degenerate parabolic problems with smoothed data, for which known parabolic theory ensures the existence of classical solutions.

The second question addresses the qualitative dependence of the solution on the data and seeks to understand to what extent initial spatial integrability propagate through the evolution governed by the double-phase operator. In particular, we aim to determine whether finer summability of the free term and improved integrability of the initial gradient, yield better regularity for the solution, in the sense of integrability of gradient and in a second-order regularity estimates.


The problem of improved integrability and second-order regularity of the flux function associated with nonlinear parabolic equations involving the $p(z)$-Laplace or double-phase type problem \eqref{eq:main} has been studied recently under various structural assumptions on the coefficients and exponents. Below, we summarize several notable contributions that fall into specific subclasses of the general framework addressed in this paper.

\textbf{(a) The evolution $p(z)$-Laplace equation $a(z)=b(z)$, $p(z)=q(z)$.} Under these assumptions
\[
\mathcal{F}(z,\nabla u)\nabla u=2a(z)|\nabla u|^{p(z)-2}\nabla u,
\]
and equation \eqref{eq:main} renders the evolution $p(z)$-Laplace equation. The weak solutions of such equations are known to exhibit \textit{improved integrability} of their gradients, going beyond the natural integrability $|\nabla u|^{p(z)} \in L^1(Q_T)$ dictated by the equation. More precisely, \[
\int_{S}|\nabla u|^{p(z)+\delta}\,dz\leq C
\]
in the arbitrary strictly interior cylinder $S\Subset Q_T$, with a constant $\delta>0$, which depends on the distance from $S$ to the parabolic boundary of $Q_T$. For variable exponent spaces, this phenomenon was first established in~\cite{Ant-Zhikov-2005} under the assumption $p(z)\geq 2$, and extended in~\cite{Zhikov-Past-2010} to the range $\frac{2N}{N+2} < p(z) \leq p^+ < \infty$ for exponents with logarithmic modulus of continuity. A global form of this estimate was obtained in~\cite{Ar-Sh-2021} for Lipschitz continuous exponents and domains with $\partial\Omega \in C^2$:
\begin{equation}
\label{eq:high-int-intr}
\int_{Q_T} |\nabla u|^{p(z)+\delta}\, dz \leq C \quad \text{for any } \delta \in \left(0, \frac{4}{N+2}\right),
\end{equation}
where the constant $C$ depends on $N$, $\max p$, $\min p$, $\delta$, and $\|u\|_{L^{\infty}(0,T;L^2(\Omega))}$. For related results involving Orlicz growth, we refer to~\cite{Hasto-OK-2021}, which addresses local higher integrability for systems of parabolic type.

The results on the second-order regularity of solutions to equations and systems that involve the $p$-Laplace operator are commonly formulated in terms of the inclusions
\[
|\nabla u|^{\lambda} \nabla u \in L^2(0,T;W^{1,2}(\Omega))^N,
\]
with $\lambda$ depending on $p$ and $N$, as demonstrated in~\cite{Seregin-Acerbi-Mignione-2004, Duzaar-Mignione-Steffen-2011, Feng-Parviainen-Sarsa-2023, DeFilippis-2020}, while global results are discussed in~\cite{Berselli-Ruzicka-2022}, see also references therein. The global second-order regularity of weak solutions to the Dirichlet problem for $p(z)$-Laplace evolution equation with different choices of the source terms and regularity of initial data was studied in \cite{Cianchi-Maz'ya-2019-1,Ar-Sh-2021,Ar-Sh-2024}.

In~\cite{Cianchi-Maz'ya-2019-1}, optimal second-order regularity is shown for approximable solutions of the homogeneous Dirichlet problem associated with the constant exponent $p$-Laplace evolution equation, proving that
\[
u_t \in L^2(Q_T), \quad |\nabla u|^{p-2} \nabla u \in L^2(0,T;W^{1,2}(\Omega))^N.
\]
The variable exponent analogue is addressed in~\cite{Ar-Sh-2021}, where Galerkin approximations are employed for $f\in L^2(0,T;W^{1,2}_0(\Omega))$, and $u_0\in W^{1,q(\cdot)}_0(\Omega)$ with $q(x)=\max\{2,p(x,0)\}$, yielding
\[
u_t \in L^2(Q_T), \quad D_{x_i}\left(|\nabla u|^{\frac{p(z)-2}{2}} D_{x_j} u\right) \in L^2(Q_T).
\]
A more flexible regularization framework was developed in~\cite{Ar-Sh-2024}, involving smooth approximations of both the data and the nonlinear terms. Classical solutions to these regularized equations are shown to converge to a weak solution of the original problem. Under the assumptions $p(z) > \frac{2(N+1)}{N+2}$, $u_0 \in W^{1,p(\cdot,0)}_0(\Omega)$, and $f \in L^2(Q_T)$, the limiting solution satisfies
\begin{equation}
\label{eq:approx-L-2}
\begin{split}
& |\nabla u|^{2(p(z)-1)+\delta} \in L^1(Q_T), \quad \text{for any } \delta \in \left(0,\frac{4}{N+2}\right), \\
& |\nabla u|^{p(z)-2} \nabla u \in L^2(0,T;W^{1,2}(\Omega)).
\end{split}
\end{equation}
These properties are derived from uniform estimates for the approximating classical solutions and the pointwise convergence of their gradients. Thus, the results of~\cite{Cianchi-Maz'ya-2019-1} are naturally extended to the variable exponent framework.

Very recently, in \cite{Ar-Sh-2025}, it is shown that if $a(z)\equiv const$ and the data satisfy conditions \eqref{eq:coeff} with $p\equiv q$, then the gradient of the solution to problem \eqref{eq:main} with the nonlinear source $f(z)+F(z,u,\nabla u)$ maintains the initial order of integrability $r\geq \max\{2,p^+\}$, achieves higher integrability, and the solution acquires the second-order regularity:
\begin{equation}\label{eq:prelim-res}
\begin{split}
& \|\nabla u\|_{L^{\infty}(0,T;L^r(\Omega))}\leq C,\quad \int_{Q_T}|\nabla u|^{p(z)+r+\rho-2}\,dz\leq C,
\\
& |\nabla u|^{\frac{p(z)+r}{2}-2}\nabla u\in L^2(0,T;W^{1,2}(\Omega))^N,
\end{split}
\end{equation}
with any $\rho\in \left(0,\frac{4}{N+2}\right)$ and constants $C$, $C'$ depending on the data, $r$ and $\rho$. For equations with the general flux $a(x,t,\nabla u)$ of constant $(p,q)$-growth and $f\equiv 0$, the local boundedness of the gradient and the second-order regularity of the solution are established in \cite{DeFilippis-2020} under suitable conditions on the gap between the exponents $p,q$ and assumptions on the regularity of $a(x,t,\nabla u)$. Systems of parabolic equations

          \[
          u^{i}_t - \sum_{k=1}^N\left(\mathbf{a}(|\nabla u|)u^{i}_{x_k}\right)_{x_k}=f^{i}\in L^{\sigma}(Q_T),\qquad \sigma>N+2,
          \]
with $\mathbf{a}(s)+s\mathbf{a}'(s)\approx (\mu^2+s^2)^{\frac{p-2}{2}}$, $\mu\in [0,1]$,  $p=const>\frac{2N}{N+2}$, are considered in \cite{Bogelein-2021}. It is shown that the gradient is bounded up to the lateral boundary of the cylinder $Q_T$ on every time interval $(\epsilon,T)$ with $\epsilon>0$. The optimal conditions for local boundedness of the gradient of solutions to systems of $p$-Laplace evolution equations are found in \cite{Kuusi-Mignione-2012} in the framework of borderline rearrangement invariant function spaces
of Lorentz type under the assumption $f\in L(N + 2,1)$, see also \cite{Kuusi-Mingione-2012-NA,Kuusi-Mingione-2013-1}.

Higher integrability of the gradient near to the initial plane is proven in \cite{Byun-Kim-Lim-2020} for systems of equations with the $p$-Laplace structure and the right-hand side $\operatorname{div}\left(|F|^{p-2}F\right)+f$. It is shown that higher integrability of $|\nabla u(x,0)|$, $F$ and $f$ yields higher integrability of $|\nabla u|$ in $Q_T$. The order of integrability is improved within an interval defined by the data.

\textbf{(b) Double phase structure with H\"older-continuous coefficients}: $a(z)=1$, $b(z)\geq 0$ in $Q_T$, $1<p(z)<q(z)$ This choice of the data renders equation \eqref{eq:main} the double-phase equation with the flux

    \[
    \mathcal{F}(z,\nabla u)\nabla u=\left(|\nabla u|^{p(z)-2}+b(z)|\nabla u|^{q(z)-2}\right)\nabla u.
    \]
The property of improved integrability is often formulated in terms of the flux. It is proven in \cite{KKM-2023} that for the solutions of the equation

\[
u_t-\operatorname{div} \mathcal{F}(z,\nabla u)\nabla u= -\operatorname{div} \mathcal{F}(z,F)F
\]
with $b\in C^{\alpha,\frac{\alpha}{2}}(Q_T)$, $\alpha\in (0,1)$,  constant exponents $p,q$ subject to the conditions $2\leq p<q\leq p+\frac{2\alpha}{N+2}$, and a given vector $F$, the flux possesses the higher integrability property: the inclusions $\mathcal{F}(z,\nabla u)|\nabla u|^2\in L^1_{loc}(Q_T)$, $\mathcal{F}(z,F)|F|^2\in L^{1+\epsilon}_{loc}(Q_T)$ imply
$\mathcal{F}(z,\nabla u)|\nabla u|^2\in L^{1+\epsilon}_{loc}(Q_T)$ with some $\epsilon>0$ depending on the data. For the counterpart singular equations with $\frac{2N}{N+2}<p\leq 2$, higher integrability of the flux is established in \cite{Wontae-Kim-NoDeA-2024} under the same relations between $p$, $q$ and $\alpha$.

\textbf{(c) Perturbed Double-phase structure with H\"older-continuous coefficients}: $0< \nu\leq a(z)<\infty$ and $b(z)=a(z)c(z)>0$ The corresponding flux function has the form
\[
  \mathcal{F}(z,\nabla u)\nabla u=a(z)\mathcal{G}(z,\nabla u)\nabla u,\qquad \mathcal{G}(z,\nabla u)=|\nabla u|^{p(z)-2}+c(z)|\nabla u|^{q(z)-2}.
\]
This class of double-phase equations is studied in \cite{Wontae-Kim-JMAA-2025} for the constant exponents $p\leq 2$ and $p<q\leq p+\frac{\alpha(p(N+2)-2N))}{2(N+2)}$, being $\alpha\in (0,1)$ the H\"older exponent of $c(z)$. The exponent $c(z)$ satisfies the vanishing mean oscillation property. It is shown that for the solutions of the equation
\[
u_t- \operatorname{div}\left(a(z)\mathcal{G}(z,\nabla u)\nabla u\right)=\operatorname{div}\mathcal{G}(z,\nabla F)\nabla F
\]
with $\inf_{Q_T}c(z)>0$ and a given vector $F$ the inclusions $\mathcal{G}(z,\nabla u)|\nabla u|^2\in L^1_{loc}(Q_T)$, $\mathcal{G}(z,F)|F|^2\in L^{1+\sigma}_{loc}(Q_T)$ imply $\mathcal{G}(z,F)|F|^2\in L^{1+\sigma}_{loc}(Q_T)$, where $\sigma>0$ belongs to an interval defined by the data.

\textbf{(d) Double-phase structure with Lipschitz-continuous coefficients and unordered exponents}: Problem \eqref{eq:main} with Lipschitz-continuous in $Q_T$ coefficients $a,b$ satisfying the inequality $a(z)+b(z)\geq \alpha>0$, and variable exponents $p,q$, is studied in \cite{RACSAM-2023}. It is shown that for $f\in L^2(0,T;W^{1,2}_0(\Omega))$ the solutions of problem \eqref{eq:main} possess the property of global higher integrability of the gradient and the second-order regularity:
\[
\begin{split}
    & \text{$|\nabla u|^{\underline{s}(z)+\rho}\in L^{1}(Q_T)$ with any $\rho\in \left(0,\frac{4}{N+2}\right)$},
    \\
    & D_{x_i}\left(\sqrt{\mathcal{F}(z,\nabla u)}D_{x_j}u\right)\in L^{2}(Q_T), \quad i,j=\overline{1,N}.
\end{split}
\]

To the best of our knowledge, no results on the existence and regularity of weak solutions to evolution double-phase problems are available when the modulating coefficients $a(z)$, $b(z)$, and the free term $f(z)$ lack regularity.
The novelty of the present work lies in extending the global regularity results previously established for the evolution \(p(z)\)-Laplace equation to the general double-phase setting with irregular coefficients \(a(z)\), \(b(z)\). Assuming only the structural conditions~\eqref{eq:coeff} and~\eqref{eq:mod-coeff}, we show that if $f\in L^{N+2}(Q_T)$, then the gradients of the solutions to problem \eqref{eq:main} maintain any initial order of integrability $r$, independently of the order of integrability of the free term. For $f\in L^{\sigma}(Q_T)$ with $\sigma\in (2,N+2)$, the order of integrability of the gradient is preserved within the limits depending on the data. The second-order regularity of the solution is proven in terms of the flux. We do not distinguish between the cases of the degenerate equation with $p^-\geq 2$, $q^-\geq 2$, and the singular one with $p^+,q^+<2$, and do not require any order relation between $p(z)$ and $q(z)$. The exponents $p(z)$, $q(z)$ may vary within the limits given in \eqref{assum1}, the restriction \eqref{eq:structure-prelim-2} means that their values are not too different. The proofs of the main results rely on the possibility of approximation of  solutions to problem \eqref{eq:main} by classical solutions of the regularized equations.

\section{Main results}
\subsection{The function spaces}
\label{subsec:spaces}
The solution to problem \eqref{eq:main} will be sought as an element of the variable Lebesgue and Sobolev spaces. While the generalized Musielak-Orlicz spaces provide a natural analytical framework for double-phase problems, the approach based on the regularization of the equation and the data allows one to reduce the study of the regularized problem with smooth data to dealing with elements of the parabolic H\"older spaces, and the variable Lebesgue and Sobolev spaces.

1) Variable Lebesgue spaces. Let $p:\Omega\mapsto [p^-,p^+]\subset (1,\infty)$ be a Lipschitz-continuous function. The space

\[
L^{p(\cdot)}(\Omega)=\left\{\text{$u:\Omega\mapsto \mathbb{R}$ is measurable}: \,\int_{\Omega}|u|^{p(x)}\,dx<\infty\right\}
\]
equipped with the norm

\[
\|u\|_{p(\cdot),\Omega}=\inf\left\{\lambda>0:\, \int_{\Omega}\left(\frac{|u|}{\lambda}\right)^{p(x)}\,dx\leq 1\right\}
\]
is a separable Banach space.

2) Variable Sobolev spaces. Given a Lipschitz-continuous function $\rho:\Omega\mapsto \mathbb{R}$, $\rho\in [\rho^-,\rho^+]\subset (1,\infty)$, we introduce the space

\[
\begin{split}
& \mathbb{V}_{\rho(\cdot)}(\Omega)=\left\{u:\Omega\mapsto \mathbb{R}\; | \; u\in W^{1,1}_0(\Omega),\,|\nabla u|^{\rho(x)} \in L^1(\Omega)\right\},
\\
& \|u\|_{\mathbb{V}_{\rho(\cdot)}}(\Omega)=\|\nabla u\|_{\rho(\cdot),\Omega}.
\end{split}
\]
If $\rho:Q_T\mapsto \mathbb{R}$, $\rho\in [\rho^-,\rho^+]\subset (1,\infty)$, we define

\[
\begin{split}
& \mathbb{W}_{\rho(\cdot,\cdot)}(Q_T)=\left\{u:(0,T)\mapsto \mathbb{V}_{\rho(\cdot,t)}(\Omega) \,|\,u\in L^2(Q_T), \; |\nabla u|^{\rho(x,t)}\in L^1(Q_T)\right\},
\\
& \|u\|_{\mathbb{W}_{\rho(\cdot,\cdot)(Q_T)}}=\|u\|_{2,Q_T}+\|\nabla u\|_{\rho(\cdot,\cdot),Q_T}.
\end{split}
\]
The spaces $\mathbb{V}_\rho(\cdot)(\Omega)$ and $\mathbb{W}_{\rho(\cdot,\cdot)}(Q_T)$ are Banach spaces.

3) By $C_{x,t}^{2+\beta,\frac{2+\beta}{2}}(Q_T)$, $\beta\in (0,1)$,  we denote the parabolic H\"older spaces, see \cite{LSU}.

\subsection{The main results}
\begin{definition}[Strong solution]
\label{def:strong-sol}
A function $u$ is called strong solution of problem \eqref{eq:main} if
\begin{enumerate}
\item $u\in C([0,T];L^2(\Omega))$, $u\in \mathbb{W}_{\overline{s}(\cdot)}(Q_T)$, $\mathcal{F}(z,\nabla u)|\nabla u|^2\in L^{\infty}(0,T;L^1(\Omega))$, $u_t\in L^2(Q_T)$;

\item for every $\phi\in \mathbb{W}_{\overline{s}(\cdot)}(Q_T)$

\begin{equation}
\label{eq:def-nondiv}
\int_{Q_T}\left(u_t\phi+ \mathcal{F}(z,\nabla u)\nabla u\cdot \nabla \phi-f\phi\right)\,dz=0;
\end{equation}

\item $\|u(\cdot,t)-u_0\|_{2,\Omega}\to 0$ as $t\to 0$.
\end{enumerate}
\end{definition}

\begin{theorem}
\label{th:main-1}
Assume that $\partial\Omega\in C^{2+\beta}$, $\beta\in (0,1)$, and conditions \eqref{assum1}, \eqref{eq:Lip-p-q}, \eqref{eq:structure-prelim-2}, \eqref{eq:coeff}, \eqref{eq:mod-coeff} are fulfilled. Then problem \eqref{eq:main} has a unique strong solution, which satisfies the estimates

\[
\|u_t\|^2_{2,Q_T}+ \int_{Q_T}|\nabla u|^{\underline{s}(z)+r+s}\,dz\leq C_1\qquad \text{with any $s\in \left(0,\frac{4}{N+2}\right)$},
\]

\begin{equation}
\label{eq:a-priori-2}
\begin{split}
\operatorname{ess}\sup_{(0,T)} \int_{\Omega}|\nabla u|^r\,dx  &
\leq C_2 + \int_{\Omega}|\nabla u_0|^r\,dx
\end{split}
\end{equation}
with positive constants $C_i$, depending on $r$, $N$, $\|\nabla u_0\|_{r,\Omega}$, $\|f\|_{\sigma,Q_T}$, and the properties of $p$, $q$, $a$, $b$; $C_1$ depends also on $s$.
\end{theorem}
\begin{remark}
    It is worth noting that the gradient of the solution attains global higher integrability, characterized by two different parameters $r$ and $s$, as a result of the improved integrability of the initial data $|\nabla u_0| \in L^r(\Omega)$ and due to the application of interpolation inequalities, respectively.
\end{remark}
\begin{theorem}
\label{th:main-2}
Under the conditions of Theorem \ref{th:main-1}, the solution $u$ of problem \eqref{eq:main} possesses the second-order regularity:
\[
\mathcal{G} \equiv a(z) |\nabla u|^{\frac{p+r-2}{2}}+b(z) |\nabla u|^{\frac{q+r-2}{2}}\in L^2(0,T;W^{1,2}(\Omega)),\qquad \|\mathcal{G}\|_{L^2(0,T;W^{1,2}(\Omega))}\leq C
\]
with a constant $C$ depending only on the data.
\end{theorem}
\begin{remark}
    In case of constant exponents $p, q=\text{const.}$, the second-order regularity result in Theorem \ref{th:main-2} signifies the propagation of integrability in the following sense: for any $r \geq \max\{p, q, 2\}$
    \[
    a(z) |\nabla u|^{p-1 + \frac{r-p}{2}}+b(z) |\nabla u|^{q-1 + \frac{r-q}{2}}\in L^2(0,T;W^{1,2}(\Omega)).
    \]
In case of $p=q=\text{const.}$ and as $r \to p$, the above second order regularity coincides with the optimal regularity results in \cite{Cianchi-Maz'ya-2019-1}.
\end{remark}
\begin{remark}
\label{rem:admiss}
Inequalities \eqref{eq:coeff} furnish the interval of the permissible values of $r$ in case of low integrability of the free term $f$.
\begin{itemize}
    \item[{\rm (i)}] When $\sigma\nearrow N+2$, the supremum of the admissible values $r$ tends to $+\infty$  independently on the properties of $p(z)$ and $q(z)$.
    \item[{\rm (ii)}]  For $\sigma \in (2,N+2)$, inequality \eqref{eq:coeff} entails the conditions on the exponents $p(z), q(z)$:
\[
\begin{split}
{\rm (a)} & \qquad \text{if $\overline{s}^+ \leq 2$, then $\displaystyle \underline{s}^-  >  \frac{\sigma}{\sigma-1}- \frac{2(\sigma-2)}{N(\sigma-1)}$},
\\
{\rm (b)} & \qquad \text{if $\overline{s}^+>2$, then $\displaystyle \overline{s}^+  < \frac{ N(\underline{s}^-(\sigma-1) -\sigma + 2)}{N +2 -\sigma}$}.
\end{split}
\]
Moreover, condition {\rm (b)} can be transformed into the restriction on the oscillation $\overline{s}^+-\underline{s}^-$ of $\underline{s}(z)$ and $\overline{s}(z)$ in $Q_T$.

\item[{\rm (iii)}] 
   The order $r$ of the gradient integrability is preserved in time if $r$ is sufficiently large, but when $\sigma \searrow 2$, the interval of the admissible values of $r$ becomes empty. These observations indicate the need to modify the method of the present work to address both issues. We leave these problems open for future study.
\end{itemize}
\end{remark}

\subsection{Organization of the paper:}
In Section \ref{sec:reg-smooth}, we study the regularized problem \eqref{eq:main} with strictly positive smooth coefficients, smooth exponents of nonlinearity, and smooth data. We prove that the problem has a classical solution in the parabolic H\"older space. The existence of a solution follows from the classical parabolic theory, provided that the global H\"older continuity of the gradient of solution is known a priori. To achieve this, we extend the solution across the lateral boundary and the bottom of the cylinder $ Q_T$ and then apply the known local result to the extended solution. For the sake of completeness, we offer a detailed description of the extension procedure. Section \ref{sec:interpolation} is devoted to refining the interpolation inequalities derived in \cite{RACSAM-2023} for the double-phase operators of the type \eqref{eq:flux} with Lipschitz-continuous coefficients $a$, $b$. We adapt the proofs to the case of irregular coefficients; these inequalities entail the property of global higher integrability of the gradient of solution. In Section \ref{sec:int-by-parts}, we derive estimates that involve the second-order derivatives, which follow after the multiplication of the regularized equation by the regularized $r$-Laplacian and application of Green's formula to the diffusion term. The estimates derived in Sections \ref{sec:interpolation}, \ref{sec:int-by-parts} are uniform concerning the parameters of regularization of the equation and the data. The interpolation inequalities and the estimates on the second-order derivatives are used in Section \ref{sec:est-source} to estimate the integral of the product of the free term $f$ and the regularized $r$-Laplacian. The main results are proven in Sections \ref{sec:nonsmooth-data}, \ref{sec:degenerate-problem}. The uniform estimates on the higher power of the gradients and the second-order derivatives allow one to extract a subsequence of solutions to the regularized problems, which converges to a strong solution of problem \eqref{eq:main} and preserves some regularity properties of smooth solutions. It is worth noting that a byproduct of the proof of the existence of strong
solutions to the regularized problem (3.1) and the main problem (1.1) is the stability of the solutions concerning the data perturbation. In Appendix \ref{sec:aux}, we collect several properties of the double-phase flux and the corresponding modular space.

By agreement, we denote

\[
D_iu=\dfrac{\partial u}{\partial x_i},\qquad  D^2_{ij}u=\dfrac{\partial^2 u}{\partial x_i\partial x_j},\qquad\vert u_{xx}\vert ^2=\sum_{i,j=1}^{N}\left( D^2_{ij}u\right)^2.
\]
The notation $C$ is used for the constants that can be calculated or estimated through the data but whose exact values are unimportant. The value of $C$ may change from line to line even inside the same formula.

\section{Regularized problem with smooth data. Existence and basic properties of solution}
\label{sec:reg-smooth}

\subsection{Existence of a unique strong solution}
Consider the problem

\begin{equation}
\label{eq:main-reg}
\begin{split}
& u_t-\operatorname{div}\mathcal{F}_\epsilon(z,\nabla u)\nabla u=f(z),
\\
& \text{$u=0$ on $\partial\Omega\times (0,T)$},\qquad \text{$u(x,0)=u_0(x)$ in $\Omega$},
\end{split}
\end{equation}
with the regularized flux

\begin{equation}
\label{eq:structure-prelim-1}
\begin{split}
&
\mathcal{F}_\epsilon(z,\xi)=a_\epsilon(z)(\epsilon^2+|\xi|^{2})^{\frac{p-2}{2}} + b_\epsilon(z)(\epsilon^2+|\xi|^{2})^{\frac{q-2}{2}},
\\
& a_\epsilon=\epsilon+a(z),\qquad b_\epsilon=\epsilon+b(z), \quad \epsilon\in (0,1).
\end{split}
\end{equation}
It is known that problem \eqref{eq:main-reg} has a unique strong solution.

\begin{proposition}[\cite{RACSAM-2023}, Theorems 7.1, 7.3]
\label{pro:exist-strong-sol}
Assume that $\partial\Omega\in C^{2}$, and $a,b,p,q$ are Lipschitz-continuous in $Q_T$. For every $f\in L^2(0,T;W^{1,2}_0(\Omega))$ and $u_0$ such that

\[
\int_\Omega\left(|\nabla u_0|^2+a(x,0)|\nabla u_0|^{p(x,0)}+b(x,0)|\nabla u_0|^{q(x,0)}\right)\,dx<\infty
\]
problem \eqref{eq:main} has a unique strong solution that satisfies the estimates

\begin{equation}
\label{eq:est-sol-Gal}
\begin{split}
& \|u\|_{\mathbb{W}_{\overline{s}(\cdot)},Q_T}\leq C_0,
\\
& \operatorname{ess}\sup_{(0,T)}\|u\|_{2,Q_T}^2+\|u_t\|_{2,Q_T}^2+ \operatorname{ess}\sup_{(0,T)}\|\nabla u\|_{2,Q_T}^2
+ \operatorname{ess}\sup_{(0,T)}\int_{\Omega}\mathcal{F}_{\epsilon}(z,\nabla u)|\nabla u|^2 \,dx\leq C_1,
\end{split}
\end{equation}

\begin{equation}
\label{eq:improved-integr-grad-1}
\int_{Q_T}|\nabla u|^{\underline{s}(z)+r}\,dz\leq C_2,\qquad r\in \left(0,\frac{4}{N+2}\right)
\end{equation}
with constants $C_i$ depending only on the data.
\end{proposition}

The strong solution of problem \eqref{eq:main-reg} is constructed in \cite{RACSAM-2023} as the limit of the sequence of Galerkin's approximations. We show first that if the data of problem \eqref{eq:main-reg} are sufficiently smooth, the strong solution is, in fact, a classical solution.

\subsection{Global regularity of the weak solution}
\subsubsection{Local H\"older continuity of the gradient}
Consider the auxiliary problem

\begin{equation}
\label{eq:nondiv}
\begin{split}
& u_t-\operatorname{div}\mathbf{a}(z,\nabla u)=f\quad \text{in $Q_T$},
\\
& \text{$u=0$ on $\partial\Omega\times (0,T)$},\qquad \text{$u(x,0)=u_0(x)$ in $\Omega$}.
\end{split}
\end{equation}
Assume that there exist a Lipschitz-continuous function $p(\cdot)$ and nonnegative constants $a_0$, $a_1$, $a_2$, $C$ such that the function $\mathbf{a}(z,\xi):Q_T\times \mathbb{R}^N\mapsto \mathbb{R}^N$ satisfies the following conditions.

\begin{equation}
\label{eq:structure}
\begin{split}
& \text{(i) \quad  Monotonicity:} \quad \forall \xi,\eta\in \mathbb{R}^N
\\
& \qquad \qquad
(\mathbf{a}(z,\xi)- \mathbf{a}(z,\eta))\cdot (\xi-\eta)
\geq C\begin{cases}
(|\xi|+|\eta|)^{p(z)-2}|\xi-\eta|^2 & \text{if \text{if $1<p(z)<2$}},
\\
|\xi-\eta|^{p(z)} & \text{if $p(z)\geq 2$}.
\end{cases}
\\
& \text{(ii) \quad Growth and boundedness: there exist positive constants $a_0,a_1,a_2$ such that }
\\
&
\qquad \qquad \begin{array}{l}
\mathbf{a}(z,\xi)\cdot \xi\geq a_0|\xi|^{p(z)}-a_1,
\qquad
|\mathbf{a}(z,\xi)|\leq a_2 |\xi|^{p(z)-1} + a_2.
\end{array}
\\
& \text{(iii) \quad Continuity:}
\\
&
\qquad \qquad \begin{array}{l} \text{there exist $\alpha\in (0,1)$ and $R>0$ such that for every cylinder}
\\
\text{$Q^{(R)}(x_0,t_0)=B_R(x_0)\times (t_0,t_0+R)\Subset Q_T$ and $z_1,z_2\in Q^{(R)}(x_0,t_0)$}
\\
|\mathbf{a}(z_1,\xi)-\mathbf{a}(z_2,\xi)|\leq Cd^{\alpha}(z_1,z_2)\left(1+\ln(1+|\xi|^2)\right)\left(1+|\xi|^{p^+_R-1}\right),
\end{array}
\end{split}
\end{equation}
where $p^+_R=\max_{Q^{(R)}(x_0,t_0)}p(z)$ and $d(\cdot,\cdot)$ denotes the parabolic distance such that
\[
d^\alpha(z_1, z_2)= |x_1-x_2|^\alpha + |t_1 - t_2|^\frac{\alpha}{2}, \quad \text{for} \ z_i=(x_i, t_i) \in Q_T,\ i=1,2 \ \text{and} \ \alpha \in (0,1].
\]
\begin{definition}[Weak solution]
\label{def:weak-sol} A function $u$ is called weak solution of problem \eqref{eq:nondiv} if
\begin{enumerate}
\item $u\in C([0,T];L^2(\Omega))$, $\nabla u\in L^{p(\cdot)}(Q_T)$,

\item for every $\phi\in W^{1,2}(0,T;L^2(\Omega))$, $\nabla \phi\in L^{p(\cdot)}(Q_T)$ such that $\phi=0$ on $\partial\Omega\times (0,T)$ and for $t=T$

\begin{equation}
\label{eq:def-nondiv-1}
\int_{0}^T\int_{\Omega}\left(-u\phi_t+ \mathbf{a}(z,\nabla u)\cdot \nabla \phi\right)\,dxdt=\int_{\Omega}u_0(x)\phi(x,0)\,dx+ \int_{0}^T\int_{\Omega}f\phi\,dxdt.
\end{equation}
\end{enumerate}
\end{definition}

\begin{proposition}[Proposition 4.1, \cite{DZZ-2020}]
\label{pro:interior}
Let $u$ be a bounded weak solution of problem \eqref{eq:nondiv}. Assume conditions \eqref{eq:structure} are fulfilled and $f\in L^{\infty}(Q_T)$. Then $\nabla u\in C^{\alpha,\frac{\alpha}{2}}_{loc}(Q_T)$ with some $\alpha\in (0,1)$.
\end{proposition}

\subsection{Properties of the regularized flux}
We check first that conditions \eqref{eq:structure} hold true for the regularized flux $\mathbf{a}(z,\nabla u)=\mathcal{F}_{\epsilon}(z,\nabla u)\nabla u$.

(i) Monotonicity. By \cite[Proposition~3.1]{RACSAM-2023}, for all $\xi,\eta\in \mathbb{R}^N$

\[
\begin{split}
(\mathcal{F}_{\epsilon}(z,\xi)\xi & - \mathcal{F}_{\epsilon}(z,\eta)\eta)\cdot (\xi-\eta) \geq C_q b_\epsilon\begin{cases}
|\xi-\eta|^q & \text{if $q\geq 2$},
\\
(\epsilon^2+|\xi|^2+|\eta|^2)^{\frac{q-2}{2}}|\xi-\eta|^2 & \text{if $1<q<2$}
\end{cases}
\\
& + C_p a_\epsilon\begin{cases}
|\xi-\eta|^p & \text{if $p\geq 2$},
\\
(\epsilon^2+|\xi|^2+|\eta|^2)^{\frac{p-2}{2}}|\xi-\eta|^2 & \text{if $1<p<2$}.
\end{cases}
\end{split}
\]
At every point $z\in Q_T$, either $\overline{s}(z)=p(z)$, or $\overline{s}(z)=q(z)$. Keeping one term on the right-hand side and dropping the rest we obtain

\[
(\mathcal{F}_{\epsilon}(z,\xi)\xi  - \mathcal{F}_{\epsilon}(z,\eta)\eta)\cdot (\xi-\eta)
\geq \epsilon C_{p,q} \begin{cases}
|\xi-\eta|^{\overline{s}} & \text{if $\overline{s}\geq 2$},
\\
(\epsilon^2+|\xi|^2+|\eta|^2)^{\frac{\overline{s}-2}{2}}|\xi-\eta|^2 & \text{if $1<\overline{s}<2$}.
\end{cases}
\]

(ii) Boundedness and growth. For all $\xi\in \mathbb{R}^N$

\[
|\mathcal{F}_\epsilon(z,\xi)\xi|\leq a_\epsilon C_p'\left(1+|\xi|^{p-1}\right)+ b_\epsilon C'_q\left(1+|\xi|^{q-1}\right)\leq C_{q,p}\left(1+|\xi|^{\overline{s}-1}\right).
\]
For every $\mu>0$

\[
|\xi|^\mu\leq (\epsilon^2+|\xi|^2)^{\frac{\mu}{2}}\;\begin{cases}
\leq (2\epsilon^2)^{\frac{\mu}{2}} & \text{if $|\xi|<\epsilon$},
\\
=(\epsilon^2+|\xi|^2)^{\frac{\mu-2}{2}}(\epsilon^2+|\xi|^2)\leq 2(\epsilon^2+|\xi|^2)^{\frac{\mu-2}{2}}|\xi|^2 & \text{if $|\xi|\geq \epsilon$},
\end{cases}
\]
whence

\[
\begin{split}
\mathcal{F}_\epsilon(z,\xi)|\xi|^2 & =a_\epsilon (\epsilon^2+|\xi|^2)^{\frac{p-2}{2}}|\xi|^2+b_\epsilon (\epsilon^2+|\xi|^2)^{\frac{q-2}{2}}|\xi|^2
\geq C'|\xi|^{\overline{s}}-C''.
\end{split}
\]

(iii) Continuity w.r.t. $z=(x,t)$. It is sufficient to check continuity for a single component of one of the two terms of $\mathcal{F}_\epsilon(z,\xi)\xi$. Fix $i\in \overline{1,N}$, a point $(x_0,t_0)$, and two arbitrary points $z_1,z_2\in Q^{(R)}(x_0,t_0)$. Write

\[
\begin{split}
a_\epsilon(z_2) & (\epsilon^2+|\xi|^2)^{\frac{p(z_2)-2}{2}}\xi_i - a_\epsilon(z_1)(\epsilon^2+|\xi|^2)^{\frac{p(z_1)-2}{2}}\xi_i
\\
& =\left(a(z_2)-a(z_1)\right)(\epsilon^2+|\xi|^2)^{\frac{p(z_2)-2}{2}}\xi_i +a_\epsilon(z_1)\left((\epsilon^2+|\xi|^2)^{\frac{p(z_2)-2}{2}} -(\epsilon^2+|\xi|^2)^{\frac{p(z_1)-2}{2}}\right)\xi_i
\\
& \equiv I_1+I_2.
\end{split}
\]
Apart from the number $p^+_R$ defined in \eqref{eq:structure} we will also need $q^+_R=\sup_{Q^{(R)}(x_0, t_0)}q(z)$. The estimate on $I_1$ follows if the coefficient $a \in C^{\alpha,\frac{\alpha}{2}}(Q_T)$, $\alpha \in (0,1)$:

\[
\begin{split}
|I_1|\leq |a(z_2)-a(z_1)|(\epsilon^2+|\xi|^2)^{\frac{p(z_2)-1}{2}} & \leq C |a(z_2)-a(z_1)| \left(1+|\xi|^{p_R^+-1}\right)  \leq C d^\alpha(z_1, z_2) \left(1+ |\xi|^{p_R^+-1}\right).
\end{split}
\]
Since
\[
\begin{split}
I_2 & = a_\epsilon(z_1)\xi_i\int_{0}^1\dfrac{d}{d\theta} \left(\epsilon^2+|\xi|^{2}\right)^{\frac{p(\theta z_2+(1-\theta)z_1)}{2}-1}\,d\theta
\\
& = \frac{1}{2}a_\epsilon(z_1)\xi_i\int_{0}^1 \left(\epsilon^2+|\xi|^{2}\right)^{\frac{p(\theta z_2+(1-\theta)z_1)}{2}-1}\ln (\epsilon^2+|\xi|^{2})\dfrac{d}{d\theta}p(\theta z_2+(1-\theta)z_1)\,d\theta,
\end{split}
\]
then

\[
\begin{split}
   |I_2| & \leq L_p\max a_\epsilon \int_{0}^1 \left(\epsilon^2+|\xi|^{2}\right)^{\frac{p(\theta z_2+(1-\theta)z_1)-1}{2}}\,d\theta|\ln (\epsilon^2+|\xi|^{2})|\left(|x_2-x_1|+|t_2-t_1|\right)
   \\
   & \leq C_\epsilon L_p \int_{0}^1 \left(1+|\xi|^{2}\right)^{\frac{p(\theta z_2+(1-\theta)z_1)-1}{2}}\,d\theta \left(1+\ln (1+|\xi|^{2})\right)|\left(|x_2-x_1|+|t_2-t_1|\right)\\
   & \leq C_\epsilon L_p d(z_1, z_2) \left(1+\ln(1+|\xi|^2)\right) \left(1+|\xi|^{p_R^+-1}\right)
\end{split}
\]
where we denoted by $L_p$ the Lipschitz constant of $p(z)$, represented

\[
\begin{split}
\ln (\epsilon^2+|\xi|^{2}) & = \ln \epsilon^2 + \ln\left(1+\dfrac{|\xi|^2}{\epsilon^2}\right)
= \ln \epsilon^2 +\int_0^{|\xi|^2}\left(\ln\left(1+\frac{s}{\epsilon^2}\right))\right)'\,ds
=  \ln \epsilon^2 + \frac{1}{\epsilon^2}\int_0^{|\xi|^2}\dfrac{ds}{1+\frac{s}{\epsilon^2}},
\end{split}
\]
and used the inequality (for $\epsilon\in (0,1)$)

\[
|\ln (\epsilon^2+|\xi|^{2})|\leq |\ln \epsilon^2| + \frac{1}{\epsilon^2}\int_0^{|\xi|^2}\dfrac{ds}{1+s}\leq |\ln \epsilon^2| + \frac{1}{\epsilon^2}\ln(1+|\xi|^2)\leq C_\epsilon \left(1+\ln(1+|\xi|^2)\right)
\]
and by Young's inequality, for $z_1,z_2\in Q^{(R)}(x_0,t_0)$ and $\theta\in (0,1)$

\[
\left(1+|\xi|^{2}\right)^{\frac{p(\theta z_2+(1-\theta)z_1)-1}{2}}\leq C(1+|\xi|^{p^+_R -1 }).
\]
Gathering these inequalities for all terms of $\mathcal{F}_\epsilon(z,\xi)\xi$ we obtain

\[
\begin{split}
|\mathcal{F}_\epsilon(z_1,\xi)\xi-\mathcal{F}_\epsilon(z_2,\xi)\xi| &\leq C(L_p+L_q) d^\alpha(z_1, z_2) \left(1+\ln(1+|\xi|^2)\right) \left(1+|\xi|^{p^+_R-1}+ |\xi|^{q^+_R-1}\right)\\
& \leq C' d^\alpha (z_1, z_2) \left(1+\ln(1+|\xi|^2)\right) \left(1+ |\xi|^{\overline{s}_R-1}\right),\quad \overline{s}_R=\sup_{Q^{(R)}}\overline{s}(z).
\end{split}
\]
It follows that $\mathcal{F}_\epsilon(z,\xi)\xi$ are continuous with respect to  $z$.

\begin{lemma}[Theorem 1, \cite{DZZ-2020}]
\label{le:global-bounded}
Let $N\geq 2$ and $\partial \Omega\in C^1$. Assume that the function $\mathbf{a}$ satisfies conditions \eqref{eq:structure}. Then the weak solution of the homogeneous Dirichlet problem for equation \eqref{eq:nondiv} satisfies the estimate $\forall t\in (0,T)$

\begin{equation}
\label{eq:global-bounded}
\operatorname{ess}\sup_{\Omega}|u(\cdot,t)|\leq C\left(\operatorname{ess}\sup_{\Omega}|u_0| +C\left[1+\int_{Q_T}|u|^{\delta(x,\tau)}\,dxd\tau\right]\right)
\end{equation}
with $\delta(x,\tau)=\max\left\{\overline{s}(x,\tau),2\right\}$ and a constant $C$ depending only on the structural constants in \eqref{eq:structure}.
\end{lemma}

By \eqref{eq:est-sol-Gal} and \cite[Lemma 4.4]{ShSS-2022} we know that for every $\gamma\in (0,\frac{1}{N})$

\[
\operatorname{ess}\sup_{(0,T)}\|u\|_{2,\Omega}+\|u\|_{\overline{s}(\cdot)+\gamma,Q_T}\leq C
\]
with a constant $C$ depending only on the data. Since the strong solution is also a weak solution of problem \eqref{eq:main-reg}, it follows from Lemma \ref{le:global-bounded} that the strong solution of problem \eqref{eq:main-reg} with $u_0\in L^{\infty}(\Omega)$ and $f\in L^{\infty}(Q_T)$ is bounded in $Q_T$. The local H\"older continuity of the gradient follows from Proposition \ref{pro:interior} if we consider problem \eqref{eq:main-reg} as problem \eqref{eq:nondiv} with the flux $\mathbf{a}(z,\xi)=\mathcal{F}_\epsilon(z,\xi)\xi$ that satisfies conditions \eqref{eq:structure} with the exponent $\overline{s}(z)$.

\begin{lemma}
\label{le:local-Hold-grad-double-phase}
Assume that $\partial\Omega\in C^1$, $a, b \in C^{\alpha,\frac{\alpha}{2}}(\overline Q_T)$, $\alpha\in (0,1)$, $p, q \in C^{0,1}(\overline Q_T)$, $f\in L^{\infty}(Q_T)$, $u_0\in L^\infty(\Omega)$. Then the strong solution $u$ of problem \eqref{eq:main-reg} is bounded in $Q_T$ and

\[
\nabla u\in C^{\beta,\frac{\beta}{2}}_{loc}(Q_T),\qquad \beta\in (0,1).
\]
\end{lemma}
\subsubsection{H\"older continuity of the gradient up to the lateral boundary}
Let $\partial \Omega\in C^{1+\sigma}$, $\sigma\in (0,1)$. Take a point $x_0\in \partial \Omega$, place the origin of the coordinate system $\{x\}$ at $x_0$ and rotate the new system so that $x_N$ points in the direction of the interior normal to $\partial\Omega$ at the origin and $x'=\{x_1,\ldots,x_{N-1}\}$ belongs to the tangent plane. There exists $\delta>0$ such that the part of the hypersurface $B_\delta(0)\cap \partial\Omega$ is the graph of a $C^{1+\sigma}$ function $\Phi(x')$: $x_N=\Phi(x')$. Introduce the new local coordinate system $\{y_i\}_{i=1}^{N}$:

\[
y_N=x_N-\Phi(y'),\qquad y_i=x_i-x_{0 i}, \quad i\not=N.
\]
Given a solution of problem \eqref{eq:main-reg}, define the function $v(y,t)=u(x,t)$. Denote by $\nabla'_{y'}$ the gradient with respect to the variables $y'=(y_1,\ldots,y_{N-1})$:

\[
\begin{split}
D_{x_i} u & =D_{x_i} v= D_{y_i}v - D_{y_N}v D_{y_i} \Phi(y'), \quad i\not=N,\qquad D_{x_N}u= D_{y_N}v,
\\
|\nabla_x u|^2 & = \sum_{i=1}^N(D_{x_i}u)^2= (D_{y_N}v)^2+\sum_{i=1}^{N-1}\left(D_{y_i}v - D_{y_N}v D_{y_i} \Phi(y')\right)^2
\\
& = (1+|\nabla'_{y'}\Phi(y')|^2)(D_{y_N}v)^2 -2D_{y_N}v (\nabla'_{y'}v,\nabla'_{y'}\Phi(y')) + |\nabla'_{y'}v|^2.
\end{split}
\]
By the definition of the change of variables $x\mapsto y$
\[
\nabla_x u= \nabla_yv\cdot\mathbf{B},\qquad \mathbf{B}=\left(
\begin{array}{cc}
\mathbf{I}_{N-1} & 0 \\
-\nabla'\Phi(y') & 1 \\
\end{array}
\right), \qquad \operatorname{det}\dfrac{\partial (x_1,\ldots,x_N)}{\partial (y_1,\ldots,y_N)}=1\;(\text{the Jacobian}),
\]
where $\mathbf{I}_{N-1}$ is the unit $(N-1)\times (N-1)$ matrix. Take some $r>0$ and define the parallelepiped

\[
K_r^+=\{y_i\in (-r,r),\,i\not=N,\;\;y_N\in (0,r)\}.
\]
The number $r$ can be chosen so small that the image of $K^+_r$ under the inverse transformation $y\mapsto x$ belongs to $\Omega\cap B_\delta(0)$. Denote $K_r=\{y_i\in (-r,r),\,i=\overline{1,N}\}$ and accept the notation

\[
\widehat p(y,t)=p(x,t),\quad \widehat q(y,t)=q(x,t), \quad \widehat a(y,t)=a(x,t),\quad \widehat b(y,t)=b(x,t),\quad \widehat f(y,t)=f(x,t)
\]
for the functions $p$, $q$, $a$, $b$ defined in $\Omega\cap B_\delta(0)$ and re-defined in $K_r^+$ as functions of the new variables $y$. Introduce the odd extensions from $K_r^+$ to $K_r$,

\[
\widetilde v(y,t)=\begin{cases}
v((y',y_N),t) & \text{if $y_N\geq 0$},
\\
-v((y',-y_N),t) & \text{if $y_N<0$}
\end{cases},\qquad \widetilde f(y,t)=\begin{cases}
\widehat f((y',y_N),t) & \text{if $y_N\geq 0$},
\\
-\widehat f((y',-y_N),t) & \text{if $y_N<0$},
\end{cases}
\]
and denote by $\widetilde p$, $\widetilde q$, $\widetilde a$, $\widetilde b$,  and $\overline{\mathfrak{s}}$ the even extensions of the coefficients and exponents $\widehat p$, $\widehat q$, $\widehat a$, $\widehat b$, $\overline{s}$ from $K_r^+$ to $K_r$. Denote
\[
\widetilde{F}_\epsilon ((y,t),\nabla_y v)\nabla_y v=(\mathbf{B}^\ast\cdot\mathbf{B})\left(\widehat a (\epsilon^2+|\nabla_y v\cdot \mathbf B|^2)^{\frac{\widehat p-2}{2}}+\widehat b (\epsilon^2+|\nabla_y v\cdot \mathbf B|^2)^{\frac{\widehat q-2}{2}}\right)\nabla_yv.
\]
The function $v$ is a local solution of the equation

\begin{equation}
\label{eq:change-coord}
v_t -\operatorname{div}_y\left(\widetilde{F}_\epsilon ((y,t),\nabla_y v)\nabla_y v\right)=\widehat f \quad \text{in $K_r^+\times (0,T)$}.
\end{equation}
By Definition \ref{eq:def-nondiv} and the rule of change of variables $x\mapsto y$, for every $\phi$ such that $\nabla \phi\in L^{\overline{\mathfrak{s}}(\cdot)}(K_r^+\times (0,T))$, $\phi_t\in L^2(K_r^+\times (0,T))$, $\phi(\cdot,t)=0$ on $\partial K_r^+$ for a.e. $t\in (0,T)$ and $\phi(y,T)=0$

\begin{equation}
\label{eq:test+}
\int_{K_r^+\times (0,T)}\left(-v\phi_t+\widetilde{F}_{\epsilon}((y,t),\nabla_y v)\nabla_yv\cdot\nabla_y\phi-\widetilde f\phi\right)\,dydt-\int_{K_r^+}v(y,0)\phi(y,0)\,dy=0.
\end{equation}
Set

\[
I(\widetilde v,\phi)=\int_{K_r\times (0,T)}\left(-\widetilde v \phi_t+\widetilde {\mathcal{F}}_\epsilon((y,t),\nabla_y \widetilde  v)\nabla_y\widetilde v\cdot \nabla_y\phi-\widetilde f \phi\right)\,dydt -\int_{K_r}\widetilde v(y,0)\phi(y,0)\,dy.
\]
Take an arbitrary $\phi(y,t)\in W^{1,\overline{\mathfrak{s}}(\cdot)}(K_r\times (0,T))$ with $\operatorname{supp} \phi(\cdot,t)\Subset K_r$ and represent $\phi(y,t)=\psi+\chi$, where $\psi$ and $\chi$ are the even and odd parts of $\phi$ with respect to $y_N$ defined by

\[
\psi(y,t)=\dfrac{\phi((y',y_N),t)+\phi((y',-y_N),t)}{2}, \qquad \chi(y,t)=\dfrac{\phi((y',y_N),t)-\phi((y',-y_N),t)}{2}.
\]
Write
\[
I(\widetilde v,\phi)= I(\widetilde v,\psi)+I(\widetilde v,\chi).
\]
Since $\widetilde v$ and $\widetilde f$ are odd functions of $y_N$, and $\psi$ is even, changing the variable $y_N\mapsto -y_N$ we obtain the equality

\[
I(\widetilde v,\psi)=\left.I(\widetilde v,\psi)\right|_{K_r^+}+\left.I(\widetilde v,\psi)\right|_{K_r\cap \{y_N\leq 0\}} = \left.I(\widetilde v,\psi)\right|_{K_r^+}-\left.I(\widetilde v,\psi)\right|_{K_r^+}=0.
\]
On the other hand, for odd $\widetilde v$, $\widetilde f$, and $\chi$

\[
I(\widetilde v,\chi)=\left.I(\widetilde v,\chi)\right|_{K_r^+}+\left.I(\widetilde v,\chi)\right|_{K_r\cap \{y_N\leq 0\}} = 2\left.I(\widetilde v,\chi)\right|_{K_r^+}.
\]
By definition, $\widetilde v$ is the odd continuation of the local solution from $K^+_r\times (0,T)$ to $K_r\times (0,T)$. Since $\chi$ is an odd function of $y_N$, it is necessary that $\chi((y',0),t)=0$, whence $\chi=0$ on $\partial K_r^+\times (0,T)$. Because $v$ is a local solution in $K^{+}_{r}\times (0,T)$, it follows that

\[
I(\widetilde v,\phi)=2\left.I(\widetilde v,\chi)\right|_{K^+_r}=0.
\]
Thus, $I(\widetilde v,\phi)=0$ for every test-function $\phi(y,t)$ such that $\nabla \phi\in L^{\overline{\mathfrak{s}}(\cdot)}(K_r\times (0,T))$, $\phi_t\in L^2(K_r\times (0,T))$, $\phi=0$ on $\partial K_r\times (0,T)$ and for $t=T$. It follows that if $v$ is a local weak solution of equation \eqref{eq:change-coord} in the cylinder $K_r^+\times (0,T)$, its odd extension $\widetilde v$ is a weak solution of the same equation in the cylinder $K_r\times (0,T)$.

To conclude about the local H\"older continuity of the gradient of the extended solution $\widetilde v$ it remains to check that the flux $\widetilde{\mathcal{F}}_{\epsilon}((y,t),\xi)\xi$ satisfies conditions \eqref{eq:structure}. The fulfillment of these conditions follows by the same arguments that were used for the flux written in the original variables $x$. Applying Lemma \ref{le:local-Hold-grad-double-phase} to the function $\widetilde v(y,t)$ in $K_r$ and then reverting to the variable $x$ we obtain the following assertion.

\begin{lemma}
\label{le:lateral-Hold-grad}
Assume $\partial\Omega\in C^{1+
\sigma}$. If the conditions of Lemma \ref{le:local-Hold-grad-double-phase} are fulfilled, then for every bounded solution $u$ of equation \eqref{eq:main-reg} and any $\lambda\in (0,T)$

\[
\nabla u\in C^{\gamma,\frac{\gamma}{2}}(\overline{\Omega}\times (\lambda,T)),\qquad \gamma\in (0,1).
\]
\end{lemma}

\subsubsection{Continuation to the cylinder $\Omega\times (-T,T)$}
Let $u$ be a weak solution of problem \eqref{eq:main-reg} in $Q_T$. Define the functions

\[
F(z)=\begin{cases}
f(z) & \text{in $Q_T=\Omega\times (0,T)$},
\\
-\operatorname{div} \mathcal{F}_\epsilon ((x,0),\nabla u(x,0))\nabla u(x,0) & \text{in $\Omega\times (-T,0)$},
\end{cases}
\qquad
w(z)=\begin{cases}
u(z) & \text{in $Q_T$},
\\
u(x,0) & \text{in $\Omega\times (-T,0)$}.
\end{cases}
\]
The function $w$ is a weak solution of the equation

\begin{equation}
\label{eq:main-extended}
w_t - \operatorname{div} \left(\mathcal{F}_\epsilon (z,\nabla w)\nabla w\right)=F(z)\qquad \text{in $\Omega\times (-T,T)$}
\end{equation}
in the sense of Definition \ref{def:weak-sol}. Assume that the data satisfy the conditions of Lemma \ref{le:lateral-Hold-grad} and, moreover, $u(x,0)\in C^2(\overline \Omega)$. Then $F\in L^\infty(\Omega\times (-T,T))$ and we may apply Lemma \ref{le:lateral-Hold-grad} to the weak solution $w$ of equation \eqref{eq:main-extended} in the cylinder $\Omega\times (-T,T)$.

\begin{lemma}
\label{le:global-Hold-grad}
Let the conditions of Lemma \ref{le:lateral-Hold-grad} be fulfilled and $u(x,0)\in C^2(\overline \Omega)$. Then there exist $\gamma\in (0,1)$ and a finite $\rho>0$ such that

\begin{equation}
\label{eq:global-grad}
\nabla u\in C^{\gamma,\frac{\gamma}{2}}(\overline Q_T),\qquad \|\nabla u\|_{C^{\gamma,\frac{\gamma}{2}}(Q_T)}\leq \rho.
\end{equation}
\end{lemma}

\subsection{Existence of a classical solution}
Assume that $a,b,p,q,f\in C^{\infty}(Q_T)$ and $u_0\in C^{\infty}(\Omega)$, $\operatorname{supp}u_0\Subset \Omega$ and $\operatorname{supp} f(\cdot,t)\Subset \Omega$ for all $t \in [0,T]$. Consider the problem

\begin{equation}
\label{eq:main-reg-equiv}
\begin{split}
& u_t - \sum_{i,j=1}^N A_{ij}(z,\nabla u)D^2_{ij}u -\sum_{i=1}^NB_{i}(z,\nabla u) D_iu=f\quad \text{in $Q_T$},
\\
& \text{$u=0$ on $\partial\Omega\times (0,T)$},\qquad \text{$u(x,0)=u_0$ in $\Omega$}
\end{split}
\end{equation}
with the coefficients

\[
\begin{split}
A_{ij}(z,\nabla u) & =a_{\epsilon}w_\epsilon^{\frac{p-2}{2}} \left[\delta_{ij}+(p-2)\dfrac{D_iu D_ju}{w_\epsilon}\right]+b_{\epsilon}w_\epsilon^{\frac{q-2}{2}} \left[\delta_{ij}+(q-2)\dfrac{D_iu D_ju}{w_\epsilon}\right],
\\
B_i(z,\nabla u) & = \ln(\epsilon^2+|\nabla u|^2)\left[a_\epsilon (\epsilon^2+|\nabla u|^2)^{\frac{p-2}{2}}D_ip + b_\epsilon (\epsilon^2+|\nabla u|^2)^{\frac{q-2}{2}}D_iq\right]
\\
&\qquad  + \left((\epsilon^2+|\nabla u|^2)^{\frac{p-2}{2}}D_ia + (\epsilon^2+|\nabla u|^2)^{\frac{q-2}{2}}D_ib\right),
\\
w_\epsilon & = \epsilon^2+|\nabla u|^2.
\end{split}
\]
Take a function $v\in C^{1+\alpha,\frac{1+\alpha}{2}}(Q_T)$ and consider the auxiliary linear problem

\begin{equation}
\label{eq:aux-1}
\begin{split}
& u_t - \sum_{i,j=1}^N A_{ij}(z,\nabla v)D^2_{ij}u -\sum_{i=1}^NB_{i}(z,\nabla v) D_iu=\tau f
\\
& \text{$u=0$ on $\partial\Omega\times (0,T)$},\qquad \text{$u(x,0)=\tau u_0$ in $\Omega$},\qquad \tau\in [0,1].
\end{split}
\end{equation}
For every $\epsilon\in (0,1)$ and $v\in C^{1+\alpha,\frac{1+\alpha}{2}}(Q_T)$ with $\alpha\in (0,1)$ the linear equation \eqref{eq:aux-1} is uniformly parabolic with the ellipticity constants depending on $\epsilon$, $a^\pm$, $b^\pm$, $p^\pm$, $q^\pm$, and the coefficients $A_{ij}(z,\nabla v),B_i(z,\nabla v)\in C^{\beta,\frac{\beta}{2}}(Q_T)$ with $\beta\in (0,1)$. According to the classical parabolic theory \cite[Ch. IV,Th. 5.2]{LSU}, for every $(v,\tau)$ problem \eqref{eq:aux-1} has a unique classical solution $u\in C^{2+\beta,\frac{2+\beta}{2}}(Q_T)$, which satisfies the estimate

\begin{equation}
\label{eq:Schauder-1}
\|u\|_{C^{2+\beta,\frac{2+\beta}{2}}(Q_T)}\leq C\left(\tau\|u_0\|_{C^{2+\beta,\frac{2+\beta}{2}}(Q_T)} +\tau\|f\|_{C^{\beta,\frac{\beta}{2}}(Q_T)}\right)
\end{equation}
with a constant $C$ depending on $p$, $q$, $a$, $b$, $\epsilon$, $N$, $\partial \Omega$, $\|\nabla v\|_{C^{\alpha,\frac{\alpha}{2}}(Q_T)}$ but independent of $u$. Problem \eqref{eq:aux-1} generates the mapping $u=\Phi(v,\tau)$ that puts the solution $u$ into correspondence to the pair $(v,\tau)$. The fixed points of the mapping $\Phi(v,1)$ are solutions of the nonlinear problem \eqref{eq:main-reg-equiv}. Let us denote

\[
B_R=\left\{v\in C^{\alpha,\frac{\alpha}{2}}(Q_T):\,\| v\|_{C^{1+\alpha,\frac{1+\alpha}{2}}(Q_T)}<R\right\},\qquad S_R=B_R\times [0,1],
\]
with $R$ to be defined.  The existence of a fixed point of the mapping
\[
\Phi(v,\tau):\, \overline S_R\to C^{1+\alpha,\frac{1+\alpha}{2}}(Q_T)
\]
follows from the Leray-Schauder principle \cite[Ch. IV,Th. 8.1]{LU}: the equation $u=\Phi(u,\tau)$ has at least one solution for all $\tau\in [0,1]$ if

\begin{enumerate}
\item the equation $u=\Phi(v,0)$ only has the trivial solution;

\item the mapping $\Phi: \, C^{1+\alpha,\frac{1+\alpha}{2}}(Q_T)\times [0,1]\mapsto C^{1+\alpha,\frac{1+\alpha}{2}}(Q_T)$ is compact;

\item $\Phi$ is uniformly continuous w.r.t. $\tau$ on $\overline S_R$;

\item the boundary of $B_R$ does contain fixed points of $\Phi$.
\end{enumerate}
Items (1) and (3) follow from inequality \eqref{eq:Schauder-1}. Item (2) follows from \eqref{eq:Schauder-1} and compactness of the embedding $C^{2+\beta,\frac{2+\beta}{2}}(Q_T)\subset C^{1+\alpha,\frac{1+\alpha}{2}}(Q_T)$ with any $\alpha,\beta\in (0,1)$. Item (4) is inferred from the global a priori estimate \eqref{eq:global-grad} by choosing $R$ large enough. Let $u\in C^{1+\alpha,\frac{1+\alpha}{2}}(Q_T)$ be a fixed point of the mapping $\Phi$. Then $u$ is a weak solution of problem \eqref{eq:main-reg} with smooth data, and it follows from \eqref{eq:global-grad} that $u$ does not belong to the boundary of the ball $B_R$ if we take $R=\rho+1$.

The functions $D_iu$ satisfy equation \eqref{eq:aux-1} differentiated in $x_i$. For every $\delta\in (0,T)$ and $\Omega'\Subset \Omega$, the inclusions $D_iu \in C^{2+\alpha,1+\frac{\alpha}{2}}(\Omega'\times (\delta,T))$ follow from the local estimates on the classical solutions to parabolic equations \cite{LSU}.

\begin{lemma}
\label{le:class-sol-existence}
Let the data of problem \eqref{eq:aux-1} satisfy the following conditions: $\partial\Omega\in C^{2+\gamma}$, $\gamma\in (0,1)$, $u_0\in C^\infty(\overline\Omega)$, $\operatorname{supp} u_0\Subset \Omega$, $f\in C^{\infty}(Q_T)$ with $\operatorname{supp} f(\cdot,t)\Subset \Omega$ for every $t\in (0,T)$, $a,b\in C^{\infty}(\overline Q_T)$, $p,q\in C^{\infty}(\overline Q_T)$. If $p^-, q^->\frac{2N}{N+2}$, then for every $\epsilon \in (0,1)$ problem \eqref{eq:aux-1} admits a classical solution $u\in C^{2+\alpha,1+\frac{\alpha}{2}}(\overline Q_T)$ with $D_iu\in C^{2+\alpha,1+\frac{\alpha}{2}}_{loc}(Q_T)$.
\end{lemma}

\section{Interpolation inequalities}
\label{sec:interpolation}

We denote $r_1, r_2: \overline{\Omega} \times [0,T] \to \mathbb{R}$, use the notation $\underline{s}$, $\overline{s}$ introduced in \eqref{eq:s} and repeatedly use the threshold values
\begin{equation}
\label{eq:threshold}
\begin{split}
r^\sharp=\dfrac{4}{N+2}, \qquad r_\ast=\dfrac{2}{N+2}.
\end{split}
\end{equation}

The bulk of this section is devoted to deriving interpolation inequalities for functions of the variable $x\in \Omega$. For this reason, we denote by $p(x)$, $q(x)$, $\underline{s}(x)$, $\overline{s}(x)$, $a(x)$, $b(x)$, $r_1(x)$, $r_2(x)$ the functions of the variables $(x,t)$ with ``frozen" $t\in [0,T]$.
\begin{equation}
\label{eq:gamma}
\begin{split}
\mathcal{F}^{(r_1,r_2)}_{\epsilon}(x,\xi)= a_\epsilon (\epsilon^2+|\xi|^2)^{\frac{p(x)+r_1(x)-2}{2}} + b_\epsilon  (\epsilon^2+|\xi|^2)^{\frac{q(x)+r_2(x)-2}{2}}.
\end{split}
\end{equation}
The function $\mathcal{F}_{\epsilon}^{(r_1,r_2)}(x,\xi)$ depends on $x$ implicitly, through the exponents $p$, $q$, $r_1$, $r_2$ and the coefficients $a$, $b$. We will write $\mathcal{F}_{\epsilon}^{(r_1,r_2)}(x,\xi)$ if the exponents and coefficients depend only on $x\in\Omega$,
or $\mathcal{F}_{\epsilon}^{(r_1,r_2)}(z,\xi)$ if at least one of these functions depends on $z=(x,t)\in Q_T$.

The inequalities proven in this section extend the interpolation inequalities derived in \cite{RACSAM-2023} to the case of irregular coefficients $a$ and $b$. We omit the parts of the proofs that can be recovered from \cite{RACSAM-2023} and present in full detail the study of the terms that contain $|\nabla a|$, $|\nabla b|$. As distinguished from \cite{RACSAM-2023}, the functions $|\nabla a|$, $|\nabla b|$ are now assumed to belong to $L^d(\Omega)$ with some $d=d(N,p(\cdot),q(\cdot))$ and may be unbounded.

\begin{lemma}[Lemma 4.1, \cite{RACSAM-2023}]
\label{lem:interpol} Let $\partial \Omega \in C^2$, $u\in
C^1(\overline{\Omega})\cap H^{2}(\Omega)\cap H_0^1(\Omega)$. Assume that $p$, $q$, $a$, $b$ are functions of $x\in \Omega$ and satisfy conditions \eqref{eq:structure-prelim-1} and \eqref{eq:s}. If

\[
\int_\Omega\mathcal{F}^{(0,0)}_{\epsilon}(x,\nabla
u)\vert u_{xx}\vert ^2\,dx<\infty,\qquad \int_{\Omega}u^2\,dx= M_0,
\]
then for every nonnegative functions $r_1,r_2\in C^{0,1}(\Omega)$ and arbitrary parameters $\delta, \nu \in (0,1)$
\begin{equation}
\label{eq:principal}
\begin{split}
\int_{\Omega} & \mathcal{F}_{\epsilon}^{(r_1,r_2)}(x,\nabla u)\vert \nabla u\vert ^2\,dx  \leq \delta \int_{\Omega}\mathcal{F}^{(0,0)}_{\epsilon}(x,\nabla
u)\vert u_{xx}\vert ^{2}\,dx
\\
& +
C_0 + C_1  \int_{\Omega}u^2\,dx
+ C_2 \int_{\Omega} \left( a_\epsilon \vert u\vert ^{\frac{p(x)+r_1(x)}{1-\nu}}+ b_\epsilon \vert u\vert ^{\frac{q(x)+r_2(x)}{1-\nu}}\right)\,dx
\\
&
+ C_3  \int_\Omega  \vert u\vert \ \left( |\nabla a| \ w_\epsilon ^{\frac{p(x)+r_1(x)-1}{2}} + |\nabla b| \ w_{\epsilon}^{\frac{q(x)+r_2(x)-1}{2}}\right)\,dx\\
& + C_4\int_\Omega \vert u\vert ^2 \ \left(  a_\epsilon w_{\epsilon}^{\frac{p(x)+2r_1(x)-2}{2}}+ a_\epsilon w_{\epsilon}^{\frac{q(x)+2r_2(x)-2}{2}}\right)\,dx
\end{split}
\end{equation}
with independent of $u$ constants $C_i=C_i(\partial
\Omega,\delta,\kappa, \nu, \underline{s}^-, \overline{s}^+,N,r_i^\pm, L_{p,q},L_{r_i}, M_0, \|a\|_{\infty,\Omega}, \|b\|_{\infty,\Omega})$.
\end{lemma}

The proof consists in application of the Green formula,

\[
\begin{split}
\int_{\Omega}a_\epsilon w_{\epsilon}^{\frac{p+r_1-2}{2}}
\vert \nabla u\vert ^2\,dx & =\int_{\Omega} a_\epsilon w_{\epsilon}^{\frac{p+r_1-2}{2}} \nabla u\cdot \nabla u\,dx
\\
&= \int_\Omega a_\epsilon u\ w_{\epsilon}^{\frac{p+r_1-2}{2}} \Delta u\,dx +
 \int_{\Omega} u w_{\epsilon}^{\frac{p+r_1-2}{2}}  \nabla u \cdot
\nabla a \,dx
\\
&  \qquad + \int_\Omega (p+r_1-2)\ a_\epsilon u w_{\epsilon}^{\frac{p+r_1-2}{2}-1}\sum_{i=1}^n
\left(u_{x_i}\sum_{j=1}^{n}u_{x_j}u_{x_ix_j}\right)\,dx
\\
&  \qquad + \int_{\Omega}a_\epsilon u w_{\epsilon}^{\frac{p+r_1-2}{2}} \ln w_{\epsilon} \nabla u \cdot
\nabla (p+r_1) \,dx
:= J_1 + J_2 + J_3 + J_4,
\end{split}
\]
and the termwise estimating of $J_i$. Inequality \eqref{eq:principal} follows after gathering the results with similar estimates for the integral of $b_\epsilon w_\epsilon^{\frac{q+r_2-2}{2}}|\nabla u|^2$.

Let us rewrite inequality \eqref{eq:principal} in the form
\begin{equation}\label{eq:complete}
\begin{split}
\int_\Omega \mathcal{F}_\epsilon^{(r_1,r_2)}(x,\nabla u)\vert \nabla u\vert ^2\,dx & \leq \delta \int_{\Omega}\mathcal{F}^{(0,0)}_{\epsilon}(x,\nabla
u)\vert u_{xx}\vert ^{2}\,dx
+ C + C_0  \int_{\Omega}u^2\,dx
\\
&
\qquad + C_1  \mathcal{Q}_1^{(r_1,r_2)} + C_2  \mathcal{Q}_2^{(r_1,r_2)} + C_3 \mathcal{Q}_3^{(r_1,r_2)},
\end{split}
\end{equation}
where
\[
\begin{split}
& \mathcal{Q}_1^{(r_1,r_2)}:= \int_{\Omega} \left(a_\epsilon \vert u\vert ^{\frac{p+r_1}{1-\nu}} + b_\epsilon \vert u\vert ^{\frac{q+r_2}{1-\nu}}\right)\,dx \quad \ \text{for some} \ \nu \in (0,1),
\\
&
\mathcal{Q}_2^{(r_1,r_2)}:= \int_\Omega  \vert u \vert \ \left( |\nabla a| w_{\epsilon}^{\frac{p+r_1-1}{2}} + |\nabla b| w_{\epsilon}^{\frac{q+r_2-1}{2}}\right)\,dx,
\\
& \mathcal{Q}_3^{(r_1,r_2)}:=  \int_\Omega \vert u\vert ^2 \ \left(  a_\epsilon w_{\epsilon}^{\frac{p+2r_1-2}{2}}+ b_\epsilon w_{\epsilon}^{\frac{q+2r_2-2}{2}}\right)\,dx.
\end{split}
\]
The integrals $\mathcal{Q}_i^{(r_1, r_2)}$ are estimated separately and under a special choice of the functions $r_1$, $r_2$. Given $p$, $q$, we introduce the Lipschitz-continuous functions
\begin{equation}
\label{eq:r-i}
r_1(x):= r^\sharp + \underline{s}(x) - p(x), \quad
r_2(x):=
r^\sharp + \underline{s}(x) - q(x)
\end{equation}
with $r^\sharp$ defined in \eqref{eq:threshold}.
Let $\{\Omega_i\}_{i=1}^{K}$ be a finite cover of $\Omega$, $\Omega_i\subset \Omega$,
$\partial\Omega_i\in C^2$, $\Omega\subseteq \bigcup_{i=1}^K\Omega_i$. Set
\[
\begin{split}
\underline{s}_i^+=\max_{\overline{\Omega}_i}\underline{s}(x), \quad \underline{s}_i^-=\min_{\overline{\Omega}_i}\underline{s}(x).
\end{split}
\]
Fix an arbitrary number $\varsigma \in (0, r^\sharp)$. By continuity of the map $x \mapsto \underline{s}(x)$,  the cover $\{\Omega_i\}_{i=1}^{K}$ can be chosen in such a way that for every $i=1,2,\ldots,K$
\begin{equation}
\label{eq:osc-loc}
0\leq  \underline{s}_i^+-\underline{s}_i^- <  \varsigma.
\end{equation}
The number of elements in the cover $\{\Omega_i\}$ depends on $\varsigma$ and the Lipschitz constant $L_{p,q}$ of the exponents  $p$ and $q$ in $\Omega$: $K=K(L_{p,q},\varsigma)$. Under the choice \eqref{eq:r-i} we have $p(x)+r_1(x)=q(x)+r_2(x)=\underline{s}(x)+r^\sharp$.

\begin{lemma}[Lemma 4.3, \cite{RACSAM-2023}]
\label{le:Q1}
Let the conditions of Lemma \ref{lem:interpol} be fulfilled and $a, b \in L^\infty(\Omega)$. Then for every $
\varsigma \in \left(0, \frac{4}{N+2} \right)$,
every $\sigma >0$, and $r_1(x)$, $r_2(x)$ defined in \eqref{eq:r-i},

\[
\mathcal{Q}_{1}^{(r_1-\varsigma, r_2-\varsigma)}\leq C+\sigma  \int_{\Omega}\mathcal{F}_{\epsilon}^{(r_1-\varsigma, r_2-\varsigma)}(x,\nabla u)\vert \nabla u\vert ^2\,dx
\]
with a constant $C$ depending on $\sigma$, $N$, $M_0$, $r_1$, $r_2$, $L_{p,q}$, $p^\pm$, $q^\pm$, $\varsigma$.
\end{lemma}

\begin{lemma}[Lemma 4.4, \cite{RACSAM-2023}]
\label{le:Q3}
Let the conditions of Lemma \ref{lem:interpol} be fulfilled and $a, b \in L^\infty(\Omega)$. Then, for every $\varsigma \in (0, r_\ast/2)$ and every $\mu>0$ and $r_1(x)$, $r_2(x)$ defined in \eqref{eq:r-i},
\[
\begin{split}
\mathcal{Q}_{3}^{(r_1-\varsigma, r_2-\varsigma)} & \leq C+\mu \int_{\Omega}\mathcal{F}_\epsilon^{(r_1-\varsigma, r_2-\varsigma)}(x,\nabla u)\vert \nabla u\vert ^2\,dx
\end{split}
\]
with a constant $C=C(\mu, L_{p,q},N,M_0,K,\alpha,\varsigma, p^\pm, q^\pm)$
\end{lemma}

\begin{lemma}
\label{le:Q2}
Assume that in the conditions of Lemma \ref{lem:interpol}
\begin{equation}
\label{eq:a-b-d}
\text{$|\nabla a|, \ |\nabla b| \in L^d(\Omega)$ with $d > 2+\frac{N+2}{2}\overline{s}^+$.}
\end{equation}
Then, for every

\begin{equation}
\label{eq:varsigma}
\varsigma \in \left(0, \dfrac{d}{d-1}\left(\dfrac{2}{N+2} - \dfrac{(\overline{s}^+ +r^{\sharp})}{d}\right)\right),
\end{equation}
and any $\lambda>0$
\[
\begin{split}
\mathcal{Q}_2^{(r_1-\varsigma, r_2-\varsigma)} & \leq C +\lambda\int_{\Omega}\mathcal{F}_\epsilon^{(r_1-\varsigma, r_2-\varsigma)}(x,\nabla u)\vert \nabla
u\vert ^2\,dx
\end{split}
\]
with a constant
\[
C=C'\left(\lambda,L_{p,q},N,M_0,\varsigma,\alpha, p^\pm, q^\pm)\right) +C''\left(\|\nabla a\|^d_{d,\Omega}+\|\nabla b\|^d_{d,\Omega}\right).
\]
\end{lemma}

\begin{proof}
Take $\varsigma$ satisfying \eqref{eq:varsigma} and choose a finite cover $\{\Omega_i\}$ of $\Omega$ according to condition \eqref{eq:osc-loc}. It is sufficient to consider the integrals $\mathcal{Q}_{2,i}^{(r_1-\varsigma, r_2-\varsigma)}$ over each element of the cover $\{\Omega_i\}$. Fix $\Omega_i$ and set $\rho=\underline{s}_i^-+r^\sharp -\varsigma$. Notice first that conditions \eqref{eq:a-b-d} and \eqref{eq:varsigma} yield the inequality $d' < \dfrac{\rho(N+2)}{N}$.
By repeated use of Young's inequality,
\begin{equation}
\label{eq:Q-2-prim}
\begin{split}
\mathcal{Q}_{2,i}^{(r_1-\varsigma, r_2-\varsigma)} & \leq C_1 \left(1+ M_0 +
\|\nabla a\|^d_{d,\Omega} + \|\nabla b\|^d_{d,\Omega}\right) + C_2 \int_{\Omega_i} \vert u\vert^{d'} \  w_{\epsilon}^{\frac{d'(\underline{s}_i(x)+r^{\sharp}-\varsigma-1)}{2}} \,dx
\\
&
\leq C + \lambda \int_{\Omega_i}\vert u\vert ^{\rho\frac{N+2}{N}}\,dx+C_\lambda \int_{\Omega_i} w_\epsilon^{\frac{\kappa}{2}}\,dx
\end{split}
\end{equation}
with an arbitrary $\lambda>0$ and the exponent
\[
\quad \kappa=(\underline{s}_i(x)+r^{\sharp}-\varsigma-1) \frac{\rho\frac{N+2}{N}}{\rho\frac{N+2}{d'N}-1}>0.
\]
The second term on the right-hand side of \eqref{eq:Q-2-prim} is estimated by the Gagliardo-Nirenberg inequality

\[
\int_{\Omega_i}\vert u\vert ^{\rho\frac{N+2}{N}}\,dx\leq C\left(1+\|  \nabla u\|  _{\rho,\Omega_i}^{\theta \rho \frac{N+2}{N}}\right)\quad\text{with}\quad \theta=\frac{\frac{1}{2}- \frac{N}{\rho(N+2)}}{\frac{N+2}{2N}-\frac{1}{\rho}} =\frac{N}{N+2}\in (0,1),
\]
whence

\[
\lambda\int_{\Omega_i}\vert u\vert ^{\rho\frac{N+2}{N}}\,dx\leq C+ C\lambda\int_{\Omega_i}\vert \nabla u\vert ^{\underline s(x)+r^\sharp-\varsigma}\,dx,\qquad C=C(\underline{s}_i^-,\overline{s}_i^+,N,M_0,\lambda).
\]
To estimate the last term of \eqref{eq:Q-2-prim} we claim that

\[
d'(\underline{s}_i^++r^{\sharp}-\varsigma-1) \frac{\rho\frac{N+2}{d'N}}{\rho\frac{N+2}{d'N}-1} =\max\limits_{\Omega_i}\kappa<\rho=\underline{s}_i^-+r^\sharp-\varsigma.
\]
This inequality is equivalent to
\[
\underline{s}_i^++r^{\sharp}-\varsigma-1  < \frac{\underline{s}_i^-+r^{\sharp}-\varsigma}{d'}\left(1-\frac{d'N}{\rho(N+2)}\right) = \frac{\underline{s}_i^-+r^{\sharp}-\varsigma}{d'}-\frac{N}{N+2},
\]
which can be written in the form
\[
\begin{split}
\underline{s}_i^+- \underline{s}_i^- & <
\left(1-\frac{N}{N+2}\right)
 + \frac{\underline{s}_i^-+r^{\sharp}-\varsigma}{d'} - (\underline{s}_i^-+r^{\sharp}-\varsigma) = \frac{2}{N+2} - \frac{\underline{s}_i^-+r^{\sharp}-\varsigma}{d}.
\end{split}
\]
By virtue of \eqref{eq:osc-loc}, to fulfill the last inequality suffices to claim
\begin{equation}
\label{eq:equiv-check}
\begin{split}
0<\varsigma < \frac{2}{N+2} - \frac{\overline{s}^++r^{\sharp}-\varsigma}{d}.
\end{split}
\end{equation}
Since

\[
\dfrac{2}{N+2}- \frac{\overline{s}^++r^{\sharp}}{d}>0\quad \Leftrightarrow \quad d>\dfrac{N+2}{2}\left(\overline{s}^++\dfrac{4}{N+2}\right)=2+\dfrac{N+2}{2}\overline{s}^+,
\]
\eqref{eq:equiv-check} is equivalent to \eqref{eq:varsigma}. Inequality \eqref{eq:Q-2-prim} continues as follows:
\[
\begin{split}
\mathcal{Q}_{2,i}^{(r_1-\varsigma, r_2-\varsigma)}
&
\leq C+\delta \int_{\Omega_i}\vert \nabla u\vert ^{\underline{s}_i(x)+r^{\sharp}-\varsigma}\,dx
\end{split}
\]
with any $\delta>0$ and $\varsigma$ satisfying \eqref{eq:varsigma}. The estimate on $\mathcal{Q}_2^{(r_1-\varsigma, r_2-\varsigma)}$ follows by collecting the estimates in $\Omega_i$ for all $i=1,2,\ldots,K$.
\end{proof}
Denote

\[
r_{d}:= \min\left\{\frac{r_\ast}{2},
\frac{d}{d-1}\left(\frac{2}{N+2} - \frac{(\overline{s}^+ +r^{\sharp})}{d}\right)\right\}.
\]

\begin{lemma}
\label{le:final-ell}
Let the conditions of Lemma \ref{lem:interpol} and condition \eqref{eq:structure-prelim-2} be fulfilled. Assume that $a$, $b$ satisfy condition \eqref{eq:a-b-d}. For every $u\in C^1(\overline{\Omega})\cap H^1_0(\Omega)\cap H^2(\Omega)$, any $\varsigma \in (0,r^\sharp)$, and an arbitrary $\delta>0$
\begin{equation}
\label{eq:complete-new}
\int_{\Omega}\mathcal{F}_{\epsilon}^{(r_1-\varsigma, r_2-\varsigma)}(x,\nabla u)\vert \nabla u\vert ^2\,dx\leq \delta \int_{\Omega}\mathcal{F}_{\epsilon}^{(0,0)}(x,\nabla u)\vert u_{xx}\vert ^2\,dx +C
\end{equation}
with a constant
\[
C=C'(L_{p,q},N,M_0,\alpha,\delta,\varsigma)+C''\left(\|\nabla a\|_{d,\Omega}^d + \|\nabla b\|_{d,\Omega}^d\right).
\]
\end{lemma}

\begin{proof} For $\varsigma\in (0,r_{d})$ inequality \eqref{eq:complete-new} follows from the estimates of Lemmas \ref{le:Q1}, \ref{le:Q3}, \ref{le:Q2} with $\sigma+\lambda+\mu<1$. If $\varsigma \in (r_{d},r^\sharp)$, it is sufficient to observe that by virtue of Young's inequality and \eqref{eq:null-eps}
\[
\begin{split}
  & \mathcal{F}_{\epsilon}^{(r_1-\varsigma, r_2-\varsigma)}(x,\nabla u)\vert \nabla u\vert ^2\,dx \leq \mathcal{F}_{\epsilon}^{(r_1-\varsigma, r_2-\varsigma)}(x,\nabla u) w_\epsilon \\
  &\leq C_1 +\mathcal{F}_{\epsilon}^{(r_1-\widetilde \varsigma, r_2-\widetilde \varsigma)}(x,\nabla u) w_\epsilon
    \leq C_2 + 2 \mathcal{F}_{\epsilon}^{(r_1-\widetilde\varsigma , r_2-\widetilde \varsigma)}(x,\nabla u)\vert \nabla u\vert ^2\,dx
\end{split}
\]
with $\widetilde \varsigma < \varsigma$, $\widetilde \varsigma \in (0, r_{d})$, and independent of $u$ constants $C_1$, $C_2$.
\end{proof}

\begin{corollary}
\label{cor:d-r}
Let in the conditions of Lemma \ref{le:final-ell} $|\nabla a|, |\nabla b|\in L^{d}(\Omega)$ where

\begin{equation}
\label{eq:d-1}
d>2+\frac{N+2}{2}\left(\overline{s}^++r\right)\quad \text{with $r\geq 0$.}
\end{equation}
Then, for every $\delta>0$ and $s\in (0,r^\sharp)$

\begin{equation}
\label{eq:d-r-ell}
\begin{split}
\int_\Omega (\epsilon^2+|\nabla u|^2)^{\frac{\underline{s}(z)+s+r-2}{2}}|\nabla u|^2\,dx
&
\leq \delta \int_{\Omega}\mathcal{F}_{\epsilon}^{(r,r)}(x,\nabla u)\vert u_{xx}\vert ^2\,dx +C
\end{split}
\end{equation}
with a constant $C$ depending on the same quantities as the constant in \eqref{eq:complete-new}, $s$ and $r$.
\end{corollary}

The estimate of Lemma \ref{le:final-ell} can be extended to the functions defined on the cylinder $Q_T$, provided that for a.e. $t\in (0,T)$ the exponents $p,q$ and the coefficients $a,b$ satisfy the conditions of Lemma \ref{lem:interpol}, \eqref{eq:structure-prelim-2}, and \eqref{eq:structure-prelim-1}.

\begin{theorem}
\label{th:integr-par} Let $\partial \Omega \in C^2$, $u\in
C^{0}([0,T];C^1(\overline{\Omega})\cap H^1_0(\Omega)\cap H^2(\Omega))$.  Assume that the exponents $p(\cdot)$, $q(\cdot)$ satisfy conditions \eqref{assum1}, \eqref{eq:Lip-p-q}, \eqref{eq:structure-prelim-2}, and the coefficients $a$, $b$ satisfy conditions \eqref{eq:structure-prelim-1}. If $r\geq 0$ and

\begin{equation}
\label{eq:cond-embed-par}
\begin{split}
& \int_{Q_T}\mathcal{F}^{(r,r)}_{\epsilon}(z,\nabla
u)\vert u_{xx}\vert ^2\,dz<\infty,\quad \sup_{(0,T)}\|  u(t)\|  _{2,\Omega}^2= M_0,
\\
& \text{$\nabla a,\nabla b\in L^d(Q_T)$ with $d>2+\frac{N+2}{2}\left(\overline{s}^++r\right)$},
\end{split}
\end{equation}
then for every $\varsigma \in (0,r^\sharp)$ and every $\beta\in (0,1)$
\begin{equation}
\label{eq:principal2}
\begin{split}
\int_{Q_T} \mathcal{F}_{\epsilon}^{(r+r_1-\varsigma, r+r_2-\varsigma)}(z,\nabla u)\vert \nabla u\vert ^2\,dz\leq \beta\int_{Q_T}\mathcal{F}^{(r,r)}_{\epsilon}(z,\nabla
u)\vert u_{xx}\vert ^{2}\,dx+C
\end{split}
\end{equation}
with
\[
r_1(z):=\underline{s}(z)+r^\sharp - p(z), \quad \text{and} \quad r_2(z):=\underline{s}(z)+r^\sharp - q(z)
\]
and a constant $C=C(\partial
\Omega,\beta,\alpha, L_{p,q},N,M_0, L, \varsigma,r)$. Moreover, for every $s\in (0,r^\sharp)$
\begin{equation}
\label{eq:principal-3}
\alpha\int_{Q_T}\vert \nabla u\vert ^{\underline{s}(z)+s+r}
\,dz\leq \beta\int_{Q_T}\mathcal{F}^{(r,r)}_{\epsilon}(z,\nabla
u)\vert u_{xx}\vert ^{2}\,dx+C
\end{equation}
with a constant $C$ depending on the same quantities as the constant in \eqref{eq:principal2}.
\end{theorem}

\begin{proof}
Since the exponents $p,q$ are Lipschitz-continuous in $\overline{Q}_T$, for every $0 < \varsigma < r_{d}$ there exists a finite cover of $Q_T$ composed of the cylinders $Q^{(i)}=\Omega_i\times (t_{i-1},t_i)$, $i=1,2,\ldots,K$, such that
\[
\begin{split}
& t_0=0,\quad t_K=T,\quad t_i-t_{i-1}=\rho,\quad Q_T\subset \bigcup_{i=1}^{K}Q^{(i)},\quad \partial\Omega_i\in C^2,
\\
& \underline{s}_i^+=\max_{\overline{Q}^{(i)}} \underline{s}(z),\quad \underline{s}_i^{-}=\min_{\overline{Q}^{(i)}}\underline{s}(z),
\quad \underline{s}_i^+-\underline{s}_i^- <\dfrac{\varsigma}{4},\quad i=1,\ldots,K,\;\;K=K(L_{p,q}\varsigma).
\end{split}
\]
For a.e. $t\in (0,T)$ the function $u(\cdot,t):\Omega\mapsto \mathbb{R}$  satisfies inequality \eqref{eq:complete-new} on each of $\Omega_i$ with $p$ and $q$ substituted by $p+r$ and $q+r$. Integrating these inequalities over the intervals $(t_{i-1},t_i)$ and summing the results we obtain \eqref{eq:principal2} with $\varsigma<r_{d}$. The case $\varsigma\in [r_{d},r^\sharp)$ is considered as in the proof of Lemma \ref{le:final-ell}. Due to the choice of $r_1(z)$, $r_2(z)$, inequality \eqref{eq:principal-3} follows from \eqref{eq:principal2} with the help of \eqref{eq:null-eps}.
\end{proof}

\section{Estimates involving the second-order derivatives}
\label{sec:int-by-parts}
The results of this section hold for all sufficiently smooth functions. As in Section \ref{sec:interpolation}, we first derive estimates for functions depending on the space variables $x$ and then extend them to the functions defined on the cylinder $Q_T$. The notation $\mathcal{F}_\epsilon(x,\nabla u)\nabla u\equiv \mathcal{F}_\epsilon^{(0,0)}(x,\nabla u)\nabla u$ stands for the regularized flux introduced in \eqref{eq:structure-prelim-1}.

\subsection{Integration by parts}
\begin{proposition}\cite[Proposition 3.1]{Ar-Sh-2025}
    Let $\partial\Omega\in C^2$ and $\vec \alpha$, $\vec \beta$ be arbitrary vectors with the components
\begin{equation}
\label{eq:vec-reg}
\alpha_i \in C^0(\overline{\Omega})\cap C^1(\Omega), \qquad \beta_i\in C^1(\overline{\Omega})\cap C^2(\Omega).
\end{equation}
Denote by $\vec\nu$ the unit exterior normal to $\partial\Omega$ and represent $\vec \alpha=\alpha_\nu\vec\nu+ \vec\alpha_\tau$, $\vec \beta=\beta_\nu\vec\nu+ \vec\beta_\tau$, where $\alpha_{\nu}=(\vec \alpha,\vec \nu)$, $\beta_\nu=(\vec \beta,\vec \nu)$ are the normal components of $\vec \alpha$, $\vec \beta$, and $\vec\alpha_\tau$, $\vec\beta_{\tau}$ belong to the tangent plane to $\partial\Omega$. If $\vec\alpha_\tau=0$ and $\vec\beta_{\tau}=0$, then
\begin{equation}
\label{eq:by-parts}
\begin{split}
\int_\Omega & \operatorname{div}\vec \alpha \operatorname{div}\vec \beta\,dx  = -\int_{\partial\Omega}\alpha_{\nu}\beta_{\nu}\operatorname{trace}\mathcal{B} \,dS + \int_\Omega \sum_{i,j=1}^N D_j\alpha_i D_i\beta_j\,dx.
\end{split}
\end{equation}
\end{proposition}
Given a smooth function $u$ defined on $\overline \Omega$, $u=0$ on $\partial\Omega$, introduce the vectors

\begin{equation}
\label{eq:p+r}
\vec \alpha= \mathcal{F}_\epsilon(x,\nabla u)\nabla u\equiv \left(a_\epsilon w_\epsilon ^{\frac{p(x)-2}{2}} + b_\epsilon w_\epsilon ^{\frac{q(x)-2}{2}}\right)\nabla u
\qquad \vec \beta=  w_\epsilon ^{\frac{r-2}{2}}\nabla u,\quad r=const\geq 2,
\end{equation}
To apply formula \eqref{eq:by-parts}  we claim that $\alpha_i$, $\beta_i$ satisfy assumptions \eqref{eq:vec-reg}. By the straightforward computation

\[
\begin{split}
D_j \alpha_i & = a_\epsilon w_\epsilon^{\frac{p-2}{2}}D^2_{ij}u+(p-2) a_\epsilon w_\epsilon^{\frac{p-2}{2}-1}D_iu\sum_{k=1}^N D_ku D^2_{kj}u
+ \frac{1}{2} a_\epsilon w_\epsilon ^{\frac{p-2}{2}}\ln w_\epsilon D_iuD_jp
\\
& \qquad  + w_\epsilon^{\frac{p-2}{2}} D_i u D_j a + w_\epsilon^{\frac{q-2}{2}} D_i u D_j b
\\
& \quad + b_\epsilon w_\epsilon^{\frac{q-2}{2}}D^2_{ij}u+(q-2) b_\epsilon w_\epsilon^{\frac{q-2}{2}-1} D_iu \sum_{k=1}^N D_ku D^2_{kj}u
+ \frac{1}{2} b_\epsilon w_\epsilon ^{\frac{q-2}{2}}\ln w_\epsilon D_iuD_jq.
\end{split}
\]

For every $\epsilon>0$ the inclusion $\alpha_i\in C^2({\Omega})$ holds if $u\in C^3({\Omega})$, $p,q,a,b \in C^2(\Omega)$. Since $u=0$ on $\partial\Omega$, then $\vec\nu=\dfrac{\nabla u}{|\nabla u|}$ and
\[
\begin{split}
&
\alpha_\nu=\vec{\alpha}\cdot \vec{\nu}=\vec{\alpha}\cdot \dfrac{\nabla u}{|\nabla u|}= a_\epsilon w_\epsilon^{\frac{p(z)-2}{2}}|\nabla u| + b_\epsilon w_\epsilon^{\frac{q(z)-2}{2}}|\nabla u|,\qquad \vec \alpha_\tau=\vec \alpha-\alpha_\nu \vec\nu=0,
\\
& \beta_\nu=\vec{\beta}\cdot \vec{\nu}=\vec{\beta}\cdot \dfrac{\nabla u}{|\nabla u|}= w_\epsilon ^{\frac{r-2}{2}}|\nabla u|,\qquad \vec \beta_\tau=\vec \beta-\beta_\nu \vec\nu=0.
\end{split}
\]
The boundary integral in \eqref{eq:by-parts} transforms into

\[
\int_{\partial \Omega} \left(\alpha_\nu \operatorname{div}\vec \beta - ((\vec \alpha\cdot \nabla)\vec \beta)\cdot \vec \nu\right)\,dS= -\int_{\partial\Omega} \left(a_\epsilon w_\epsilon^{\frac{p+r}{2}-2} + b_\epsilon w_\epsilon^{\frac{q+r}{2}-2} \right)|\nabla u|^2\operatorname{trace}\mathcal{B}\,dS
\]
and formula \eqref{eq:by-parts}  becomes

\begin{equation}
\label{eq:double-final-e}
\begin{split}
\int_{\Omega} \operatorname{div} & \left(\mathcal{F}_\epsilon(z,w_\epsilon)\nabla u\right)
\operatorname{div}\left( w_\epsilon ^{\frac{r-2}{2}}\nabla u\right)\,dx
=- \int_{\partial\Omega} \left(a_\epsilon w_\epsilon^{\frac{p+r}{2}-2} + b_\epsilon w_\epsilon^{\frac{q+r}{2}-2} \right)|\nabla u|^2\operatorname{trace}\mathcal{B}\,dS \\
& \qquad + \sum_{i,j=1}^N\int_{\Omega}D_{i}\left( \left(a_\epsilon w_\epsilon ^{\frac{p-2}{2}} + b_\epsilon  w_\epsilon^{\frac{q-2}{2}} \right)D_{j}u \right) D_{j}\left(w_\epsilon ^{\frac{r-2}{2}}D_{i}u\right)\,dx.
\end{split}
\end{equation}

\subsection{Pointwise inequalities}

For a given function $u$ we denote $\vec{\eta}=\dfrac{\nabla u}{  \sqrt{w_\epsilon} }$, $|\vec \eta|<1$. By the straightforward computation
\[
\begin{split}
D_{i} \left(a_\epsilon w_\epsilon ^{\frac{p-2}{2}}D_{j}u\right)
& = w_\epsilon^{\frac{p-2}{2}}\left[ a_\epsilon \left(D^2_{ij}u + (p-2)\eta_j\sum_{k=1}^N D^2_{ki}u\eta_k+ \frac{1}{2}\ln w_\epsilon D_juD_ip\right) + D_j u D_i a\right],
\\
D_j\left(w_\epsilon^{\frac{r-2}{2}} D_ju\right) & =
 w_\epsilon ^{\frac{r-2}{2}}\left(D^2_{ij}u + (r-2)\eta_i\sum_{k=1}^N D^2_{kj}u\eta_k\right),
\end{split}
\]
the the analogous formula holds for $D_{i} \left(b_\epsilon w_\epsilon ^{\frac{q-2}{2}}D_{j}u\right)$. Combining these formulas we have

\begin{equation}
\label{eq:product}
\begin{split}
D_{i} & \left(\mathcal{F}_\epsilon(z,\nabla u)D_ju\right)
D_{j}\left( w_\epsilon ^{\frac{r-2}{2}}D_{i}u\right)
\\
&
=  w_\epsilon^{\frac{p+r}{2}-2} \left[ a_\epsilon \left(D^2_{ij}u + (p-2)\eta_j\sum_{k=1}^N D^2_{ki}u\eta_k+ \frac{1}{2}  \sqrt{w_\epsilon} \ln (\epsilon^2+ |\nabla u|^2) \eta_jD_ip \right) + \sqrt{w_\epsilon} \eta_j D_ia\right]
\\
& \qquad \times
\left(D^2_{ij}u + (r-2)\eta_i\sum_{k=1}^N D^2_{kj}u\eta_k\right)
\\
& \quad + w_\epsilon ^{\frac{q+r}{2}-2} \left[ b_\epsilon  \left(D^2_{ij}u + (q-2)\eta_j\sum_{k=1}^N D^2_{ki}u\eta_k+ \frac{1}{2}  \sqrt{w_\epsilon} \ln (\epsilon^2+ |\nabla u|^2) \eta_j D_iq \right) + \sqrt{w_\epsilon} \eta_j D_ib\right]
\\
& \quad \qquad \times
\left(D^2_{ij}u + (r-2)\eta_i\sum_{k=1}^N D^2_{kj}u\eta_k\right)
\\
&
\equiv  w_\epsilon ^{\frac{p+r}{2}-2} \mathcal{K}_{ij}^{(1)} + w_\epsilon ^{\frac{q+r}{2}-2} \mathcal{K}_{ij}^{(2)}.
\end{split}
\end{equation}
Denote by $\mathcal{H}(u)$ the Hessian matrix of $u$
and write $\mathcal{K}_{ij}^{(1)}$ as

\begin{equation}
\label{eq:K}
\begin{split}
\mathcal{K}_{ij}^{(1)} & = a_\epsilon \mathcal{H}^2_{ij}(u) + a_\epsilon\mathcal{H}_{ij}(u)\left[\sum_{k=1}^N((p-2)\eta_j\mathcal{H}_{ki}(u)+(r-2)\eta_i \mathcal{H}_{kj}(u))\eta_k\right]
\\
& + (p-2)(r-2) a_\epsilon \eta_i\eta_j \sum_{k,l=1}^N
\mathcal{H}_{li}(u)\mathcal{H}_{kj}(u)\eta_l\eta_k
\\
& + \frac{a_\epsilon}{2} \mathcal{H}_{ij}(u)\ln w_\epsilon D_iuD_jp
+ \frac{a_\epsilon}{2}(r-2)\eta_i\eta_j  \sqrt{w_\epsilon} \sum_{k=1}^N\mathcal{H}_{ki}(u)\eta_k \ln  w_\epsilon D_jp
\\
& + \sqrt{w_\epsilon}  \mathcal{H}_{ij}(u) \eta_j D_i a + (r-2)\eta_i \eta_j \sqrt{w_\epsilon} \sum_{k=1}^N \mathcal{H}_{kj}(u) \eta_k  D_i a
 \equiv \sum_{s=1}^{7}\mathcal{J}^{(p,s)}_{ij}.
\end{split}
\end{equation}
Summing up and regrouping we represent

\[
\begin{split}
\sum_{i,j=1}^N \mathcal{K}_{ij}^{(1)} &
= a_\epsilon \left(\operatorname{trace}\mathcal{H}^2(u) + (p+r-4)|(\mathcal{H}(u),\eta)|^2+(p-2)(r-2) \left((\mathcal{H}(u),\eta)\cdot \eta\right)^2 \right)
+\sum_{i,j=1}^N\sum_{s=4}^7\mathcal{J}_{ij}^{(p,s)}.
\end{split}
\]
Similarly, the following representation holds for $\mathcal{K}_{ij}^{(2)}$:
\[
\begin{split}
\sum_{i,j=1}^N \mathcal{K}_{ij}^{(2)} & =  b_\epsilon  \left(\operatorname{trace}\mathcal{H}^2(u) + (q+r-4)|(\mathcal{H}(u),\eta)|^2+(q-2)(r-2) \left((\mathcal{H}(u),\eta)\cdot \eta\right)^2 \right)
+\sum_{i,j=1}^N\sum_{s=4}^7\mathcal{J}_{ij}^{(q,s)},
\end{split}
\]
where the residual terms that depend on the derivatives of $p, q$ and $a, b$ have the form

\begin{equation}
\label{eq:residuals}
\begin{split}
& \mathcal{J}_{ij}^{(p,4)} =\frac{a_\epsilon}{2}\mathcal{H}_{ij}(u)\ln w_\epsilon D_iuD_jp,
\quad \mathcal{J}_{ij}^{(p,5)} = \frac{a_\epsilon}{2}(r-2)\sum_{k=1}^N\mathcal{H}_{kj}(u)\eta_k \eta_j\ln  w_\epsilon D_jp D_iu,\
\\
& \mathcal{J}_{ij}^{(p,6)} =   \mathcal{H}_{ij}(u) D_j u D_i a, \qquad \mathcal{J}_{ij}^{(p,7)} =  \frac{(r-2)}{w_\epsilon} D_i u D_j u \sum_{k=1}^N \mathcal{H}_{kj}(u) D_k u  D_i a,
\end{split}
\end{equation}
and
\begin{equation}
\label{eq:residuals-2}
\begin{split}
& \mathcal{J}_{ij}^{(q,4)} =\frac{b_\epsilon }{2}\mathcal{H}_{ij}(u)\ln w_\epsilon D_iu D_jq,
\quad \mathcal{J}_{ij}^{(q,5)} = \frac{b_\epsilon }{2}(r-2)\sum_{k=1}^N\mathcal{H}_{kj}(u)\eta_k \eta_j\ln  w_\epsilon D_j q D_iu,\
\\
& \mathcal{J}_{ij}^{(q,6)} =   \mathcal{H}_{ij}(u) D_j u D_i b, \qquad \mathcal{J}_{ij}^{(q,7)} =  \frac{(r-2)}{w_\epsilon} D_i u D_j u \sum_{k=1}^N \mathcal{H}_{kj}(u) D_k u  D_i b.
\end{split}
\end{equation}
Accept the notation
\[
\mathcal{G}_e(\eta)\equiv \operatorname{trace}\mathcal{H}^2(u) + (e+r-4)|(\mathcal{H}(u),\eta)|^2+(e-2)(r-2) \left((\mathcal{H}(u),\eta)\cdot \eta\right)^2, \quad e \in \{p,q\}.
\]

\begin{proposition}\cite[Proposition 3.2]{Ar-Sh-2025}
\label{pro:pointwise}
Let $u\in C^2(\Omega)$. If $p^-, q^->1$ and $r\geq 2$, then
\begin{equation}
\label{eq:est-from-below}
\mathcal{G}_e( \eta)\geq \sigma\operatorname{trace}\mathcal{H}^2(u)\equiv \sigma|u_{xx}|^2\qquad \forall\,\eta \in \mathbb{R}^N,\;\;|\eta|\leq 1, \quad e = \{p,q\},
\end{equation}
with the constant $\sigma=\min\{1,p^--1, q^--1\}$.
\end{proposition}

Substituting \eqref{eq:est-from-below} into \eqref{eq:product} we arrive at the following pointwise inequality.

\begin{lemma}
\label{le:pointwise-1}
If $u\in C^2(\Omega)$, $p^->1$, $q^- >1$ and $r\geq 2$, then

\begin{equation}
\label{eq:est-from-below-1}
\begin{split}
\sum_{i,j=1}^N & D_i\left(\mathcal{F}_\epsilon(x,\nabla u)D_ju\right)
D_{j}\left( w_\epsilon ^{\frac{r-2}{2}}D_{i}u\right)
\geq \sigma \left( a(x) w_\epsilon^{\frac{p+r}{2}-2} +  b(x) w_\epsilon^{\frac{q+r}{2}-2} \right) |u_{xx}|^2
\\
&
+ w_\epsilon ^{\frac{p+r}{2}-2} \sum_{i,j=1}^N \sum_{s=4}^7 \mathcal{J}_{ij}^{(p,s)}
+  w_\epsilon ^{\frac{q+r}{2}-2}\sum_{i,j=1}^N\sum_{s=4}^7 \mathcal{J}_{ij}^{(q,s)}
\end{split}
\end{equation}
with $\sigma=\min\{1,p^--1, q^--1\}$, and $\mathcal{J}_{ij}^{(p, s)}$, $\mathcal{J}_{ij}^{(q, s)}$ defined in \eqref{eq:residuals} and \eqref{eq:residuals-2}.
\end{lemma}

\subsection{Integral inequalities}

Combining \eqref{eq:est-from-below} with \eqref{eq:double-final-e} we obtain the following integral inequality.

\begin{lemma}
\label{le:principal-e} Let $\partial\Omega\in C^2$, $u\in C^3(\Omega)\cap
C^{2}(\overline{\Omega})$, $p, q \in C^{0,1}(\overline{\Omega})$. If $p^-, q^-> 1$ and $r\geq 2$, then

\begin{equation}
\label{eq:p-est-1}
\begin{split}
\int_{\Omega} & \operatorname{div}\left(a_\epsilon w_\epsilon^{\frac{p-2}{2}}\nabla u + b_\epsilon w_\epsilon^{\frac{q-2}{2}}\nabla u\right)
\operatorname{div}\left(w_\epsilon^{\frac{r-2}{2}}\nabla
u\right)\,dx\\
&  \geq - \int_{\partial\Omega} \left(a_\epsilon w_\epsilon^{\frac{p+r}{2}-2} + b_\epsilon w_\epsilon^{\frac{q+r}{2}-2} \right)|\nabla u|^2\operatorname{trace}\mathcal{B}\,dS
 \\
 & \quad + \sigma \int_{\Omega} \left( a_\epsilon w_\epsilon^{\frac{p+r}{2}-2} +  b_\epsilon w_\epsilon^{\frac{q+r}{2}-2} \right) |u_{xx}|^2 \,dx +
\int_{\Omega} w_\epsilon ^{\frac{p+r}{2}-2}\sum_{i,j=1}^N\sum_{s=4}^7 \mathcal{J}_{ij}^{(p,s)} ~dx \\
& \qquad + \int_{\Omega} w_\epsilon ^{\frac{q+r}{2}-2}\sum_{i,j=1}^N\sum_{s=4}^7 \mathcal{J}_{ij}^{(q,s)} ~dx
\end{split}
\end{equation}
where $\mathcal{B}$ is the second fundamental form of the surface $\partial\Omega$, $\sigma=\min\{1,p^--1, q^--1\}$ and $\mathcal{J}_{ij}^{(p,s)}$, $\mathcal{J}_{ij}^{(q,s)}$ are defined in \eqref{eq:residuals} and \eqref{eq:residuals-2}.
\end{lemma}

By the Cauchy-Schwarz inequality

\begin{equation}
\label{eq:F-1}
\begin{split}
w_\epsilon|u_{xx}|^2 & \geq |\nabla u|^2|u_{xx}|^2 =\left(\sum_{k=1}^N\left(D_ku\right)^2\right)\left(\sum_{i,j=1}^N\left(D^2_{ij}u\right)^2\right)
\geq \sum_{i=1}^N\left(\sum_{j=1}^N D^2_{ij}uD_ju\right)^2
\\
& =\frac{1}{4}\sum_{i=1}^N\left(D_i\left(\epsilon^2+\sum_{j=1}^N \left(D_j u\right)^2\right)\right)^2 \equiv \frac{1}{4}|\nabla w_\epsilon|^2,
\end{split}
\end{equation}
and

\begin{equation}
\label{eq:second-order-prim}
\begin{split}
\left|\nabla \left(w_\epsilon^{\frac{p+r-2}{4}}\right)\right|^2 & = \frac{(p+r-2)^2}{16}w_\epsilon^{\frac{p+r}{2}-3}|\nabla w_\epsilon|^2
\\
&
 \quad
+ \frac{p+r-2}{8}w_\epsilon^{\frac{p+r}{2}-2}\ln w_\epsilon (\nabla w_\epsilon,\nabla p)
 +\frac{1}{16}w_\epsilon^{\frac{p+r}{2}-1}\ln^2w_\epsilon|\nabla p|^2.
\end{split}
\end{equation}
Under the assumptions of Lemma \ref{le:principal-e} the second term on the right-hand side of \eqref{eq:p-est-1} has the lower bound: for the term with the power $p$

\begin{equation}
\label{eq:F-2}
\begin{split}
w_\epsilon^{\frac{p+r}{2}-2}|u_{xx}|^2
& \geq
 \dfrac{1}{4} w_\epsilon^{\frac{p+r}{2}-3}|\nabla w_\epsilon|^2
 \geq \frac{4}{(p+r-2)^2} \left|\nabla \left(w_\epsilon^{\frac{p+r-2}{4}}\right)\right|^2
 \\
 & \quad
- \frac{2L_p}{p+r-2} w_\epsilon^{\frac{p+r}{2}-2}|\ln w_\epsilon||\nabla w_\epsilon|
- \frac{4L_p^2}{(p+r-2)^2}  w_\epsilon^{\frac{p+r}{2}-1}\ln^2 w_\epsilon
\\
& \equiv \frac{4}{(p+r-2)^2} \left|\nabla \left(w_\epsilon^{\frac{p+r-2}{4}}\right)\right|^2 + w_\epsilon^{\frac{p+r}{2}-2}\left(\mathcal{M}_{1,p}+\mathcal{M}_{2,p}\right)
\end{split}
\end{equation}
and the analogous inequality holds for the term with the power $q$. Gathering these inequalities with \eqref{eq:p-est-1} we find that

\begin{equation}
\label{eq:p-est-1-trans}
\begin{split}
\int_{\Omega} & \operatorname{div}\left(\mathcal{F}_\epsilon(x,\nabla u)\nabla u\right) \operatorname{div}\left(w_\epsilon^{\frac{r-2}{2}}\nabla
u\right)\,dx
\geq - \int_{\partial\Omega} \left(a_\epsilon w_\epsilon^{\frac{p+r}{2}-2} + b_\epsilon w_\epsilon^{\frac{q+r}{2}-2} \right)|\nabla u|^2\operatorname{trace}\mathcal{B}\,dS
 \\
 & \quad +  \frac{4 \sigma}{(p^++r-2)^2} \int_{\Omega} a_\epsilon \left|\nabla \left(w_\epsilon^{\frac{p+r-2}{4}}\right)\right|^2 \,dx + \int_{\Omega} a_\epsilon w_\epsilon^{\frac{p+r}{2}-2}\left(\mathcal{M}_{1,p}+\mathcal{M}_{2,p}\right) \,dx \\
& \quad +  \frac{4 \sigma}{(q^++r-2)^2} \int_{\Omega} b_\epsilon \left|\nabla \left(w_\epsilon^{\frac{q+r-2}{4}}\right)\right|^2 \,dx + \int_{\Omega} b_\epsilon  w_\epsilon^{\frac{q+r}{2}-2}\left(\mathcal{M}_{1,q}+\mathcal{M}_{2,q}\right) \,dx\\
&
+ \int_{\Omega} w_\epsilon ^{\frac{p+r}{2}-2}\sum_{i,j=1}^N\sum_{s=4}^7 \mathcal{J}_{ij}^{(p,s)} ~dx + \int_{\Omega} w_\epsilon ^{\frac{q+r}{2}-2}\sum_{i,j=1}^N\sum_{s=4}^7 \mathcal{J}_{ij}^{(q,s)} ~dx,
\end{split}
\end{equation}
with

\begin{equation}
\label{eq:residuals-1}
\mathcal{M}_{1,e}=\frac{2L_e}{e+r-2} |\ln w_\epsilon||\nabla w_\epsilon|,\qquad \mathcal{M}_{2,e}=\frac{4L_e^2}{(e+r-2)^2}  w_\epsilon\ln^2 w_\epsilon, \quad e \in \{p,q\}.
\end{equation}

\subsection{Estimates on the residual terms}
We turn to estimating the terms $\mathcal{J}_{ij}^{(e,s)}$ and $\mathcal{M}_{\ell,e}$ for $s \in \{ 4,5,6,7\}$, $\ell \in \{1,2\}$ and $e \in \{p,q\}$.
Assume that $a$ and $b$ satisfy the conditions of Corollary \ref{cor:d-r}, and recall that
in \eqref{eq:p-est-1-trans} the residual terms $\mathcal{J}_{ij}^{(e,s)}$  are defined by formulas \eqref{eq:residuals} and \eqref{eq:residuals-2}, and the terms $\mathcal{M}_{\ell,e}$ are defined in \eqref{eq:residuals-1}.

Since $|\eta|<1$, $\mathcal{J}_{i,j}^{(e,s)}$ with $s =4,5$ satisfy the inequalities

\[
\begin{split}
& |\mathcal{J}_{ij}^{(p,4)} | +|\mathcal{J}_{ij}^{(p,5)} |  \leq C L_p a_\epsilon(x)|\ln w_\epsilon| |\nabla u||u_{xx}|,
\\
& |\mathcal{J}_{ij}^{(q,4)}| +|\mathcal{J}_{ij}^{(q,5)}|\leq C L_q b_\epsilon(x) |\ln w_\epsilon||\nabla u||u_{xx}|,
\\
& |\mathcal{J}_{ij}^{(p,6)}| + |\mathcal{J}_{ij}^{(p,7)}| \leq  C|\nabla a|(r-1)|\nabla u||u_{xx}|,
\\
& |\mathcal{J}_{ij}^{(q,6)}| + |\mathcal{J}_{ij}^{(q,7)}| \leq  C|\nabla b| (r-1)|\nabla u||u_{xx}|,
\qquad i,j=\overline{1,N}.
\end{split}
\]
It follows from \eqref{eq:log} and the Young inequality that for every $\delta>0$
\begin{equation}
\label{eq:est-J-p-s-a}
\begin{split}
\sum_{s=4}^5 \int_\Omega & w_\epsilon^{\frac{p+r}{2}-2}|\mathcal{J}^{(p,s)}_{ij}| \,dx
 \leq C \int_\Omega a_\epsilon w_\epsilon^{\frac{p+r-4}{2}}|\nabla u||u_{xx}||\ln w_\epsilon|\,dx
\\
& = C \int_{\Omega} \left(a_\epsilon^\frac{1}{2} w_\epsilon^{\frac{p+r-4}{4}}|u_{xx}|\right) \left( a_\epsilon^\frac{1}{2} w_\epsilon^{\frac{p+r-4}{4}}|\nabla u||\ln w_\epsilon|\right)\,dx
\\
&
\leq \delta \int_\Omega a_\epsilon w_\epsilon^{\frac{p+r-4}{2}}|u_{xx}|^2 + C_1 \int_\Omega a_\epsilon w_\epsilon^{\frac{p+r+2\lambda-4}{2}}|\nabla u|^2\,dx+C'
\end{split}
\end{equation}
with an independent of $r$ constant $C_1$ and $C'=C'(r,\delta,\lambda)$. By virtue of \eqref{eq:d-r-ell}, the second term on the right-hand side of \eqref{eq:est-J-p-s-a} is bounded for every $0<\lambda<\dfrac{2}{N+2}$. Gathering \eqref{eq:est-J-p-s-a} with the analogous estimate for $w_\epsilon^{\frac{q+r}{2}-2}|\mathcal{J}^{(q,s)}_{ij}|$, $s=4,5$,  we find that for every $\delta'>0$
\begin{equation}
\label{eq:residuals-4-5}
\begin{split}
\sum_{s=4}^5 \int_\Omega & \left(w_\epsilon^{\frac{p+r}{2}-2}|\mathcal{J}^{(p,s)}_{ij}| +w_\epsilon^{\frac{q+r}{2}-2}|\mathcal{J}^{(q,s)}_{ij}|\right) \,dx
\leq \delta'\int_\Omega \left( a_\epsilon w_\epsilon^{\frac{p+r}{2}-2} + b_\epsilon w_\epsilon^{\frac{q+r}{2}-2} \right)|u_{xx}|^2\,dx + C''
\end{split}
\end{equation}
with a constant $C''$ depending on $r$ and $\delta'$.

To estimate $\mathcal{J}^{(\rho,s)}_{ij}$ with $\rho\in \{p,q\}$ and $s=6,7$, we use the inequality $\sqrt{\alpha}\leq \sqrt{a+b}\leq \sqrt{a_\epsilon}+\sqrt{b_\epsilon}$:

\begin{equation}
\label{eq:est-J-p-s-a-1}
\begin{split}
& \sum_{s=6}^7 \int_\Omega w_\epsilon^{\frac{p+r}{2}-2}|\mathcal{J}^{(p,s)}_{ij}| \,dx \leq \frac{C}{\sqrt{\alpha}} \int_\Omega (\sqrt{a_\epsilon}+\sqrt{b_\epsilon})|\nabla a| w_\epsilon^{\frac{p+r}{2}-2}|\nabla u||u_{xx}|\,dx
\\
& = \frac{C}{\sqrt{\alpha}}\int_\Omega \sqrt{a_\epsilon}|\nabla a| w_\epsilon^{\frac{p+r-4}{2}}|\nabla u||u_{xx}|\,dx+\frac{C}{\sqrt{\alpha}}\int_\Omega \sqrt{b_\epsilon}|\nabla a| w_\epsilon^{\frac{p+r-4}{2}}|\nabla u||u_{xx}|\,dx
\equiv C'\left(\mathcal{I}_{p,a}+\mathcal{I}_{p,b}\right).
\end{split}
\end{equation}
By using Young's inequality and noting that $|\nabla u|^2\leq w_{\epsilon}$ we find that for every $\delta>0$
\[
\begin{split}
\mathcal{I}_{p,a} & =\int_{\Omega}\left(|\nabla a|w_{\epsilon}^{\frac{p+r-4}{4}}|\nabla u|\right)\left(a_\epsilon  w_{\epsilon}^{\frac{p+r-4}{2}}|u_{xx}|^2\right)^{\frac{1}{2}}\,dx
\leq \delta \int_{\Omega}a_\epsilon w_{\epsilon}^{\frac{p+r-4}{2}}|u_{xx}|^2\,dx + C_\delta \int_\Omega |\nabla a|^2w_{\epsilon}^{\frac{p+r-2}{2}}\,dx. 
\end{split}
\]
By the Young inequality with the exponents $\frac{d}{2}$ and $\frac{d}{d-2}$ we estimate the second term by

\begin{equation}
\label{eq:2-est}
\int_{\Omega}w_{\epsilon}^{\frac{p+r-2}{2}\frac{d}{d-2}}\,dx + C'\int_\Omega |\nabla a|^{d}\,dx.
\end{equation}
Assume that $d$ satisfies condition \eqref{eq:d-1}. The second term of \eqref{eq:2-est} is bounded by assumption. The estimate on the first one will follow from \eqref{eq:d-r-ell} if we claim that everywhere in $\Omega$

\[
\frac{d}{d-2}(p(x)+r-2)< \underline{s}(x)+r +r^\sharp\quad \Leftrightarrow \quad p(x)+r-2+\dfrac{2}{d-2}(p(x)+r-2) < \underline{s}(x)+r + r^\sharp.
\]
Since $p(x)\leq \overline{s}(x)$, the last inequality is surely fulfilled if $\overline{s}(x)-\underline{s}(x) + \dfrac{2}{d-2}(\overline{s}(x)+r-2)<2 + r^\sharp$. By \eqref{eq:structure-prelim-2} and \eqref{eq:threshold}, the last inequality holds if

\[
\dfrac{2}{d-2}(\overline{s}(x)+r-2)<2 + \frac{2}{N+2}=2\dfrac{N+3}{N+2}\quad \Leftrightarrow \quad \frac{N+2}{N+3}(\overline{s}(x)+r-2)+2<d.
\]
This is true for $d$ satisfying \eqref{eq:d-1}.

To estimate $\mathcal{I}_{p,b}$ we make use of assumption \eqref{eq:structure-prelim-2} on the gap between $p$ and $q$:

\[
\begin{split}
\mathcal{I}_{p,b} & =\int_{\Omega}\left(|\nabla a |w_{\epsilon}^{\frac{2p-q+r-4}{4}}|\nabla u|\right)\left(b_\epsilon w_{\epsilon}^{\frac{q+r-4}{2}}|u_{xx}|^2\right)^{\frac{1}{2}}\,dx
\leq \delta \int_{\Omega}b_\epsilon w_{\epsilon}^{\frac{q+r-4}{2}}|u_{xx}|^2\,dx + C_\delta \int_\Omega |\nabla a|^2w_{\epsilon}^{\frac{2p-q+r-2}{2}}\,dx
\\
& \leq  \delta \int_{\Omega}b_\epsilon w_{\epsilon}^{\frac{q+r-4}{2}}|u_{xx}|^2\,dx + C'\int_{\Omega} w_{\epsilon}^\gamma\,dx +C'' \int_{\Omega}|\nabla a|^d\,dx,\qquad \gamma=\frac{2p-q+r-2}{2}\frac{d}{d-2},
\end{split}
\]
where

\[
\gamma\leq \frac{1}{2}(p+\overline{s}-\underline{s}+r-2)\frac{d}{d-2}< \frac{1}{2}(p+r-2+r_\ast)\frac{d}{d-2}.
\]
By virtue of \eqref{eq:d-r-ell}, the integral of $w_\epsilon^\gamma$ is bounded if we claim

\[
(p+r-2+r_\ast)\frac{d}{d-2}=p+r-2+r_\ast +\dfrac{2}{d-2}(p+r-2+r_\ast)<p+r+r^\sharp\quad \Leftrightarrow \quad \dfrac{1}{d-2}(p+r-2+r_\ast)<\dfrac{N+3}{N+2}.
\]
Since $p\leq \overline{s}$, the last inequality holds if
\[
\dfrac{N+2}{N+3}(\overline{s}+r)+\dfrac{4}{N+3}<d.
\]
which is true due to \eqref{eq:d-1}. The integrals $\mathcal{I}_{q,a}$, $\mathcal{I}_{q,b}$ are estimated in the same way. Gathering the results, we find that if $d$ satisfies \eqref{eq:d-1}, then for every $\delta>0$
\begin{equation}
\label{eq:I-p-q}
\sum_{s=6}^7 \int_\Omega \left(w_\epsilon^{\frac{p+r}{2}-2}|\mathcal{J}^{(p,s)}_{ij}|+ w_\epsilon^{\frac{q+r}{2}-2}|\mathcal{I}^{(q,s)}_{ij}|\right) \,dx\leq  \delta \int_\Omega \left(a_\epsilon w_{\epsilon}^{\frac{p+r}{2}-2}+b_\epsilon w_{\epsilon}^{\frac{q+r}{2}-2}\right)|u_{xx}|^2\,dx + C
\end{equation}
with a constant $C$ depending on $N$, $\max a$, $\max b$, $\|\nabla a\|_{d,\Omega}$, $\|\nabla b\|_{d,\Omega}$, $r$, $\delta$. The terms $\mathcal{M}_{\ell,e}$ are bounded by

\[
\begin{split}
|\mathcal{M}_{1,e}| & \leq C |\ln w_\epsilon| |\nabla u||u_{xx}|,
\qquad
|\mathcal{M}_{2,e}|
\leq C' w_\epsilon \ln^2 w_\epsilon, \quad e \in \{p,q\}
\end{split}
\]
with constants $C$, $C'$ depending on $N$, $p^\pm$, $q^\pm$, $r$ and $L_{p,q}$. The estimate on the integrals involving $\mathcal{M}_{\ell, e}$ follows from \eqref{eq:log} and \eqref{eq:complete-new}: for every $\delta>0$

\begin{equation}\label{eq:est-M}
    \begin{split}
        \sum_{\ell =1}^2 \int_{\Omega} & \left(a_\epsilon w_\epsilon^{\frac{p+r}{2}-2} \mathcal{M}_{\ell,p} + b_\epsilon w_\epsilon^{\frac{q+r}{2}-2} \mathcal{M}_{\ell,q}\right) \,dx
        \leq  \delta \int_\Omega \left( a_\epsilon w_\epsilon^{\frac{p+r}{2}-2} + b_\epsilon w_\epsilon^{\frac{q+r}{2}-2} \right)|u_{xx}|^2\,dx + C.
    \end{split}
\end{equation}

\subsection{The boundary integrals}

\begin{lemma}\label{th:trace-main} Let $\partial\Omega\in C^2$, $u\in C^1(\overline{\Omega})\cap H^1_0(\Omega)\cap H^2(\Omega)$.  Assume that $a(x)$, $b(x)$, $p(x)$, $q(x)$ satisfy the conditions of Corollary \ref{cor:d-r}. If

\begin{equation}
\label{eq:d-boundary}
|\nabla a|, |\nabla b|\in L^d(\Omega)\quad \text{with $\displaystyle d>2+\frac{N+2}{2}\left(\overline{s}^++r\right)$ and $\underline{s}^-+r\geq 2$},
\end{equation}
then for every $\lambda\in (0,1)$
\begin{equation}
\label{eq:trace-3} \int_{\partial \Omega} \mathcal{F}^{(r,r)}_{\epsilon}(x,\nabla u)\vert \nabla u\vert ^{2}\,dS\leq
\lambda \int_{\Omega} \mathcal{F}^{(r,r)}_{\epsilon}(x,\nabla u) \vert u_{xx}\vert ^2\,dx+ C
\end{equation}
with a constant
$C=C'(\lambda,
\underline{s}^-, \underline{s}^+, N, L_{p,q}, \alpha, r, \|u\|_{2,\Omega}) + C''\left( \|\nabla a\|^d_{d,\Omega}, \|\nabla b\|^d_{d,\Omega}\right)
$.
\end{lemma}

\begin{proof}
By \cite[Lemma 1.5.1.9]{Grisvard-2011} there exist a constant $\gamma>0$ and a function $\vec \mu\in C^{\infty}(\overline{\Omega})^N$ such that $\vec \mu\cdot\nu \geq \gamma>0$ on $\partial\Omega$. Then

\[
\begin{split}
& \gamma \int_{\partial\Omega} \left(a_\epsilon w_\epsilon^{\frac{p+r}{2}-2} + b_\epsilon w_\epsilon^{\frac{q+r}{2}-2} \right)|\nabla u|^2\,dS  \leq \int_{\Omega}\operatorname{div}\left( \left(a_\epsilon w_\epsilon^{\frac{p+r}{2}-2} + b_\epsilon w_\epsilon^{\frac{q+r}{2}-2}\right) |\nabla u|^2\vec \mu\right)\,dx
 \\
 &
 = \int_{\Omega} \vec \mu \cdot\nabla \left(\left(a_\epsilon w_\epsilon^{\frac{p+r}{2}-2} + b_\epsilon w_\epsilon^{\frac{q+r}{2}-2}\right)|\nabla u|^2\right)\,dx
 + \int_{\Omega}(\operatorname{div}\vec \mu ) \left(a_\epsilon w_\epsilon^{\frac{p+r}{2}-2} + b_\epsilon w_\epsilon^{\frac{q+r}{2}-2}\right)|\nabla u|^2\,dx
\\
& \leq C_1 \int_{\Omega} \left(a_\epsilon w_\epsilon^{\frac{p+r}{2}-2} + b_\epsilon w_\epsilon^{\frac{q+r}{2}-2}\right)|\nabla u|^2\,dx\\
& \quad + C_2 \left(\max\{p^+, q^+\} + r-1\right) \int_{\Omega} \left(a_\epsilon w_\epsilon^{\frac{p+r}{2}-3}|\nabla u|^2 +  b_\epsilon w_\epsilon^{\frac{q+r}{2}-3}|\nabla u|^2 \right) |\nabla u||u_{xx}| \,dx
\\
&
+ C_3 \int_{\Omega} \left(a_\epsilon w_\epsilon^{\frac{p+r}{2}-2} |\nabla p| + b_\epsilon w_\epsilon^{\frac{q+r}{2}-2} |\nabla q|\right) |\nabla u|^2 |\ln w_\epsilon|\,dx
+ C_4\int_{\Omega} \left(w_\epsilon^{\frac{p+r}{2}-2} |\nabla a|+ w_\epsilon^{\frac{q+r}{2}-2} |\nabla b|\right)|\nabla u|^2\,dx
\\
& \equiv C_1\mathcal{I}_1 +C_2\mathcal{I}_2 + C_3\mathcal{I}_3 + C_4 \mathcal{I}_4.
\end{split}
\]
To estimate $\mathcal{I}_2$, we set $C_{p,q} = \max\{p^+, q^+\} + r-1$ and apply the Cauchy inequality:

\[
\begin{split}
\mathcal{I}_2 & \leq C_{p,q} \int_\Omega a_\epsilon w_\epsilon^{\frac{p+r}{2}-2}|\nabla u||u_{xx}|\,dx + C_{p,q} \int_\Omega b_\epsilon w_\epsilon^{\frac{q+r}{2}-2}|\nabla u||u_{xx}|\,dx
\\
&
= C_{p,q} \int_{\Omega} \left[ a_\epsilon^\frac{1}{2} w_\epsilon ^{\frac{p+r}{4}-1}|u_{xx}|\right] \left[ a_\epsilon^\frac{1}{2} w_\epsilon^{\frac{p+r}{4}-1}|\nabla u|\right]\,dx
+ C_{p,q} \int_{\Omega} \left[ b_\epsilon^\frac{1}{2}  w_\epsilon ^{\frac{q+r}{4}-1}|u_{xx}|\right] \left[b_\epsilon^\frac{1}{2}  w_\epsilon^{\frac{q+r}{4}-1}|\nabla u|\right]\,dx
\\
& \leq \sigma  \int_{\Omega} \left( a_\epsilon w_\epsilon^{\frac{p+r}{2}-2} + b_\epsilon w_\epsilon^{\frac{q+r}{2}-2} \right)|u_{xx}|^2\,dx
+ C_\sigma C_{p,q}^2 \int_\Omega \left( a_\epsilon w_\epsilon^{\frac{p+r}{2}-2} + b_\epsilon w_\epsilon^{\frac{q+r}{2}-2} \right)|\nabla u|^2\,dx
\\
& \leq \sigma  \int_{\Omega} \left( a_\epsilon w_\epsilon^{\frac{p+r}{2}-2} + b_\epsilon w_\epsilon^{\frac{q+r}{2}-2} \right) |u_{xx}|^2\,dx + C' C_{p,q}^2 \mathcal{I}_1
\end{split}
\]
with an arbitrary $\sigma>0$ and an independent of $r$ constant $C'$. The estimate on $\mathcal{I}_1$ follows directly from \eqref{eq:d-r-ell}. To estimate $\mathcal{I}_3$ we use \eqref{eq:log}: for every $\theta>0$ there is a constant $C_\theta$ such that
\[
\left(w_\epsilon^{\frac{p+r}{2}-1} + w_\epsilon^{\frac{q+r}{2}-1}\right)\vert \ln w_\epsilon \vert
\leq C_\theta\left(1+ \vert \nabla u\vert ^{\underline{s}(x)+ r + \theta}\right).
\]
The integral of this function is bounded by \eqref{eq:d-r-ell} if $\theta\in (0,r^\sharp)$. The estimate on $\mathcal{I}_4$ also follows from \eqref{eq:d-r-ell}. Since $\underline{s}^-+r\geq 2$,  it follows from Young's inequality
\[
\begin{split}
|\nabla a| w_{\epsilon}^{\frac{p+r}{2}-2}|\nabla u|^{2} & + |\nabla b|w_\epsilon^{\frac{q+r}{2}-2} |\nabla u|^{2}
\leq |\nabla a| w_{\epsilon}^{\frac{p+r}{2}-1} + |\nabla b|w_\epsilon^{\frac{q+r}{2}-1}
\\
&
\leq |\nabla a|^d + |\nabla b|^d + w_\epsilon^{\frac{d(p+r-2)}{2(d-1)}} + w_\epsilon^{\frac{d(q+r-2)}{2(d-1)}}
\leq  |\nabla a|^d + |\nabla b|^d + C \left( 1+ w_{\epsilon}^{\frac{d}{2(d-1)}(\overline{s}(x)+r-2)}\right)
\end{split}
\]
with a constant $C$ independent of $u$. By \eqref{eq:d-r-ell}, the integral of the last term is bounded if

\[
\dfrac{d}{d-1} (\overline{s}(x)+r-2) =\overline{s}(x)+r-2+\frac{1}{d-1}(\overline{s}(x)+r-2) < \underline{s}(x) + r^\sharp +r.
\]
Since $0\leq \overline{s}(x)-\underline{s}(x)<\dfrac{2}{N+2}$ by assumption  \eqref{eq:structure-prelim-2}, the previous inequality is fulfilled if

\[
\dfrac{2(d-1)(N+3)}{N+2}>\overline{s}^++r-2\qquad \Leftrightarrow \qquad d>\dfrac{1}{N+3}+\dfrac{N+2}{2(N+3)}(\overline{s}^++r),
\]
which is true due to \eqref{eq:d-boundary}. Gathering these inequalities we obtain \eqref{eq:trace-3}.
\end{proof}

We may now refine Lemma \ref{le:principal-e}.

\begin{lemma}
\label{le:principal-improved}
Let in the conditions of Lemma \ref{le:principal-e}, $p,q\in C^{0,1}(\overline{\Omega})$, and $a$, $b$ satisfy \eqref{eq:d-boundary} with $r \geq 2$.
There exist independent of $\epsilon$ positive constants $C_1,C_1'$ and

\[
C_2=C'_2\left(\alpha, N, r, p^+, q^+, L_{p,q},\|u\|_2,\Omega\right) +C_2''\left(\|\nabla a\|_{d,\Omega}+\|\nabla b\|_{d,\Omega}\right)
\]
such that

\begin{equation}
\label{eq:p-est-1-prim}
\begin{split}
C_1\int_{\Omega} & \left( a_\epsilon w_\epsilon^{\frac{p+r}{2}-2} + b_\epsilon w_\epsilon^{\frac{q+r}{2}-2} \right) |u_{xx}|^2\,dx\\
& \leq \int_{\Omega}\operatorname{div}\left(\left(a_\epsilon w_\epsilon^{\frac{p-2}{2}} + b_\epsilon w_\epsilon^{\frac{q-2}{2}} \right)\nabla u\right)\operatorname{div}\left(w_\epsilon^{\frac{r-2}{2}}\nabla u\right)\,dx +C_2,
\end{split}
\end{equation}

\begin{equation}
\label{eq:p-est-2-prim}
\begin{split}
C'_1\int_{\Omega}  &  \left(a_\epsilon \left|\nabla \left(w_\epsilon^{\frac{p+r-2}{4}}\right)\right|^2 + b_\epsilon w \left|\nabla \left(w_\epsilon^{\frac{q+r-2}{4}}\right)\right|^2 \right)\,dx\\
& \leq \int_{\Omega}\operatorname{div}\left(\left(a_\epsilon w_\epsilon^{\frac{p-2}{2}} + b_\epsilon w_\epsilon^{\frac{q-2}{2}} \right)\nabla u\right)\operatorname{div}\left(w_\epsilon^{\frac{r-2}{2}}\nabla u\right)\,dx +C_2.
\end{split}
\end{equation}
\end{lemma}

\begin{proof}
Inequalities \eqref{eq:p-est-1-prim} and \eqref{eq:p-est-2-prim} follows from \eqref{eq:p-est-1} by using estimates \eqref{eq:p-est-1-trans}, the estimates on $\mathcal{J}_{ij}^{(e,s)}$ in \eqref{eq:I-p-q}, the estimates on $\mathcal{M}_{\ell,e}$ in \eqref{eq:est-M}, the boundary integral estimate \eqref{eq:trace-3}, and inequality \eqref{eq:F-2}.
\end{proof}

\begin{corollary}
\label{cor:p-prim} Let the functions $u(z)$, $p(z)$, $q(z)$ and the constant $r$ satisfy the conditions of Lemma \ref{le:principal-improved} for a.e. $t\in (0,T)$. Then

\begin{equation}
\label{eq:est-cyl}
\begin{split}
c_1\int_{Q_T} & \left(a_\epsilon(z) w_\epsilon^{\frac{p(z)+r}{2}-2} + b_\epsilon(z) w_\epsilon^{\frac{q(z)+r}{2}-2} \right)|u_{xx}|^2\,dz
\\
&
\leq \int_{Q_T}\operatorname{div}\left(\left( a(z) w_\epsilon^{\frac{p(z)-2}{2}} + b(z) w_\epsilon^{\frac{q(z)-2}{2}} \right)\nabla u\right)\operatorname{div}\left(w_\epsilon^{\frac{r-2}{2}}\nabla u\right)\,dz +c_2,
\end{split}
\end{equation}

\begin{equation}
\label{eq:est-cyl-1}
\begin{split}
c'_1\int_{Q_T} & \left( a_\epsilon(z) \left|\nabla \left(w_\epsilon^{\frac{p(z)+r-2}{4}}\right)\right|^2 + b_\epsilon(z) \left|\nabla \left(w_\epsilon^{\frac{q(z)+r-2}{4}}\right)\right|^2 \right)\,dz
\\
&
\leq \int_{Q_T}\operatorname{div}\left( \left(a_\epsilon(z) w_\epsilon^{\frac{p(z)-2}{2}} + b_\epsilon(z) w_\epsilon^{\frac{q(z)-2}{2}} \right)\nabla u\right)\operatorname{div}\left(w_\epsilon^{\frac{r-2}{2}}\nabla u\right)\,dz +c'_2
\end{split}
\end{equation}
with the constants $c_1,c_1',c_2, c_2'$ depending on the same quantities as the constants in \eqref{eq:p-est-1-prim}, $\operatorname{ess}\sup\limits_{(0,T)}\|u\|_{2,\Omega}^2$, $T$, and $\|\nabla a\|_{d,Q_T}$, $\|\nabla b\|_{d,Q_T}$.
\end{corollary}

\section{Uniform estimates on the gradient of a classical solution}
\label{sec:est-source}
In this section, we establish the global uniform estimates on the gradients of classical solutions to the regularized problem \eqref{eq:main-reg} with smooth data and $\epsilon\in (0,1)$. For the sake of convenience of presentation, we will say that the constants appearing in computations depend on \textbf{data} and they can be evaluated through $r$, $N$, $p^\pm$, $q^\pm$, $L_{p,q}$, $\|\nabla a\|_{d,Q_T}$, $\|\nabla b\|_{d,Q_T}$, $\|u_0\|_{2,\Omega}$, $\|f\|_{\sigma,Q_T}$, and the properties of $\partial\Omega$.

\subsection{Case I: $\sigma= N+2$} Let $u$ be the classical solution of problem \eqref{eq:main-reg} with smooth data.
Multiply equation \eqref{eq:main-reg} by $-\div\left(w_\epsilon^{\frac{r-2}{2}} \nabla u\right)$ with $r \geq  \max\{2,p^+,q^+\}$, $w_\epsilon= \epsilon^2+|\nabla u|^2$, and integrate the result over $\Omega$:

\[
\begin{split}
\frac{1}{r}\dfrac{d}{dt} & \int_{\Omega}w_\epsilon^{\frac{r}{2}}\,dx + \int_{\Omega}\operatorname{div}\left(\left( a_\epsilon (z) w_\epsilon^{\frac{p(z)-2}{2}} + b_\epsilon (z) w_\epsilon^{\frac{q(z)-2}{2}} \right)\nabla u\right)\operatorname{div}\left(w_\epsilon^{\frac{r-2}{2}}\nabla u\right)\,dx\\
& = -\int_{\Omega}f \operatorname{div}\left(w_\epsilon^{\frac{r-2}{2}}\nabla u\right)\,dx.
\end{split}
\]
By \eqref{eq:p-est-1} and \eqref{eq:p-est-1-prim}, \eqref{eq:p-est-2-prim}

\begin{equation}
\label{eq:a-priori-1}
\begin{split}
\frac{1}{r}\dfrac{d}{dt} \int_{\Omega}w_\epsilon^{\frac{r}{2}}\,dx
& + C_0 \int_{\Omega} \left(a_\epsilon (z) w_\epsilon^{\frac{p(z)+r}{2}-2} + b_\epsilon (z) w_\epsilon^{\frac{q(z)+r}{2}-2} \right)|u_{xx}|^2\,dx
 \\
& + C_1\int_\Omega \left(a_\epsilon (z) \left|\nabla \left(w_\epsilon^{\frac{p+r-2}{4}}\right)\right|^2 + b_\epsilon (z) \left|\nabla \left(w_\epsilon^{\frac{q+r-2}{4}}\right)\right|^2 \right) \,dx
\\
&\qquad
\leq \int_\Omega |f|\left|\div\left(w_\epsilon^{\frac{r-2}{2}} \nabla u\right)\right|\,dx +C \equiv\mathcal{J} +C
\end{split}
\end{equation}
with constants $C$, $C_0$, $C_1$ depending on $r$, $p^\pm$, $q^\pm$, $N$, $\partial \Omega$, $\|u(t)\|_{2,\Omega}$. By Young's inequality and assumption \eqref{eq:structure-prelim-1}

\[
\begin{split}
\mathcal{J} & = \int_{\Omega}|f|\left |w_\epsilon^{\frac{r-2}{2}} \Delta u + (r-2) w_\epsilon^{\frac{r-2}{2} -1} \sum_{i,j =1}^N D_iuD_juD_{ij}^2u\right|\,dx
\\
    & \leq C \int_{\Omega} |f| w_\epsilon^{\frac{r-2}{2}} |u_{xx}| ~dx + (r-2) \int_{\Omega} |f| w_\epsilon^{\frac{r-2}{2}-1} |\nabla u|^2 |u_{xx}| ~dx
    \\
    & \leq \frac{C (r-1)}{\alpha} \int_{\Omega} (a(z) + b(z))|f| w_\epsilon^{\frac{r-2}{2}} |u_{xx}|  ~dx
    \\
    & \leq \frac{C (r-1)}{\alpha} \int_{\Omega} \left(\sqrt{a(z)}|f| w_\epsilon^{\frac{r-2}{2} - \frac{r+p}{4}+1} \right) \left(\sqrt{a(z)} w_\epsilon^{\frac{r+p}{4}-1} |u_{xx}| \right) ~dx \\
    & \qquad + \frac{C (r-1)}{\alpha} \int_{\Omega} \left( \sqrt{b(z)} |f| w_\epsilon^{\frac{r-2}{2} - \frac{r+q}{4}+1} \right) \left( \sqrt{b(z)} w_\epsilon^{\frac{r+q}{4}-1} |u_{xx}| \right) ~dx
    \\
    & \leq \delta \int_{\Omega} \left( a_\epsilon (z) w_\epsilon^{\frac{r+p}{2}-2} + b_\epsilon (z) w_\epsilon^{\frac{r+q}{2}-2}\right) |u_{xx}|^2~dx + C_\delta r^2 \int_{\Omega} f^2  \left( a_\epsilon (z)w_\epsilon^{\frac{r-p}{2}} + b_\epsilon (z)w_\epsilon^{\frac{r-q}{2}} \right) ~dx
\end{split}
\]
with any $\delta>0$. For the sufficiently small $\delta$ the first term on the right-hand side is absorbed in the left-hand side, and \eqref{eq:a-priori-1} is continued as

\begin{equation}
\label{eq:a-priori-5}
\begin{split}
\frac{1}{r}\dfrac{d}{dt} \int_{\Omega}w_\epsilon^{\frac{r}{2}}\,dx
& + C_0 \int_{\Omega} \left(a_\epsilon (z) w_\epsilon^{\frac{p+r}{2}-2} + b_\epsilon (z) w_\epsilon^{\frac{q+r}{2}-2} \right)|u_{xx}|^2\,dx
 \\
& + C_1\int_\Omega \left(a_\epsilon (z) \left|\nabla \left(w_\epsilon^{\frac{p+r-2}{4}}\right)\right|^2 + b_\epsilon (z) \left|\nabla \left(w_\epsilon^{\frac{q+r-2}{4}}\right)\right|^2 \right) \,dx
\\
& \leq C_2 +C_3r^2\int_{\Omega}f^2 \left( a_\epsilon (z)w_\epsilon^{\frac{r-p}{2}} + b_\epsilon (z)w_\epsilon^{\frac{r-q}{2}} \right)  \,dx
\end{split}
\end{equation}
with known constants $C_i$ depending on \textbf{data}, and $C_2$ depending also on $\|u(t)\|_{2,\Omega}$.
Integrating \eqref{eq:a-priori-5} in $t$ we obtain

\begin{equation}
\label{eq:ODI-new-1}
\begin{split}
\frac{1}{r}\sup_{(0,T)} \int_{\Omega}w_\epsilon^{\frac{r}{2}}\,dx  & + C_0 \int_{Q_T} \left(a_\epsilon (z) w_\epsilon^{\frac{p+r}{2}-2} + b_\epsilon (z) w_\epsilon^{\frac{q+r}{2}-2} \right)|u_{xx}|^2\,dx
 \\
& + C_1\int_{Q_T} \left(a_\epsilon (z) \left|\nabla \left(w_\epsilon^{\frac{p+r-2}{4}}\right)\right|^2 + b_\epsilon (z) \left|\nabla \left(w_\epsilon^{\frac{q+r-2}{4}}\right)\right|^2 \right) \,dx
\\
&
\leq C_2T+ \frac{1}{r}\int_{\Omega}w_{\epsilon}^{\frac{r}{2}}(x,0)\,dx + C_3r^2\int_{Q_T}f^2 \left(a_\epsilon (z)w_\epsilon^{\frac{r-p}{2}} + b_\epsilon (z) w_\epsilon^{\frac{r-q}{2}} \right) \,dz.
\end{split}
\end{equation}
It remains to estimate

\begin{equation}
\label{eq:integral-I}
\mathcal{I}=\int_{Q_T}f^2 \left( a_\epsilon (z)w_\epsilon^{\frac{r-p}{2}} + b_\epsilon (z)w_\epsilon^{\frac{r-q}{2}} \right)  \,dz.
\end{equation}
By H\"older inequality, we have

\[
\mathcal{I}\leq \|f\|_{N+2,Q_T}^2 \left(\int_{Q_T} \left( a_\epsilon (z)w_\epsilon^{\frac{r-p}{2}} + b_\epsilon (z)w_\epsilon^{\frac{r-q}{2}} \right)^{\frac{N+2}{N}}\,dz\right)^{\frac{N}{N+2}}.
\]
Observe that, since $ \min\{p(z), q(z)\}>\frac{2N}{N+2}$, $r \geq \max\{2, p^+, q^+\}$ and $a, b \in L^\infty(Q_T)$, then

\[
\frac{(r-\ell)(N+2)}{2N}<r\dfrac{N+2}{2N}-1 =r\left(\frac{1}{2}+\frac{1}{N}\right)-1<\frac{\ell+r-2+\frac{2r}{N}}{2}, \quad \ell \in \{p(z), q(z)\},
\]
and by the Young inequality, for every $\mu>0$
\[
\begin{split}
    \left( a_\epsilon (z)w_\epsilon^{\frac{r-p}{2}} + b_\epsilon (z)w_\epsilon^{\frac{r-q}{2}} \right)^{\frac{N+2}{N}} \leq C(\mu) + \mu \left(a_\epsilon (z)w_\epsilon^{\frac{p+r-2+\frac{2r}{N}}{2}} + b_\epsilon (z)w_\epsilon^{\frac{q+r-2+\frac{2r}{N}}{2}}\right)
\end{split}
\]
where $C$ depends upon $\|a\|_{L^\infty}$, $\|b\|_{L^\infty}$, $\mu$ and $N$ but is independent of $\epsilon$. Using the above inequality, we have for every $\mu>0$

\begin{equation}
  \label{eq:I-1}
  \mathcal{I}\leq \|f\|_{N+2,Q_T}^2 \left[C+ \mu \int_{Q_T}\left(a_\epsilon (z)w_\epsilon^{\frac{p+r-2+\frac{2r}{N}}{2}} + b_\epsilon (z)w_\epsilon^{\frac{q+r-2+\frac{2r}{N}}{2}}\right) \,dz\right]^{\frac{N}{N+2}},\qquad C=C(\mu).
  \end{equation}
By the Sobolev embedding $W^{1,\frac{2N}{N+2}}(\Omega)
\subset L^2(\Omega)$. Hence for every $v\in W^{1,\frac{2N}{N+2}}(\Omega)$

\begin{equation}
\label{eq:Sob-1}
\|v\|_{2,\Omega}^2\leq C_{s}\left(\|\nabla v\|_{\frac{2N}{N+2},\Omega}^2 + \|v\|^2_{\frac{2N}{N+2},\Omega}\right)
\end{equation}
with an independent of $v$ constant $C_s$. Denote
\[
v_\epsilon(z) := a_\epsilon (z)w_\epsilon^{\frac{p+r-2+\frac{2r}{N}}{2}} + b_\epsilon (z)w_\epsilon^{\frac{q+r-2+\frac{2r}{N}}{2}}.
\]
Recall notation \eqref{eq:s} and notice that

\begin{equation}\label{lowerest-1}
    v_\epsilon(z) \geq  \begin{cases}
(a(z)+b(z)) w_\epsilon^{\frac{\overline{s}+r-2+\frac{2r}{N}}{2}} \geq \alpha  w_\epsilon^{\frac{\overline{s}+r-2+\frac{2r}{N}}{2}} & \text{if $|w_\epsilon| < 1$},
\\
(a(z)+b(z)) w_\epsilon^{\frac{\underline{s}+r-2+\frac{2r}{N}}{2}} \geq \alpha  w_\epsilon^{\frac{\underline{s}+r-2+\frac{2r}{N}}{2}} & \text{if $|w_\epsilon|\geq 1$}.
\end{cases}
\end{equation}
Applying \eqref{eq:Sob-1} to $v_\epsilon^\frac{1}{2}$ we arrive at the inequality
\[
\int_{\Omega} v_\epsilon ~dx \leq C \left( \left(\int_{\Omega} \left(\frac{\left|\nabla v_\epsilon\right|}{\sqrt{v_\epsilon}}\right)^\frac{2N}{N+2} ~dx\right)^\frac{N+2}{N} + \left(\int_{\Omega} v_\epsilon^\frac{N}{N+2} ~dx\right)^\frac{N+2}{N}\right)
\]
with $C$ depending on $C_s$ and $N$. We start with estimating the first integral on the right-hand side of the above inequality. By the straightforward computation

\begin{equation}\label{eqest-1}
\begin{split}
& D_i \left(a_\epsilon (z)w_\epsilon^{\frac{p+r-2+\frac{2r}{N}}{2}}\right)\\
& =\frac{a_\epsilon (z)}{2}  w_\epsilon^{\frac{p+r-2+\frac{2r}{N}}{2}}\ln w_\epsilon D_ip + \left(p+r-2+\frac{2r}{N}\right) a_\epsilon (z) w_{\epsilon}^{\frac{p+r-2+\frac{2r}{N}}{2}-1}\sum_{j=1}^N
D_{j}uD_{ij}^2u + w_\epsilon^{\frac{p+r-2+\frac{2r}{N}}{2}} D_i a.
\end{split}
\end{equation}
The assumptions
\[
\min\{p(z), q(z)\} >\frac{2N}{N+2} \quad \text{for all} \ z \in \overline{Q}_T \quad \text{and} \ \max_{x \in \overline{Q}_T} |p(z) - q(z)| < \frac{2}{N+2}
\]
imply

\begin{equation}\label{lowerest-2}
    2p - \max\{p,q\} > 0.
\end{equation}
Now, by using \eqref{lowerest-1}, \eqref{eqest-1} and \eqref{lowerest-2}, we have
\begin{equation}\label{upperest-1}
\begin{split}
    \dfrac{1}{\sqrt{v_{\epsilon}}} & \left|\nabla \left(a_\epsilon (z)w_\epsilon^{\frac{p+r-2+\frac{2r}{N}}{2}}\right)\right|\leq C_1 \sqrt{a_\epsilon (z)} w_\epsilon^{\frac{p+r-2+\frac{2r}{N}}{4}} |\ln w_\epsilon| +  C_2 r \sqrt{a_\epsilon (z)} w_\epsilon^{\frac{p+r-4+\frac{2r}{N}}{4}} |u_{xx}| \\
    & \qquad + C_3 \begin{cases}
 w_\epsilon^{\frac{2p-\overline{s} +r-2+\frac{2r}{N}}{4}}  & \text{if $|w_\epsilon| < 1$},
\\
w_\epsilon^{\frac{2p - \underline{s}+r-2+\frac{2r}{N}}{4}}  & \text{if $|w_\epsilon|\geq 1$}
\end{cases}\\
& \leq  C_1 \sqrt{a_\epsilon (z)} w_\epsilon^{\frac{p+r-2+\frac{2r}{N}}{4}} |\ln w_\epsilon| +  C_2 r \sqrt{a_\epsilon (z)} w_\epsilon^{\frac{p+r-4+\frac{2r}{N}}{4}} |u_{xx}|
+ C_3 w_\epsilon^{\frac{p + r-2+\frac{2r}{N} + \lambda}{4}} + C_4,
\end{split}
\end{equation}
where $C_i$ are independent of $\epsilon$ and $r$ and
\[
\lambda:= \max_{z \in \overline{Q}_T} |p(z)-q(z)|.
\]
Similarly, we have
\begin{equation}
\label{eq:upperest-2}
    \begin{split}
    \frac{1}{\sqrt{v_\epsilon}}\left|\nabla \left(b_\epsilon (z)w_\epsilon^{\frac{q+r-2+\frac{2r}{N}}{2}}\right)\right|
& \leq  C_1' \sqrt{b_\epsilon (z)} w_\epsilon^{\frac{q+r-2+\frac{2r}{N}}{4}} |\ln w_\epsilon| +  C_2' r \sqrt{b_\epsilon (z)} w_\epsilon^{\frac{q+r-4+\frac{2r}{N}}{4}} |u_{xx}| \\
& \qquad + C_3' w_\epsilon^{\frac{q+ r-2+\frac{2r}{N} + \lambda}{4}} + C_4'.
\end{split}
\end{equation}
Combining \eqref{upperest-1} and \eqref{eq:upperest-2}, we obtain
\[
\begin{split}
\left(\int_{\Omega} \left(\frac{\left|\nabla v_\epsilon\right|}{\sqrt{v_\epsilon}}\right)^\frac{2N}{N+2} ~dx\right)^\frac{N+2}{N} &\leq C_1'' \left(r^2\int_\Omega \left(\left(\sqrt{a_\epsilon (z)} w_\epsilon^{\frac{p+r-4+\frac{2r}{N}}{4}} + \sqrt{b_\epsilon (z)} w_\epsilon^{\frac{q+r-4+\frac{2r}{N}}{4}}\right)|u_{xx}|\right)^{\frac{2N}{N+2}}\,dx\right)^{\frac{N+2}{N}}
\\
& \quad + C_2'' \left(\int_\Omega \left( \left(\sqrt{a_\epsilon (z)} w_\epsilon^{\frac{p+r-2+\frac{2r}{N}}{4}} + \sqrt{b_\epsilon (z)} w_\epsilon^{\frac{q+r-2+\frac{2r}{N}}{4}}\right) |\ln w_\epsilon| \right)^{\frac{2N}{N+2}}\,dx\right)^{\frac{N+2}{N}}
\\
& \quad + C_3''\left(\int_\Omega \left( w_\epsilon^{\frac{\max\{p,q\}+r-2+\frac{2r}{N} +\lambda}{4}} \right)^{\frac{2N}{N+2}}\,dx\right)^{\frac{N+2}{N}} + C_4''
\\
& \equiv C'' \left(r^2I_1+I_2+I_3 + 1\right)
\end{split}
\]
with a constant $C_i''$ depending on $C_i$, $C_i'$, $p^\pm$, $q^\pm$, $L_{p,q}$, $N$, but independent of $r$ and $\epsilon$. The integrals $I_k$ are estimated separately. By H\"older's inequality with the conjugate exponents $\frac{N+2}{N}$ and $\frac{N+2}{2}$

\[
\begin{split}
I_1^{\frac{N}{N+2}} &  \leq C \int_{\Omega}\left(\left(a_\epsilon (z) w_{\epsilon}^{\frac{p+r-4}{2}} + b_\epsilon (z) w_{\epsilon}^{\frac{q+r-4}{2}}\right)|u_{xx}|^2\right)^{\frac{N}{N+2}} \left(w_{\epsilon}^{\frac{r}{N+2}}\right)\,dx \\
& \leq C \left(\int_{\Omega} \left(a_\epsilon (z) w_{\epsilon}^{\frac{p+r-4}{2}} + b_\epsilon (z) w_{\epsilon}^{\frac{q+r-4}{2}}\right) |u_{xx}|^2\,dx\right)^{\frac{N}{N+2}} \left(\int_\Omega w_\epsilon^{\frac{r}{2}}\,dx\right)^{\frac{2}{N+2}},
\end{split}
\]
whence

\[
I_1 \leq C \left(\int_{\Omega} \left(a_\epsilon (x) w_{\epsilon}^{\frac{p+r-4}{2}} + b_\epsilon (x) w_{\epsilon}^{\frac{q+r-4}{2}}\right) |u_{xx}|^2\,dx \right) \left(\int_\Omega w_\epsilon^{\frac{r}{2}}\,dx\right)^{\frac{2}{N}}.
\]
Proceeding in the same way, we estimate

\[
\begin{split}
I_2 & \equiv \left(\int_\Omega \left(\left(\sqrt{a_\epsilon (z)} w_\epsilon^{\frac{p+r-2}{4}} + \sqrt{b_\epsilon (z)} w_\epsilon^{\frac{q+r-2}{4}}\right)|\ln w_\epsilon|\right)^{\frac{2N}{N+2}}\left(w_\epsilon^{\frac{r}{N+2}}\right)\,dx\right)^{\frac{N+2}{N}}
\\
&
\leq C \left(\int_\Omega  \left(a_\epsilon (z) w_{\epsilon}^{\frac{p+r-2}{2}} + b_\epsilon (z) w_{\epsilon}^{\frac{q+r-2}{2}}\right)  \ln^2 w_\epsilon\,dx\right) \left(\int_\Omega w_\epsilon^{\frac{r}{2}}\,dx\right)^{\frac{2}{N}}.
\end{split}
\]
To estimate $I_3$, we apply the higher integrability estimate of Corollary \ref{cor:d-r} with
\[
s = 2 \lambda,\quad \mathfrak{p} = p+r-2, \quad\mathfrak{q} = q+r-2.
\]
It is easy to see that for the so chosen $\mathfrak{p}$ and $\mathfrak{q}$

\[
\min\{\mathfrak{p}(z), \mathfrak{q}(z)\} >\frac{2N}{N+2} \quad \text{for all} \ z \in \overline{Q}_T \quad \text{and} \ \max_{z \in \overline{Q}_T} |\mathfrak{p}(z) - \mathfrak{q}(z)| < \frac{2}{N+2}.
\]
By applying H\"older's inequality with the conjugate exponents $\frac{N+2}{N}$ and $\frac{N+2}{2}$ and using \eqref{eq:d-r-ell}, we have the following estimate:

\[
\begin{split}
I_3 & \equiv \left(\int_\Omega \left( w_\epsilon^{\frac{\overline{s}(z)+r-2+\frac{2r}{N} +\lambda}{4}} \right)^{\frac{2N}{N+2}}\,dx\right)^{\frac{N+2}{N}} \leq C_1 + C_2 \left(\int_\Omega \left( w_\epsilon^{\frac{\underline{s}(z)+r-2+\frac{2r}{N} + 2\lambda}{4}} \right)^{\frac{2N}{N+2}}\,dx\right)^{\frac{N+2}{N}} \\
& = C_1 + C_2 \left(\int_\Omega \left( w_\epsilon^{\frac{\underline{s}(z)+r-2+ 2\lambda}{4}} \right)^{\frac{2N}{N+2}} w_\epsilon^\frac{r}{N+2} \,dx\right)^{\frac{N+2}{N}}
\\
& \leq C_1 + C_2 \left(\int_\Omega  w_\epsilon^{\frac{\underline{s}(z)+r-2+ 2\lambda}{2}} \,dx\right) \left(\int_{\Omega} w_\epsilon^\frac{r}{2} \,dx\right)^\frac{2}{N}\\
& \leq C_3 + C_4 \left(\int_{\Omega} w_\epsilon^\frac{r}{2} \,dx\right)^\frac{2}{N} + C_5 \left(\int_{\Omega} \left(a_\epsilon (x) w_{\epsilon}^{\frac{p+r-4}{2}} + b_\epsilon (x) w_{\epsilon}^{\frac{q+r-4}{2}}\right)|u_{xx}|^2\,dx\right)\left(\int_{\Omega} w_\epsilon^\frac{r}{2} \,dx\right)^\frac{2}{N}
\end{split}
\]
Gathering the estimates on $I_k$, we obtain the inequality
\begin{equation}
\label{eq:new-1}
\left(\int_{\Omega} \left(\frac{\left|\nabla v_\epsilon\right|}{\sqrt{v_\epsilon}}\right)^\frac{2N}{N+2} ~dx\right)^\frac{N+2}{N} \leq \widetilde{C} \Pi_1(t) \Pi_2^{\frac{2}{N}}(t) + \tilde{C},
\end{equation}
where $\widetilde C$, $\tilde{C}$ are constants independent of $r$ and $\epsilon$ and
\[
\begin{split}
& \Pi_1(t) = 1+ r^2\int_\Omega \left(a_\epsilon (z) w_{\epsilon}^{\frac{p+r-4}{2}} + b_\epsilon (z) w_{\epsilon}^{\frac{q+r-4}{2}}\right) |u_{xx}|^2\,dx \\
& \qquad \qquad \qquad \qquad  + \int_{\Omega} \left(a_\epsilon (z) w_{\epsilon}^{\frac{p+r-2}{2}} + b_\epsilon (z) w_{\epsilon}^{\frac{q+r-2}{2}}\right) \left(1+\ln^2w_\epsilon\right)\,dx,
\\
& \Pi_2(t)= \int_\Omega w_\epsilon^{\frac{r}{2}}\,dx.
\end{split}
\]
Again, applying the H\"older's inequality with the conjugate exponents $\frac{N+2}{N}$ and $\frac{N+2}{2}$, we have
\[
\begin{split}
\left(\int_{\Omega} v_\epsilon^\frac{N}{N+2} ~dx\right)^\frac{N+2}{N} & \leq C \left(\int_\Omega  \left(a_\epsilon (z) w_{\epsilon}^{\frac{p+r-2}{2}} + b_\epsilon (z) w_{\epsilon}^{\frac{q+r-2}{2}}\right)\,dx\right) \left(\int_\Omega w_\epsilon^{\frac{r}{2}}\,dx\right)^{\frac{2}{N}}
\leq C \Pi_1(t) \Pi_2^{\frac{2}{N}}(t).
\end{split}
\]
It follows that
\begin{equation}
\label{eq:new-2}
\int_{Q_T}\left(a_\epsilon (z)w_\epsilon^{\frac{p+r-2+\frac{2r}{N}}{2}} + b_\epsilon (z)w_\epsilon^{\frac{q+r-2+\frac{2r}{N}}{2}}\right) \,dz \leq \widetilde{C} \left(\sup_{(0,T)}\Pi_2(t)\right)^{\frac{2}{N}}\int_{Q_T}\Pi_1(t)\,dt.
\end{equation}
Now we plug the obtained inequalities into \eqref{eq:I-1}:

\[
\begin{split}
\mathcal{I} & \leq \|f\|^2_{N+2,Q_T}\left(C(\mu)+\mu C'\left(\sup_{(0,T)}\Pi_2(t)\right)^{\frac{2}{N}}\int_{Q_T}\Pi_1(t)\,dt \right)^{\frac{N}{N+2}}
\\
& \leq C''\|f\|^2_{N+2,Q_T}\left(C(\mu)+\mu C' \left(\sup_{(0,T)}\Pi_2(t)\right)^{\frac{2}{N+2}}\left(\int_{Q_T}\Pi_1(t)\,dt\right)^{\frac{N}{N+2}} \right)
\end{split}
\]
with an arbitrary $\mu>0$ and constants $C(\mu)$, $C'$, $C''$ independent of $w_\epsilon$ and $u$. By Young's inequality with the conjugate exponents $\frac{N+2}{N}$ and $\frac{N+2}{2}$ we continue the last inequality as follows:

\[
\begin{split}
\mathcal{I} & \leq C''\|f\|^2_{N+2,Q_T}\left(C(\mu)+ C'\left(\mu^{\frac{N+2}{4}} \sup_{(0,T)}\Pi_2(t)\right)^{\frac{2}{N+2}} \left(\mu^{\frac{N+2}{2N}}\int_{Q_T}\Pi_1(t)\,dt\right)^{\frac{N}{N+2}}
\right)
\\
& \leq  C''\|f\|^2_{N+2,Q_T}\left(C(\mu)+ C'\mu^{\frac{N+2}{4}} \sup_{(0,T)}\Pi_2(t)+ C'  \mu^{\frac{N+2}{2N}}\int_{Q_T}\Pi_1(t)\,dt\right)
\\
& \equiv C''\|f\|^2_{N+2,Q_T}\left(C(\mu)+ \mu^{\frac{N+2}{4}} \sup_{(0,T)}\|w_\epsilon\|_{\frac{r}{2},\Omega}^{\frac{r}{2}} + C'\mu^{\frac{N+2}{2N}}\left[r^2\int_{Q_T} \left(a_\epsilon (z) w_{\epsilon}^{\frac{p+r-4}{2}} + b_\epsilon (z) w_{\epsilon}^{\frac{q+r-4}{2}}\right) |u_{xx}|^2\,dz \right.\right.\\
& \left.\left. \qquad \qquad + \int_{Q_T} \left(a_\epsilon (z) w_{\epsilon}^{\frac{p+r-2}{2}} + b_\epsilon (z) w_{\epsilon}^{\frac{q+r-2}{2}}\right)\left(1+\ln^2w_\epsilon\right)\,dz\right]\right).
\end{split}
\]
By \eqref{eq:log} and Theorem \ref{th:integr-par} the second integral in the square brackets is bounded by

\[
\nu \int_{Q_T}\left(a_\epsilon (z) w_{\epsilon}^{\frac{p+r-4}{2}} + b_\epsilon (z) w_{\epsilon}^{\frac{q+r-4}{2}}\right)|u_{xx}|^2\,dz + \widehat C
\]
with an arbitrary $\nu>0$ and a constant $\widehat C=\widehat C(r,N,p^\pm, q^\pm, r,L_{p,q},\nu,\|u\|_{2,\Omega})$. Thus,

\begin{equation}
\label{eq:I-2}
\begin{split}
\mathcal{I}&\leq C_1 \|f\|_{N+2,Q_T}^2\left( \mu^{\frac{N+2}{4}} \sup_{(0,T)}\|w_\epsilon\|_{\frac{r}{2},\Omega}^{\frac{r}{2}} \right.
\\
&
\left. \qquad + \mu^{\frac{N+2}{2N}} \left(r^2+\nu\right) \int_{Q_T}\left(a_\epsilon (z) w_{\epsilon}^{\frac{p+r-4}{2}} + b_\epsilon (z) w_{\epsilon}^{\frac{q+r-4}{2}}\right) |u_{xx}|^2\,dz + C_2\right)
\end{split}
\end{equation}
with a constant $C_1$ depending only on $\textbf{data}$, and $C_2$ depending on $\textbf{data}$ and $\mu$, $\nu$. Substituting \eqref{eq:I-2} into \eqref{eq:ODI-new-1}
and choosing $\mu$ and $\nu$ sufficiently small we transform \eqref{eq:ODI-new-1} into

\begin{equation}
\label{eq:final}
\begin{split}
\sup_{(0,T)} \int_{\Omega}w_\epsilon^{\frac{r}{2}}\,dx  & + \alpha  \int_{Q_T} \left(a_\epsilon (z) w_\epsilon^{\frac{p+r}{2}-2} + b_\epsilon (z) w_\epsilon^{\frac{q+r}{2}-2} \right)|u_{xx}|^2\,dz
 \\
& + \beta \int_{Q_T} \left(a_\epsilon (z) \left|\nabla \left(w_\epsilon^{\frac{p+r-2}{4}}\right)\right|^2 + b_\epsilon (z) \left|\nabla \left(w_\epsilon^{\frac{q+r-2}{4}}\right)\right|^2 \right) \,dz
\leq \gamma + \int_{\Omega}w_{\epsilon}^{\frac{r}{2}}(x,0)\,dx
\end{split}
\end{equation}
with finite positive constants $\alpha$, $\beta$, $\gamma$ depending only on $\textbf{data}$, and independent of $\epsilon$.

\subsection{Case II: $\sigma \in (2,N+2)$}
\label{subsec:est-2}
We refine the estimates on the integral \eqref{eq:integral-I} and extend them to the case of low integrability of $f$.
Let us take $\sigma>2$, $r\geq \max\{2,p^+,q^+\}$, and search for a (optimal) $\beta$ such that the following inequality holds true:
\begin{equation}\label{imp:ineq:modi}
    \frac{(r-p)\sigma}{2(\sigma-2)} \leq \frac{p+r-2+\frac{2r}{\beta}}{2} \quad \text{and} \quad \frac{(r-q)\sigma}{2(\sigma-2)} \leq \frac{q+r-2+\frac{2r}{\beta}}{2}.
\end{equation}
By the Young inequality (cf. with \eqref{eq:I-1})
\begin{equation}
  \label{eq:I-1-modi}
  \begin{split}
  \mathcal{I} & \leq \|f\|_{\sigma,Q_T}^2 \left(\int_{Q_T} \left( a_\epsilon (z)w_\epsilon^{\frac{r-p}{2}} + b_\epsilon (z)w_\epsilon^{\frac{r-q}{2}} \right)^{\frac{\sigma}{\sigma-2}}\,dz\right)^{\frac{\sigma-2}{\sigma}}\\
  & \leq \|f\|_{\sigma,Q_T}^2 \left(C + \int_{Q_T} \left( a_\epsilon (z)w_\epsilon^\frac{p+r-2+\frac{2r}{\beta}}{2} + b_\epsilon (z)w_\epsilon^\frac{q+r-2+\frac{2r}{\beta}}{2} \right)\,dz\right)^{\frac{\sigma-2}{\sigma}}.
  \end{split}
  \end{equation}
Fix
\[
\mu \in \left(\frac{N}{N+2},1\right) \qquad \text{and set} \qquad \alpha_\sharp = 2 \mu, \quad  \quad \alpha_\sharp^\ast =\frac{2\mu N}{N-2\mu} > 2.
\]
Since $\alpha_\sharp^\ast$ is the Sobolev conjugate exponent, by the Sobolev embedding theorem $W^{1,\alpha_\sharp}(\Omega)
\subset L^{\alpha_\sharp^\ast}(\Omega) \subset L^2(\Omega)$, and for every $v\in W^{1, \alpha_\sharp}(\Omega)$
\begin{equation}
\label{eq:Sob-1-modi}
\|v\|_{2,\Omega}^{2} \leq C_{s}\left(\|\nabla v\|_{\alpha_\sharp,\Omega}^{2} + \|v\|^{2}_{\alpha_\sharp,\Omega}\right)
\end{equation}
with an independent of $v$ constant $C_s$. Denote
\[
s_\epsilon(z) := a_\epsilon (z)w_\epsilon^{\frac{p+r-2+\frac{2r}{\beta}}{2}} + b_\epsilon (z)w_\epsilon^{\frac{q+r-2+\frac{2r}{\beta}}{2}}.
\]
Applying \eqref{eq:Sob-1-modi} to $s_\epsilon^\frac{1}{2}$ we arrive at the inequality

\[
\begin{split}
\int_{\Omega} s_\epsilon \,dx \leq C_s'\left(\int_\Omega \left|\frac{\nabla s_\epsilon}{\sqrt{s_\epsilon}}\right|^{2 \mu}\,dz \right)^{\frac{1}{\mu}} + C_s' \left(\int_\Omega s_\epsilon^{\mu}\,dz \right)^{\frac{1}{\mu}},
\end{split}
\]
with $C'_s$ depending on $C_s$ and $\sigma$. Using \eqref{lowerest-2}, we have
\begin{equation}\label{upperest-3}
\begin{split}
    \frac{1}{\sqrt{s_\epsilon}}\left|\nabla \left(a_\epsilon (z)w_\epsilon^{\frac{p+r-2+\frac{2r}{\beta}}{2}}\right)\right|
& \leq  C_1 a_\epsilon (z)^\frac{1}{2} w_\epsilon^{\frac{p+r-2+\frac{2r}{\beta}}{4}} |\ln w_\epsilon| +  C_2 r a_\epsilon (z)^\frac{1}{2} w_\epsilon^{\frac{p+r-4+\frac{2r}{\beta}}{4}} |u_{xx}| \\
& \qquad + C_3 w_\epsilon^{\frac{p + r-2+\frac{2r}{\beta} + \lambda}{4}} + C_4
\end{split}
\end{equation}
and
\begin{equation}\label{upperest-2}
    \begin{split}
    \frac{1}{\sqrt{s_\epsilon}}\left|\nabla \left(b_\epsilon (z)w_\epsilon^{\frac{q+r-2+\frac{2r}{\beta}}{2}}\right)\right|
& \leq  C_1' b_\epsilon (z)^\frac{1}{2} w_\epsilon^{\frac{q+r-2+\frac{2r}{\beta}}{4}} |\ln w_\epsilon| +  C_2' r b_\epsilon (z)^\frac{1}{2} w_\epsilon^{\frac{q+r-4+\frac{2r}{\beta}}{4}} |u_{xx}| \\
& \qquad + C_3' w_\epsilon^{\frac{q+ r-2+\frac{2r}{\beta} + \lambda}{4}} + C_4'.
\end{split}
\end{equation}
Combining \eqref{upperest-1} and \eqref{upperest-1}, we obtain
\[
\begin{split}
\left(\int_{\Omega} \left(\frac{\left|\nabla s_\epsilon\right|}{\sqrt{s_\epsilon}}\right)^{2 \mu} ~dx\right)^\frac{1}{\mu} &\leq C_1'' \left(r^2\int_\Omega \left(\left(\sqrt{a_\epsilon (z)} w_\epsilon^{\frac{p+r-4+\frac{2r}{\beta}}{4}} + \sqrt{b_\epsilon (z)} w_\epsilon^{\frac{q+r-4+\frac{2r}{\beta}}{4}}\right)|u_{xx}|\right)^{2 \mu}\,dx\right)^{\frac{1}{\mu}}
\\
& \quad + C_2'' \left(\int_\Omega \left( \left(\sqrt{a_\epsilon (z)} w_\epsilon^{\frac{p+r-2+\frac{2r}{\beta}}{4}} + \sqrt{b_\epsilon (z)} w_\epsilon^{\frac{q+r-2+\frac{2r}{\beta}}{4}}\right) |\ln w_\epsilon| \right)^{2 \mu}\,dx\right)^{\frac{1}{\mu}}
\\
& \quad + C_3''\left(\int_\Omega \left( w_\epsilon^{\frac{\max\{p,q\}+r-2+\frac{2r}{\beta} +\lambda}{4}} \right)^{2 \mu}\,dx\right)^{\frac{1}{\mu}} + C_4''
\\
& \equiv C'' \left(r^2I_1+I_2+I_3 + 1\right)
\end{split}
\]
with a constant $C_i''$ depending on $C_i$, $C_i'$, $p^\pm$, $q^\pm$, $L_{p,q}$, $\beta$, but independent of $r$ and $\epsilon$. The integrals $I_k$ are estimated separately. By H\"older's inequality with the conjugate exponents $\frac{2}{\alpha_\sharp} =\frac{1}{\mu} >1$ and $ \frac{1}{1-\mu}$

\[
\begin{split}
I_1^{\mu} &  \leq C \int_{\Omega} \left(\left(a_\epsilon (z) w_{\epsilon}^{\frac{p+r-4}{2}} + b_\epsilon (z) w_{\epsilon}^{\frac{q+r-4}{2}}\right) |u_{xx}|^2\right)^{\mu} \left(w_{\epsilon}^{\frac{r}{2}}\right)^{\frac{2\mu}{\beta}}\,dx \\
& \leq \left(\int_{\Omega} \left(a_\epsilon (z) w_{\epsilon}^{\frac{p+r-4}{2}} + b_\epsilon (z) w_{\epsilon}^{\frac{q+r-4}{2}}\right)|u_{xx}|^2\,dx\right)^{\mu} \left(\int_\Omega \left(w_\epsilon^{\frac{r}{2}}\right)^{\frac{2\mu }{\beta(1-\mu)}}\,dx\right)^{1-\mu}.
\end{split}
\]
Let

\[
\beta=\frac{2\mu}{1-\mu}.
\]
With this choice of $\beta$ we may estimate

\[
I_1\leq \left(\int_{\Omega} \left(a_\epsilon (z) w_{\epsilon}^{\frac{p+r-4}{2}} + b_\epsilon (z) w_{\epsilon}^{\frac{q+r-4}{2}}\right)|u_{xx}|^2\,dx\right) \left(\int_\Omega w_\epsilon^{\frac{r}{2}}\,dx\right)^{\frac{(1-\mu)}{\mu}},
\]
while inequality \eqref{imp:ineq:modi} takes on the form

\begin{equation}
\notag
     \begin{split}
     \frac{(r-p)\sigma}{2(\sigma-2)} & \leq \frac{p+r-2+\frac{r(1-\mu)}{\mu}}{2} \qquad  \Leftrightarrow \quad
      \mu \sigma (r-p)  \leq  \mu (p+r-2)(\sigma-2) + r(1-\mu) (\sigma-2)
     \end{split}
     \end{equation}
and
\begin{equation}
\notag
     \begin{split}
     \frac{(r-q)\sigma}{2(\sigma-2)} & \leq \frac{q+r-2+\frac{r(1-\mu)}{\mu}}{2} \qquad  \Leftrightarrow \quad
      \mu \sigma (r-q)  \leq  \mu (q+r-2)(\sigma-2) + r(1-\mu) (\sigma-2).
     \end{split}
     \end{equation}

Solving the last two inequalities for $r$, gathering the results with the conditions of Proposition \ref{pro:pointwise}, and using the fact that $\mu > \frac{N}{N+2} > \frac{\sigma-2}{\sigma}$, we arrive at the following restriction on $r$ in terms of $\sigma$, $p^\pm$, $q^\pm$ and $N$:

\begin{equation}
\label{eq:r-modi}
\max\{p^+, q^+, 2\}\leq r \leq \frac{2(\min\{p^-, q^-\}(\sigma-1) -\sigma + 2)}{\sigma - \frac{\sigma-2}{\mu}}:=\kappa(\mu).
\end{equation}

Proceeding in the same way, we estimate

\[
\begin{split}
I_2 & \equiv \left(\int_\Omega \left( \left(a_\epsilon (z) w_{\epsilon}^{\frac{p+r-2}{2}} + b_\epsilon (z) w_{\epsilon}^{\frac{q+r-2}{2}}\right) \ln^2 w_\epsilon\right)^{\mu}\left(w_{\epsilon}^{\frac{r}{2}} \right)^\frac{2\mu}{\beta}\,dx\right)^{{\frac{1}{\mu}}}
\\
&
\leq \int_\Omega  \left(a_\epsilon (z) w_{\epsilon}^{\frac{p+r-2}{2}} + b_\epsilon (z) w_{\epsilon}^{\frac{q+r-2}{2}}\right)  \ln^{2} w_\epsilon\,dx \left(\int_\Omega w_\epsilon^{\frac{r}{2}}\,dx\right)^{\frac{1-\mu}{\mu}},
\end{split}
\]

\[
\begin{split}
I_3 & \equiv \left(\int_\Omega \left( w_\epsilon^{\frac{\max\{p,q\}+r-2+\frac{2r}{\beta} +\lambda}{4}} \right)^{2\mu}\,dx\right)^{\frac{1}{\mu}} \leq C_1 + C_2 \left(\int_\Omega \left( w_\epsilon^{\frac{\min\{p,q\}+r-2+\frac{2r}{\beta} + 2\lambda}{4}} \right)^{2\mu}\,dx\right)^{\frac{1}{\mu}} \\
& =  C_1 + C_2 \left(\int_\Omega \left( w_\epsilon^{\frac{\min\{p,q\}+r-2+\frac{2r}{\beta} + 2\lambda}{4}} \right)^{2\mu} (w_\epsilon^\frac{r}{2})^\frac{2\mu}{\beta}\,dx\right)^{\frac{1}{\mu}}\\
&\leq C_1 + C_2 \left(\int_\Omega w_\epsilon^{\frac{r}{2}}\,dx\right)^{\frac{1-\mu}{\mu}} + C_3 \int_\Omega  \left(a_\epsilon (z) w_{\epsilon}^{\frac{p+r-2}{2}} + b_\epsilon (z) w_{\epsilon}^{\frac{q+r-2}{2}}\right) \,dx \left(\int_\Omega w_\epsilon^{\frac{r}{2}}\,dx\right)^{\frac{1-\mu}{\mu}}.
\end{split}
\]
Gathering the estimates on $I_k$, $k=1,2,3$, we obtain the inequality

\begin{equation}
\label{eq:new-1-modi}
\left(\int_{\Omega} \left(\frac{\left|\nabla s_\epsilon\right|}{\sqrt{s_\epsilon}}\right)^{2\mu} ~dx\right)^\frac{1}{\mu}
 \leq \widetilde{C} \Pi_1(t) \Pi_2^{\frac{1-\mu}{\mu}}(t),
\end{equation}
where

\[
\begin{split}
& \Pi_1(t) = 1+ r^2\int_\Omega \left(a_\epsilon (z) w_{\epsilon}^{\frac{p+r-4}{2}} + b_\epsilon (z) w_{\epsilon}^{\frac{q+r-4}{2}}\right) |u_{xx}|^2\,dx \\
& \qquad \qquad \qquad \qquad  + \int_{\Omega} \left(a_\epsilon (z) w_{\epsilon}^{\frac{p+r-2}{2}} + b_\epsilon (z) w_{\epsilon}^{\frac{q+r-2}{2}}\right) \left(1+\ln^2w_\epsilon\right)\,dx,
\\
& \Pi_2(t)= \int_\Omega w_\epsilon^{\frac{r}{2}}\,dx.
\end{split}
\]
Again, applying the H\"older's inequality with the conjugate exponents $\frac{1}{\mu}$ and $\frac{1}{1-\mu}$, we have
\[
\begin{split}
\left(\int_{\Omega} s_\epsilon^\mu~dx\right)^\frac{1}{\mu} & \leq C \left(\int_\Omega  \left(a_\epsilon (z) w_{\epsilon}^{\frac{p+r-2}{2}} + b_\epsilon (z) w_{\epsilon}^{\frac{q+r-2}{2}}\right)\,dx\right) \left(\int_\Omega w_\epsilon^{\frac{r}{2}}\,dx\right)^{\frac{1-\mu}{\mu}}\\
& \leq \widetilde{C} \left(\sup_{(0,T)}\Pi_2(t)\right)^{\frac{1-\mu}{\mu}}\int_{Q_T}\Pi_1(t)\,dt.
\end{split}
\]
Now we plug the obtained inequalities into \eqref{eq:I-1-modi}:

\begin{equation}
\label{eq:last-modi}
\begin{split}
\mathcal{I} & \leq \|f\|^2_{\sigma,Q_T}\left(C+ C'\left(\sup_{(0,T)}\Pi_2(t)\right)^{\frac{1-\mu}{\mu}
}\int_{Q_T}\Pi_1(t)\,dt \right)^{\frac{\sigma-2}{\sigma}}
\\
& \leq C''\|f\|^2_{\sigma,Q_T}\left(1+ \left(\sup_{(0,T)}\Pi_2(t)\right)^{\frac{(1-\mu)(\sigma-2)}{\mu \sigma} }\left(\int_{Q_T}\Pi_1(t)\,dt\right)^{\frac{\sigma-2}{\sigma}} \right)
\end{split}
\end{equation}
with constants $C'$, $C''$ independent of $w_\epsilon$ and $u$.

For any $\sigma \in (2, N+2)$ and $1>\mu >  \frac{N}{N+2}$, we have the following inequalities:
\begin{equation}
\label{eq:ineq-app}
\begin{split}
\mu > \frac{N}{N+2} > \frac{\sigma-2}{\sigma} \quad  & \Leftrightarrow \quad 2 \mu - (1-\mu)(\sigma-2) >0 \quad \Leftrightarrow \quad  \frac{\mu(\sigma-2)}{(\mu \sigma-(1-\mu)(\sigma-2))}<1
\end{split}
\end{equation}
and
\begin{equation}
     \mu > \frac{N}{N+2} > \frac{(\sigma-2)}{2(\sigma-1)}  \qquad \Leftrightarrow \quad \sigma \mu - (1-\mu)(\sigma-2) >0 \qquad \Leftrightarrow \quad  \frac{(1-\mu)(\sigma-2)}{\mu \sigma} < 1.
\end{equation}
By applying Young's inequality two times with the exponents
\[
\frac{\mu \sigma}{(1-\mu)(\sigma-2)} ,\quad \gamma= \frac{\mu \sigma}{\mu \sigma-(1-\mu)(\sigma-2)}, \quad \text{and}\quad  \frac{\sigma}{\gamma(\sigma-2)},\quad  \frac{\sigma}{\sigma-\gamma(\sigma-2)},
\]
we find that for every $\lambda >0$
\[
\begin{split}
\left(\sup_{(0,T)}\Pi_2(t)\right)^{\frac{(1-\mu)(\sigma-2)}{\mu \sigma}} & \left(\int_{Q_T}\Pi_1(t)\,dt\right)^{\frac{\sigma-2}{\sigma}}
\leq\lambda^\frac{\mu \sigma}{(1-\mu)(\sigma-2)}\sup_{(0,T)}\Pi_2(t) + \frac{1}{\lambda^\gamma} \left(\int_{Q_T}\Pi_1(t)\,dt\right)^{\frac{\gamma(\sigma-2)}{\sigma}}\\
& \leq\lambda^\frac{\mu \sigma}{(1-\mu)(\sigma-2)}\sup_{(0,T)}\Pi_2(t) +  \lambda^\frac{\sigma}{\sigma-2} \int_{Q_T}\Pi_1(t)\,dt + C(\lambda).
\end{split}
\]
Estimate \eqref{eq:last-modi} is then continued as follows: for $\sigma\in (2,N+2)$ and  $\mu \in \left(\frac{N}{N+2}, 1 \right)$,
\[
\begin{split}
\mathcal{I}
& \leq  C''\|f\|^2_{\sigma,Q_T}\left(C(\lambda)+ \lambda^\frac{\mu \sigma}{(1-\mu)(\sigma-2)}
 \sup_{(0,T)}\Pi_2(t)+  \lambda^{\frac{\sigma}{\sigma-2}}\int_{Q_T}\Pi_1(t)\,dt\right)
\\
& \equiv C''\|f\|^2_{\sigma,Q_T}\left(C(\lambda)+ \lambda^\frac{\mu \sigma}{(1-\mu)(\sigma-2)}  \sup_{(0,T)}\|w_\epsilon\|_{\frac{r}{2},\Omega}^{\frac{r}{2}}\right.
\\
& \qquad \qquad \qquad \qquad  \left.
+ C'\lambda^{\frac{\sigma}{(\sigma-2)}}\left[r^2\int_{Q_T} \left(a_\epsilon (z) w_{\epsilon}^{\frac{p+r-2}{2}} + b_\epsilon (z) w_{\epsilon}^{\frac{q+r-2}{2}}\right) |u_{xx}|^2\,dz\right.\right.\\
& \left.\left. \qquad \qquad \qquad \qquad+\int_{Q_T} \left(a_\epsilon (z) w_{\epsilon}^{\frac{p+r-2}{2}} + b_\epsilon (z) w_{\epsilon}^{\frac{q+r-2}{2}}\right)\left(1+\ln^2w_\epsilon\right)\,dz\right]\right).
\end{split}
\]
By Corollary \ref{cor:d-r} 
and \eqref{eq:log}, the second integral in the square brackets is bounded by

\[
\nu \int_{Q_T} \left(a_\epsilon (z) w_{\epsilon}^{\frac{p+r-2}{2}} + b_\epsilon (z) w_{\epsilon}^{\frac{q+r-2}{2}}\right) |u_{xx}|^2\,dz + \widehat C
\]
with an arbitrary $\nu>0$ and a constant $\widehat C$ depending on $r$, $N$, $p^\pm$, $q^\pm$ $L_{p,q}$, $\nu$, $\sup_{(0,T)}\|u(t)\|_{2,\Omega}$, and $a^\pm$, $b^\pm$, $\|\nabla a\|_{d,Q_T}$, $\|\nabla b\|_{d,Q_T}$. Thus,

\begin{equation}
\label{eq:I-2-new}
\begin{split}
\mathcal{I}\leq C_1 \|f\|_{\sigma,Q_T}^2 & \left(\lambda^\frac{\mu \sigma}{(1-\mu)(\sigma-2)}  \sup_{(0,T)}\|w_\epsilon\|_{\frac{r}{2},\Omega}^{\frac{r}{2}}\right.
\\
& \qquad \left.
 + \lambda^{\frac{\sigma}{\sigma-2}}\left(r^2+\nu\right) \int_{Q_T} \left(a_\epsilon(z) w_{\epsilon}^{\frac{p+r-2}{2}} + b_\epsilon(z) w_{\epsilon}^{\frac{q+r-2}{2}}\right) |u_{xx}|^2\,dz + C_2\right)
 \end{split}
\end{equation}
with a constant $C_1$ depending only on the same quantities as $\widehat C$, and $C_2$ depending also on $\mu$, $\nu$. Substituting \eqref{eq:I-2-new} into \eqref{eq:ODI-new-1}, using \eqref{eq:d-r-ell} and choosing $\lambda$ and $\nu$ sufficiently small we transform \eqref{eq:ODI-new-1} into \eqref{eq:final} with finite positive constants $\alpha$, $\beta$, $\gamma$ independent of $\epsilon$ and depending on the same quantities as $\widehat C$ and on $\|f\|_{\sigma,Q_T}$. The function $\kappa(\mu)$ in \eqref{eq:r-modi} is monotone decreasing and attains its maximum as $\mu\to \frac{N}{N+2}$. Since $\mu \in \left(\frac{N}{N+2},1\right)$ is arbitrary, condition \eqref{eq:r-modi} allows one to choose $\mu$ so close to $\frac{N}{N+2}$ that \eqref{eq:r-modi} holds true.

The results of this section are summarized in the following assertion.

\begin{lemma}
\label{le:clas-sol-summary}
The classical solution $u$ of problem \eqref{eq:main-reg} with smooth $u_0$, $f$, $a$, $b$, $p$, $q$ satisfies estimate \eqref{eq:final} with constants $\alpha$, $\beta$, $\gamma$ depending only on \textbf{data}. If $\sigma=N+2$, $r\geq \max\{2,p^+,q^+\}$ is an arbitrary finite number. If $\sigma\in (2,N+2)$, estimate \eqref{eq:final} holds for $r$ satisfying inequalities  \eqref{eq:r-modi}.
\end{lemma}

\begin{corollary}
\label{cor:high-int-fin}
Estimate \eqref{eq:final} and Theorem \ref{th:integr-par} yield the uniform estimate

\begin{equation}
\label{eq:high-int-uniform}
\int_{Q_T}\left(a_\epsilon(z)w_{\epsilon}^{\frac{p(z)+r+s}{2}-1} +b_\epsilon(z)w_{\epsilon}^{\frac{q(z)+r+s}{2}-1}\right)|\nabla u|^2\,dz\leq C, \qquad s\in \left(0,\frac{4}{N+2}\right)
\end{equation}
with a constant $C$ depending on the same quantities as the constant in \eqref{eq:final}. By Corollary \ref{cor:cont-embed}

\[
\|u\|_{\mathbb{W}_{\underline{s}(\cdot)+r+s}(Q_T)}\leq C
\]
with a constant $C$ depending on the same quantities as the constant in \eqref{eq:high-int-uniform}.
\end{corollary}

\section{Regularized problem with nonsmooth data}
\label{sec:nonsmooth-data}
Consider the regularized problem \eqref{eq:main-reg} with the data satisfying conditions \eqref{assum1}, \eqref{eq:Lip-p-q}, \eqref{eq:coeff}, \eqref{eq:mod-coeff}. Following \cite[Subsec.3.2]{Ar-Sh-2024}, we approximate the data by sequences of smooth functions $\{p_m\}$, $\{q_m\}$, $\{f_m\}$, $\{u_{0m}\}$, $\{a_m\}$, $\{b_m\}$ such that

\begin{equation}
\label{eq:data-approx-1}
\begin{split}
& \text{$p_m(z)\nearrow p(z)$, $q_m(z)\nearrow q(z)$ in $C^{0,1}(Q_T)$},
\\
& \text{$u_{0m}\in C_0^{\infty}(\Omega)$, $u_{0m}\to u_0$ in $W_0^{1,r}(\Omega)$},
\\
& \text{$f_m\in C_0^{\infty}(Q_T)$, $f_m\to f$ in $L^{\sigma}(Q_T)$},
\\
& \text{$a_m,b_m\in C^{2+\alpha,\frac{2+\alpha}{2}}(Q_T)$},
\\
& \text{$a_{m t}, b_{m t}, D_ia_m,\,D_ib_m \to a_{t}, b_{t}, D_ia,\,D_ib$ in $L^d(Q_T)$}.
\end{split}
\end{equation}
By agreement we will denote $\mathbf{data}_m=\{f_m,u_{0m},a_{m},b_m,p_m,q_m\}$ and $\mathbf{data}=\{f,u_0,a,b,p,q\}$. When $m\to \infty$ the norms of the components of $\mathbf{data}_m$ converge to the norms of the corresponding components of $\mathbf{data}$, which means that the constants depending on $\|\mathbf{data}_m\|$ in fact depend on $\|\mathbf{data}\|$.

By Lemma \ref{le:class-sol-existence} the problem

\begin{equation}
\label{eq:main-nonsmooth}
\begin{split}
& (u_m)_t-\operatorname{div}\left(\mathcal{F}^{(p_m-p,q_m-q)}_{\epsilon,m}(z,\nabla u_m)\nabla u_m\right)=f_m(z)\quad \text{in $Q_T$},
\\
& \text{$u_m=0$ on $\partial\Omega\times (0,T)$},\qquad \text{$u_m(x,0)=u_{0m}(x)$ in $\Omega$}
\end{split}
\end{equation}
with the smooth $\mathbf{data}_m$ and the flux function
\[
\mathcal{F}^{(p_m-p,q_m-q)}_{\epsilon,m}(z,\xi)=(a_m+\epsilon)(\epsilon^2+|\xi|^2)^{\frac{p_m-2}{2}} + (b_m+\epsilon)(\epsilon^2+|\xi|^2)^{\frac{q_m-2}{2}},\quad \epsilon
>0,
\]
has a unique classical solution.

\begin{theorem}
\label{th:limit-m}
Let the data of problem \eqref{eq:main-reg} satisfy conditions \eqref{eq:structure-prelim-1}, \eqref{assum1}, \eqref{eq:Lip-p-q}, \eqref{eq:structure-prelim-2} and \eqref{eq:coeff}. Then the sequence $\{u_m\}$ of classical solutions to problem \eqref{eq:main-nonsmooth} converges to the unique strong solution $u$ of problem \eqref{eq:main-reg}. The solution satisfies the estimates

\begin{equation}
\label{eq:a-priori-eps-1}
\|u_t\|^2_{2,Q_T}+ \int_{Q_T}|\nabla u|^{\underline{s}(z)+r+s}\,dz\leq C_1,\qquad s\in \left(0,\frac{4}{N+2}\right),
\end{equation}

\begin{equation}
\label{eq:a-priori-eps-2}
\begin{split}
\sup_{(0,T)} \int_{\Omega}w_\epsilon^{\frac{r}{2}}\,dx  &
\leq C_2 + \int_{\Omega}w_{\epsilon}^{\frac{r}{2}}(x,0)\,dx
\end{split}
\end{equation}
with positive constants $C_i$, depending only on $\textbf{data}$, $r$, and $\|f\|_{\sigma,Q_T}$ but independent of $\epsilon$.
\end{theorem}

\begin{theorem}
\label{th:second-order-epsilon}
Under the conditions of Theorem \ref{th:limit-m} the solution possesses the second order regularity:

\begin{equation}
\label{eq:second-order-eps}
\mathcal{G}_\epsilon \equiv a_\epsilon w_\epsilon^{\frac{p+r-2}{4}}+b_\epsilon w_\epsilon^{\frac{q+r-2}{4}}\in L^2(0,T;W^{1,2}(\Omega)),\qquad \|\mathcal{G}_{\epsilon}\|_{L^2(0,T;W^{1,2}(\Omega))}\leq C
\end{equation}
with an independent of $\epsilon$ constant $C$.
\end{theorem}

\subsection{Proof of Theorem \ref{th:limit-m}}
Multiplication of equation \eqref{eq:main-nonsmooth} by $u_m$ and integration by parts in $\Omega$ lead to the inequality

\begin{equation}
\label{eq:prelim}
\dfrac{1}{2}\dfrac{d}{dt }\|u_m(t)\|_{2,\Omega}^2+\int_{\Omega}\mathcal{F}^{^{(p_m-p,q_m-q)}}_{\epsilon,m}(z,\nabla u_m)|\nabla u_m|^2\,dz= \int_{\Omega}f_mu_m\,dz\leq \frac{1}{2}\|f_m\|_{2,\Omega}^2 + \frac{1}{2}\|u_{m}\|_{2,\Omega}^2.
\end{equation}
Dropping the second term on the left-hand side, integrating in $t$ and applying the Gr\"onwall inequality we estimate $\sup_{(0,T)}\|u_m\|_{2,\Omega}^2$ by a constant depending on $\|u_0\|_{2,\Omega}$, $\|f\|_{2,Q_T}$, and $T$. Using this estimate in \eqref{eq:prelim} and then using \eqref{eq:null-eps-prelim} with $s_1=s_2=0$ we obtain the uniform in $m$ and $\epsilon$ estimate

\begin{equation}
\label{eq:1-est}
\begin{split}
\sup_{(0,T)}\|u_m(t)\|_{2,\Omega}^2 & +\int_{Q_T}\mathcal{F}_{\epsilon, m}^{(p_m-p,q_m-q)}(z,\nabla u_m)|\nabla u_m|^2\,dz\leq C
\end{split}
\end{equation}
with an independent of $m$ and $\epsilon$ constant $C$. By \eqref{eq:1-est}, \eqref{eq:principal-3} and \eqref{eq:structure-prelim-2}, for all sufficiently large $m$

\begin{equation}
\label{eq:interm-1}
\int_{Q_T}|\nabla u_m|^{\overline{s}(z)}\,dz\leq 1+\int_{Q_T}|\nabla u_m|^{\underline{s}_m(z)+r_\ast}\,dz \leq C
\end{equation}
uniformly in $m$ and $\epsilon$.

\begin{lemma}
\label{le:conv-F}
For every $\phi\in \mathbb{W}_{\overline{s}(\cdot)}(Q_T)$

\[
\int_{Q_T}\left(\mathcal{F}_{\epsilon m}^{(p_m-p,q_m-q)}(z,\nabla u_m)-\mathcal{F}_{\epsilon}(z,\nabla u_m)\right)\nabla u_m\cdot \nabla \phi\,dz\to 0\quad \text{as $m\to \infty$}.
\]
\end{lemma}
\begin{proof}
It is sufficient to prove the convergence for the first term of the flux, the convergence of the second term follows by the same arguments. Let us estimate
\[
\begin{split}
& \left|\int_{Q_T}  a_{\epsilon m}(\epsilon^2+|\nabla u_m|^2)^{\frac{p_m-2}{2}}(\nabla u_m\cdot\nabla \phi)\,dz
-\int_{Q_T}a_\epsilon (\epsilon^2+|\nabla u_m|^2)^{\frac{p-2}{2}}(\nabla u_m\cdot\nabla \phi)\,dz\right|
\\
& \qquad \leq  \int_{Q_T}|a_{\epsilon m}-a_\epsilon| (\epsilon^2+|\nabla u_m|^2)^{\frac{p-2}{2}}|\nabla u_m||\nabla \phi|\,dz
\\
&
\qquad \qquad +\int_{Q_T}a_{\epsilon m} \left|(\epsilon^2+|\nabla u_m|^2)^{\frac{p_m-2}{2}}-(\epsilon^2+|\nabla u_m|^2)^{\frac{p-2}{2}}\right||\nabla u_m||\nabla \phi|\,dz.
\end{split}
\]
The terms on the right-hand side are estimated by
\[
\begin{split}
& \max_{\overline{Q}_T}|a_{m}-a|\int_{Q_T}(\epsilon^2+|\nabla u_m|^2)^{\frac{p-1}{2}}|\nabla \phi|\,dz
\\
& \qquad \qquad + \max_{\overline{Q}_T}a_{\epsilon m}\int_{Q_T}\left|(\epsilon^2+|\nabla u_m|^2)^{\frac{p_m-1}{2}}-(\epsilon^2+|\nabla u_m|^2)^{\frac{p-1}{2}}\right||\nabla \phi|\,dz\equiv \mathcal{I}_1+\mathcal{I}_2.
\end{split}
\]

For the sufficiently large $m$
\begin{equation}
\label{eq:p-m-p}
(p-1)\frac{\overline s}{\overline s-1}\leq \overline{s}< p_m+\dfrac{4}{N+2},
\end{equation}
whence, by Theorem \ref{th:integr-par} with $r=0$,
\[
\begin{split}
\int_{Q_T}(\epsilon^2+|\nabla u_m|^2)^{\frac{p-1}{2}}|\nabla \phi|\,dz & \leq \int_{Q_T}|\nabla \phi|^{\overline{s}(z)}\,dz + \int_{Q_T}(\epsilon^2+|\nabla u_m|^2)^{\frac{p-1}{2}\frac{\overline s}{\overline s-1}}\,dz\leq C
\end{split}
\]
with an independent of $\epsilon$ and $m$ constant $C$. Hence, $\mathcal{I}_1\to 0$ as $m\to \infty$.

By the Lagrange mean value theorem

\[
\begin{split}
\mathcal{J} & := (\epsilon^2+|\nabla u_m|^2)^{\frac{p_m-1}{2}} -(\epsilon^2+|\nabla u_m|^2)^{\frac{p-1}{2}} = \int_0^1\dfrac{d}{d\theta}(\epsilon^2+|\nabla u_m|^2)^{\frac{\theta p_m+(1-\theta)p-1}{2}}\,d\theta
\\
& =  \int_0^1(\epsilon^2+|\nabla u_m|^2)^{\frac{\theta p_m+(1-\theta)p-1}{2}}\,d\theta \ln (\epsilon^2+|\nabla u_m|^2) (p_m-p).
\end{split}
\]
Since the sequence $\{p_m\}$ is monotone increasing, then $\theta p_m+(1-\theta)p-1 \leq p-1$. Combining this inequality with \eqref{eq:log} and applying Young's inequality we estimate

\[
|\mathcal{J}|\leq C_\lambda\left(1+(\epsilon^2+|\nabla u_m|^2)^{\frac{p-1+\lambda}{2}}\right)|p_m-p|
\]
with any $\lambda>0$ and an absolute constant $C_\lambda$. It follows that

\[
|\mathcal{I}_2|\leq C_\lambda\max_{\overline{Q}_T}|p_m-p|\left(1+\int_{Q_T}|\nabla \phi|^{\overline{s}}\,dz + \int_{Q_T}|\nabla u_m|^{(p-1+\lambda)\frac{\overline{s}}{\overline{s}-1}}\,dz\right).
\]
The strict inequality \eqref{eq:p-m-p} allows one to choose $m$ so large and $\lambda$ so small that $(p-1+\lambda)\frac{\overline{s}}{\overline{s}-1}<p_m+\frac{4}{N+2}$. Using Theorem \ref{th:integr-par} with $r=0$ we conclude that $\mathcal{I}_2\to 0$ as $m\to \infty$.
\end{proof}

\begin{lemma}
\label{le:fund-sequence}
The sequence $\{u_m\}$ is a Cauchy sequence in $L^\infty(0,T;L^2(\Omega))\cap \mathbb{W}_{\underline{s}(\cdot)}(Q_T)$, and

\[
\int_{Q_\tau}\left( \mathcal{F}_{\epsilon}(z,\nabla u_m)\nabla u_m -\mathcal{F}_{\epsilon}(z,\nabla u_n)\nabla u_n\right)\nabla (u_m-u_n)\,dz\to 0\quad \text{as $m,n\to \infty$}.
\]
\end{lemma}
\begin{proof}
Take $n,m\in \mathbb{N}$, multiply equations \eqref{eq:main-nonsmooth} for $u_m$, $u_n$ by $u_m-u_n$, integrate by parts in $Q_\tau$ with any $\tau\in (0,T]$ and combine the results:

\[
\begin{split}
\dfrac{1}{2} & \|u_m-u_n\|_{2,\Omega}^2(\tau) +\int_{Q_\tau}\left( \mathcal{F}_{\epsilon}(z,\nabla u_m)\nabla u_m -\mathcal{F}_{\epsilon}(z,\nabla u_n)\nabla u_n\right)\nabla (u_m-u_n)\,dz
\\
&
=\frac{1}{2}\|u_{0m}-u_{0n}\|^2_{2,\Omega} + \int_{Q_\tau}(f_m-f_n)(u_m-u_n)\,dz
\\
& + \int_{Q_\tau}\left(\mathcal{F}_{\epsilon m}^{(p_m-p,q_m-q)}(z,\nabla u_m)-\mathcal{F}_{\epsilon}(z,\nabla u_m)\right)\nabla u_m\cdot \nabla (u_m-u_n)\,dz
\\
& +\int_{Q_\tau}\left(\mathcal{F}_{\epsilon n}^{(p_n-p,q_n-q)}(z,\nabla u_n)-\mathcal{F}_{\epsilon}(z,\nabla u_n)\right)\nabla u_n\cdot \nabla (u_m-u_n)\,dz.
\end{split}
\]
By Lemma \ref{le:conv-F} the last two terms on the right-hand side tend to zero as $n,m\to \infty$. The first term tends to zero by the choice of the sequence $\{u_{0m}\}$, the second term tends to zero by the choice of $\{f_m\}$ and the uniform estimate on $\|u_m\|_{2,Q_T}$. The claim follows now from Proposition \ref{pro:racsam-prop-3.2}.
\end{proof}

By Lemma \ref{le:fund-sequence} the sequence $\{u_m\}$ contains a subsequence (for which we will use the same notation) such that for some function $u(z)$

\begin{equation}
\label{eq:converge-2}
\text{$u_m\to u$ $\star$-weakly in $L^\infty(0,T;L^2(\Omega))$ and strongly in $ \mathbb{W}_{\underline{s}(\cdot)}(Q_T)$},
\end{equation}
and $\nabla u_m\to \nabla u$ a.e. in $Q_T$. The sequence $\{\nabla u_m\}$ is uniformly bounded in $L^{\underline{s}(\cdot)+s}(Q_T)$ with any $s<\frac{4}{N+2}$, and $\overline{s}(z)<\underline{s}(z)+\frac{2}{N+2}$. Since $a_{\epsilon m},b_{\epsilon m} \rightrightarrows a_\epsilon, b_\epsilon$,  it follows from the Vitali convergence theorem that

\begin{equation}
\label{eq:converge-3}
\begin{split}
& \text{$a_{\epsilon m}(\epsilon^2+|\nabla u_m|^2)^{\frac{p-2}{2}}\nabla u_m\to a_{\epsilon}(\epsilon^2+|\nabla u|^2)^{\frac{p-2}{2}}\nabla u$ in $L^{p'(\cdot)}(Q_T)$},
\\
& \text{$b_{\epsilon m}(\epsilon^2+|\nabla u_m|^2)^{\frac{q-2}{2}}\nabla u_m\to b_{\epsilon}(\epsilon^2+|\nabla u|^2)^{\frac{q-2}{2}}\nabla u$ in $L^{q'(\cdot)}(Q_T)$}\quad \text{as $m\to \infty$}.
\end{split}
\end{equation}

Multiply equation \eqref{eq:main-nonsmooth} by $u_{m t}$ and integrate over $Q_T$. By the straightforward computation (see, e.g., \cite[Lemma 6.4]{RACSAM-2023})

\begin{equation}
\label{eq:time-der-1}
\begin{split}
\frac{1}{2}\|u_{m t}\|^2_{2,Q_T} & +\int_{\Omega} \left.\left(\dfrac{a_{\epsilon m}(\epsilon^2+|\nabla u_m|^2)^{\frac{p_m}{2}}}{p_m}+ \dfrac{b_{\epsilon m}(\epsilon^2+|\nabla u_m|^2)^{\frac{q_m}{2}}}{q_m}\right)\right|^{t=T}_{t=0}\,dx \leq \frac{1}{2}\int_{Q_T}f^2_m \,dz
 + \sum_{i=1}^4 \mathcal{I}_i,
\end{split}
\end{equation}
where

\[
\begin{split}
& \mathcal{I}_1 = - \int_{Q_T}\frac{a_{\epsilon m}}{p_m^2}p_{m t}(\epsilon^2+|\nabla u_m|^2)^{\frac{p_m}{2}}\left(1-\frac{p_m}{2}\ln (\epsilon^2+|\nabla u_m|^2)\right)\,dz,
\\
& \mathcal{I}_2 = - \int_{Q_T}\frac{b_{\epsilon m}}{q_m^2}q_{m t}(\epsilon^2+|\nabla u_m|^2)^{\frac{q_m}{2}}\left(1-\frac{q_m}{2}\ln (\epsilon^2+|\nabla u_m|^2)\right)\,dz,
\\
&
\mathcal{I}_3 = \int_{Q_T}\dfrac{\partial_t a_{\epsilon m}}{p_m}(\epsilon^2+|\nabla u_m|^2)^{\frac{p_m}{2}}\,dz,
\\
& \mathcal{I}_4 = \int_{Q_T}\dfrac{\partial_t b_{\epsilon m}}{q_m}(\epsilon^2+|\nabla u_m|^2)^{\frac{q_m}{2}}\,dz.
\end{split}
\]
The integrals $\mathcal{I}_1$ and $\mathcal{I}_2$ are bounded by virtue \eqref{eq:interm-1} and \eqref{eq:log}. By the Young inequality

\[
\mathcal{I}_3\leq C\int_{Q_T}|\partial_ta_{\epsilon m}|^{\lambda'}\,dz + \int_{Q_T} (\epsilon^2+|\nabla u_m|^2)^{\lambda\frac{p_m}{2}}\,dz
\]
with any $\lambda>0$. For all sufficiently large $m$, the second integral is uniformly bounded if $\lambda$ satisfies the inequality

\[
\begin{split}
\lambda p(z) & \leq \lambda \overline{s}(z)<\lambda (\underline{s}(z)+r_\ast)<\underline{s}(z)+r^\sharp\quad \Leftrightarrow \quad (\lambda -1)\underline{s}(z) <\frac{4}{N+2}-\lambda\dfrac{2}{N+2}=\frac{2}{N+2}-(\lambda-1)\frac{2}{N+2}
\\
& \Leftrightarrow \qquad \lambda <1+ \frac{2}{2+(N+2)\underline{s}(z)} =\dfrac{4+(N+2)\underline{s}(z)}{2+(N+2)\underline{s}(z)}\quad \Leftrightarrow \quad \lambda'> 2+\frac{N+2}{2}\underline{s}(z).
\end{split}
\]
The last inequality is fulfilled due to assumption  \eqref{eq:mod-coeff}. The integral $\mathcal{I}_4$ is estimated in the same way. It follows that

\begin{equation}
\label{eq:u-t-unif-1}
\|u_{m t}\|_{2,Q_T}^2\leq C,\quad C=C\left(\textbf{data},\|a_t\|_{d,Q_T},\|b_t\|_{d,Q_T},\|\nabla u_0\|_{r,\Omega}\right),
\end{equation}
which allows one to choose a subsequence of $\{u_{m t}\}$ that converges to $u_t$ weakly in $L^2(Q_T)$. By the Aubin-Lions lemma the inclusions $u_t\in L^2(Q_T)$, $u\in \mathbb{W}_{\underline{s}(\cdot)}(Q_T)$ with $\underline{s}^->\frac{2N}{N+2}$ imply $u\in C([0,T];L^2(\Omega))$.

We are now in position to show that the sequence $\{u_{m}\}$ of classical solutions to problems \eqref{eq:main-nonsmooth} with $\mathbf{data}_m$ converges to a strong solution of the regularized problem \eqref{eq:main-reg} with nonsmooth $\mathbf{data}$. For an arbitrary test-function $\phi$ satisfying the conditions of Definition \ref{def:strong-sol}

\[
\begin{split}
\int_{Q_T} & \left(u_{m t}\phi + \mathcal{F}_{\epsilon,m}^{(p_m-p,q_m-q)}(z,\nabla u_m)\nabla u_m\cdot \nabla \phi\right)\,dz=\int_{Q_T}f_m\phi\,dz.
\end{split}
\]
Letting $m\to \infty$ and using \eqref{eq:converge-3}, the weak convergence $u_{m t}$ to $u_t$ in $L^2(Q_T)$, and the properties of $\mathbf{data}_m$, we conclude that $u$ satisfies \eqref{eq:def-nondiv}. The initial condition for $u$ is fulfilled by continuity.

If $u_0$ and $f$ satisfy the conditions of Lemma \ref{le:clas-sol-summary}, then for every $s\in \left(0,\frac{4}{N+2}\right)$

\begin{equation}
\label{eq:summ-est-reg}
\begin{split}
\sup_{(0,T)}\|\nabla u_m\|_{r,\Omega}^{r} & + \int_{Q_T}(\epsilon^2+|\nabla u_m|^2)^{\frac{\underline{s}(z)+r+s}{2}-1}|\nabla u_m|^2
\,dz
\\
&
+\int_{Q_T}\mathcal{F}_{\epsilon,m}^{(p_m-p+r-2,q_m-q+r-2)}(z,\nabla u_m)|(u_m)_{xx}|^2\,dz\leq C
\end{split}
\end{equation}
with a constant $C$ depending on the $\mathbf{data}$, $s$, $r$, but independent of $\epsilon$ and $m$.

Estimates \eqref{eq:a-priori-eps-1}, \eqref{eq:a-priori-eps-2} follow from the uniform estimates \eqref{eq:principal-3}, \eqref{eq:final}, \eqref{eq:1-est}, \eqref{eq:u-t-unif-1}.

\subsection{Proof of Theorem \ref{th:second-order-epsilon}}
Denote $V_m=\nabla \left((a_m+\epsilon) (\epsilon^2+|\nabla u_m|^2)^{\frac{p_m+r-2}{4}}\right)$ and observe that

\[
\begin{split}
\|V_m\|_{2,Q_T}^2 & \equiv \int_{Q_T}\left|\nabla \left((a_m+\epsilon) (\epsilon^2+|\nabla u_m|^2)^{\frac{p_m+r-2}{4}}\right)\right|^2 \,dz \leq  2\int_{Q_T}|\nabla a_m|^2 (\epsilon^2+|\nabla u_m|^2)^{\frac{p_m+r-2}{2}}\,dz
\\
&
+ 2\int_{Q_T}|(a_m+\epsilon)|^2 \left|\nabla \left((\epsilon^2+|\nabla u_m|^2)^{\frac{p_m+r-2}{4}}\right)\right|^2\,dz\equiv \mathcal{I}_1 +\mathcal{I}_2.
\end{split}
\]
By virtue of \eqref{eq:final}, $\mathcal{I}_2$ is bounded uniformly with respect to $m$ and $\epsilon$. By Young's inequality

\[
\mathcal{I}_1\leq C\left(\int_{Q_T}|\nabla a_m|^d\,dz + \int_{Q_T}(\epsilon^2+|\nabla u_m|^2)^{\frac{d}{d-2}\frac{p_m+r-2}{2}}\,dz\right),
\]
where the first integral is bounded by assumption while the second one coincides with the first term of \eqref{eq:2-est} and is estimated in the same way. It follows that the sequence $\{V_m\}$ converges weakly in $L^2(Q_T)^N$ to some $\eta \in L^{2}(Q_T)^N$. To identify $\eta$ we use the a.e. convergence $a_m\to a$ and $\nabla u_m\to \nabla u$.  For every $\phi\in C^\infty_0(Q_T)$ and $i=\overline{1,N}$

\[
\begin{split}
\int_{Q_T}\eta_i \phi\,dz & = \lim_{m\to \infty} \int_{Q_T}D_i\left((a_m+\epsilon) (\epsilon^2+|\nabla u_m|^2)^{\frac{p_m+r-2}{4}}\right)\phi\,dz
\\
&
= -\lim_{m\to \infty} \int_{Q_T} (a_m+\epsilon) (\epsilon^2+|\nabla u_m|^2)^{\frac{p_m+r-2}{4}}D_i\phi \,dz
\\
&
= -\int_{Q_T}(a+\epsilon)(\epsilon^2+|\nabla u|^2)^{\frac{p+r-2}{4}}D_i\phi \,dz,
\end{split}
\]
which means that $\eta=\nabla \left((a+\epsilon)(\epsilon^2+|\nabla u|^2)^{\frac{p+r-2}{4}}\right)$.

\section{The degenerate problem. Proof of Theorems \ref{th:main-1}, \ref{th:main-2}}
\label{sec:degenerate-problem}

Let $\{u_\epsilon\}$ be the family of strong solutions of the regularized problem \eqref{eq:main-reg}. The proof of the existence theorem reduces to justifying the limit passage as $\epsilon\to 0^+$ in the integral identity \eqref{eq:def-nondiv} for $u_\epsilon$. We omit the details of the proof because it is a literal repetition of the proof in \cite[Sec.~8]{RACSAM-2023}. The proof of the second-order regularity follows the proof of Theorem 2.3 in \cite[Subsec.~8.4]{RACSAM-2023}. The estimate on $\|\nabla u\|_{L^\infty(0,T;L^r(\Omega))}$ is an immediate byproduct of the uniform estimate \eqref{eq:a-priori-eps-2}.

\appendix

\section{Auxiliaries}
\label{sec:aux}
Let $\mathcal{F}_\epsilon^{(s_1,s_2)}(x,\xi)$ be the function defined in \eqref{eq:gamma}.
The following easily verified properties hold.

\begin{enumerate}
\item  For the nonnegative coefficients $a, b \in L^\infty(\Omega)$, $\epsilon \in (0,1)$, $\xi\in \mathbb{R}^N$,  and the parameters $s_1,s_2 \geq 0$
\begin{equation}
\label{eq:null-eps-prelim}
\begin{split}
\mathcal{F}^{(s_1, s_2)}_{\epsilon}(x,\xi)\beta_{\epsilon}(\xi) & \leq \begin{cases}
a_\epsilon(2\epsilon^2)^{p+s_1}+b_\epsilon (2\epsilon^2)^{q+s_2} & \text{if $\vert \xi\vert \leq \epsilon$},
\\
2 \mathcal{F}^{(s_1, s_2)}_\epsilon(x,\xi)\vert \xi\vert ^2 & \text{if $\vert \xi\vert >\epsilon$}
\end{cases}
\leq C+ 2\mathcal{F}_\epsilon^{(s_1, s_2)}(x,\xi)\vert \xi\vert ^2.
\end{split}
\end{equation}
In particular,
\begin{equation}
\label{eq:null-eps}
\begin{split}
\mathcal{F}^{(s_1, s_2)}_0(x,\xi)\vert \xi\vert ^2\,dx & \equiv a_\epsilon\vert \xi\vert ^{p+s_1}+b_\epsilon \vert \xi\vert ^{q+s_2}
\leq \mathcal{F}^{(s_1, s_2)}_{\epsilon}(x,\xi)\beta_{\epsilon}(\xi)
\leq  C+2\mathcal{F}^{(s_1, s_2)}_\epsilon(x,\xi)\vert \xi\vert ^2.
\end{split}
\end{equation}
\item For every $\lambda>0$, $\xi\in \mathbb{R}^N$, and $\mu\in (0,\lambda)$
\begin{equation}
\label{eq:log}
\begin{split}
\vert \xi\vert ^{\lambda}\vert \ln \vert \xi\vert \vert  & =\begin{cases}
\left(\vert \xi\vert ^{\lambda+\mu}\right)\left(\vert \xi\vert ^{-\mu}\vert \ln \vert \xi\vert \vert \right) & \text{if $\vert \xi\vert \geq 1$},
\\
\left(\vert \xi\vert ^{\lambda-\mu}\right)\left(\vert \xi\vert ^{\mu}\vert \ln \vert \xi\vert \vert \right)  & \text{if $\vert \xi\vert < 1$},
\end{cases}
\leq C(\mu,\lambda)\left(1+\vert \xi\vert ^{\lambda+\mu}\right).
\end{split}
\end{equation}
\end{enumerate}

Let $a,b\in C^{0}(\overline{Q}_T)$. Introduce the quantities

\[
\begin{split}
& \mathcal{N}(\nabla w)=\int_{Q_T}\left(a(z)|\nabla w|^{p(z)}+b(z)|\nabla w|^{q(z)}\right)\,dz,
\\
& \mathcal{G}_\epsilon(\nabla u,\nabla v)=\int_{Q_T}\left(\mathcal{F}_{\epsilon}(z,\nabla u)\nabla u- \mathcal{F}_{\epsilon}(z,\nabla v)\nabla v\right)\cdot \nabla (u-v)\,dz.
\end{split}
\]

\begin{proposition}[\cite{RACSAM-2023}, Proposition 3.2]
\label{pro:racsam-prop-3.2}
For every $u,v\in \mathbb{W}_{\overline{s}(\cdot)}(Q_T)$ and any $\epsilon\in (0,1)$

\[
\mathcal{N}(\nabla (u-v))\leq C\left(\mathcal{G}_\epsilon^{\frac{\overline{s}^+}{2}}(\nabla u,\nabla v)+ \mathcal{G}_\epsilon^{\frac{\underline{s}^-}{2}}(\nabla u,\nabla v)+\mathcal{G}_\epsilon(\nabla u,\nabla v)\right)
\]
with a constant C depending on $p^\pm$, $q^\pm$, $a^+$, $b^+$, $\|\nabla u\|_{\overline{s}(\cdot),Q_T }$, $\|\nabla v\|_{\overline{s}(\cdot),Q_T }$.
\end{proposition}

\begin{corollary}[\cite{RACSAM-2023}, Lemma 3.1]
\label{cor:cont-embed}
There is the continuous embedding

\[
\{u:Q_T\mapsto \mathbb{R}:\,u\in L^{2}(Q_T),\,\mathcal{N}(\nabla u)<\infty\}\subset \mathbb{W}_{\underline{s}(\cdot)}(Q_T).
\]
\end{corollary}

\begin{proposition}[\cite{RACSAM-2023}, Lemma 3.2]
\label{pro:conv-W}
If $u,v\in \mathbb{W}_{\overline{s}(\cdot)}(Q_T)$ and $\epsilon\in (0,1)$, then

\[
\int_{Q_T}|\nabla (u-v)|^{\underline{s}(z)}\,dz\to 0\quad \text{when $\mathcal{G}_\epsilon(\nabla u,\nabla v)\to 0$}.
\]
\end{proposition}

\begin{corollary}
\label{cor:conv-1}
If a sequence $\{u_m\}$ is uniformly bounded in $\mathbb{W}_{\overline{s}(\cdot)}(Q_T)$ and $\mathcal{G}_\epsilon(\nabla u_m,\nabla u_k)\to 0$ as $m,k\to \infty$, then

\[
\text{$|\nabla (u_m-u_k)|\to 0$ a.e. in $Q_T$}.
\]
\end{corollary}

\bibliographystyle{elsarticle-num}
\bibliography{bib-db-14-2}

\end{document}